\documentclass[11pt]{amsart}

\usepackage[utf8x]{inputenc}

\usepackage{amsmath, amsthm, amsfonts, amssymb}
\usepackage{epsfig, enumerate}
\usepackage{color}

\usepackage{hyperref}
\usepackage[all]{xy}

\usepackage{float}

\usepackage{tikz}
\usetikzlibrary{arrows,shapes,trees}
\usepackage{graphicx}
\usepackage{amssymb}
\usepackage{amsfonts}
\usepackage{amsthm}
\usepackage{amsmath}
\usepackage{mathrsfs}
\usepackage{makecell}

\newtheorem{theorem}{Theorem}[section]
\newtheorem{lemma}[theorem]{Lemma}
\newtheorem{coro}[theorem]{Corollary}
\newtheorem{prop}[theorem]{Proposition}
\newtheorem{proposition}[theorem]{Proposition}

\newtheorem{Claim}{Claim}

\newtheorem*{assumptions*}{Assumptions}

\newtheorem*{rem*}{Remark}
\newtheorem{rem}[theorem]{Remark} 

\theoremstyle{remark}

\newtheorem*{remark*}{Remark}

\theoremstyle{definition}
\newtheorem{definition}{Definition}

\newcommand{\B}{{\mathbf B}}
\newcommand{\C}{{\mathbf C}}

\newcommand{\E}{{\mathbf E}}

\newcommand{\I}{{\mathbf I}}

\newcommand{\Q}{{\mathbf Q}}
\newcommand{\R}{{\mathbf R}}

\newcommand{\T}{{\mathbf T}}

\newcommand{\V}{{\mathbf V}}
\newcommand{\W}{{\mathbf W}}

\newcommand{\Z}{{\mathbf Z}}

\newcommand{\CC}{{\mathbb C}}
\newcommand{\DD}{{\mathbb D}}

\newcommand{\MM}{{\mathbb M}}

\newcommand{\RR}{{\mathbb R}}

\newcommand{\TT}{{\mathbb T}}

\newcommand{\ZZ}{{\mathbb Z}}

\newcommand{\SL}{{\rm SL}}
\newcommand{\GL}{{\rm GL}}
\newcommand{\Sp}{{\rm Sp}}

\newcommand{\HSp}{{\rm HSp}}

\newcommand{\tr}{{\rm Tr}}

\newcommand{\diag}{{\rm diag}}

\DeclareMathOperator{\Tr}{{\rm Tr}}

\newcommand{\be}[1]{\begin{equation} \label{#1} }
\newcommand{\ee}{\end{equation}}
\newcommand{\beq}{\begin{equation}}

\def \W{{\mathcal W}}

\def \C{{\mathcal{C}}}
\def \cC{{\mathcal{C}}}
\def \W{\mathcal{W}}

\def \hx0{\hat{x_0}}

\def \id{\mathrm{id}}

\def \E{\mathcal{E}}

\def \B{\mathcal{B}}
\def \V{\mathcal{V}}
\def \Z{\mathcal{Z}}

\def\Z{{\mathbb Z}}
\def\C{{\mathbb C}}
\def\R{{\mathbb R}}
\def\Q{{\mathbb Q}}
\def\T{{\mathbb T}}

\begin{document}
\title[Subordinacy Theory for Long-Range Operators]{Subordinacy Theory for Long-Range Operators: Hyperbolic Geodesic Flow Insights and Monotonicity Theory}

\author{Zhenfu Wang}
\address{Chern Institute of Mathematics and LPMC, Nankai University, Tianjin 300071, China}
\email{zhenfuwang@mail.nankai.edu.cn}

\author{Disheng Xu}
\address{School of Science, Great Bay University and Great bay institute for advanced study, 
Songshan Lake International Innovation Entrepreneurship Community A5, Dongguan 523000, China}
\email{xudisheng@gbu.edu.cn}

\author{Qi Zhou}
\address{
Chern Institute of Mathematics and LPMC, Nankai University, Tianjin 300071, China
}
\email{qizhou@nankai.edu.cn}

\begin{abstract}
We introduce a comprehensive framework for subordinacy theory applicable to long-range operators on $\ell^2(\mathbb{Z})$, bridging dynamical systems and spectral analysis. For finite-range operators, we establish a correspondence between the dynamical behavior of partially hyperbolic (Hermitian-)symplectic cocycles and the existence of purely absolutely continuous spectrum, resolving an open problem posed by Jitomirskaya. For infinite-range operators—where traditional cocycle methods become inapplicable—we characterize absolutely continuous spectrum through the growth of generalized eigenfunctions, extending techniques from higher-dimensional lattice models.


Our main results include the \textit{first} rigorous proof of purely absolutely continuous spectrum for quasi-periodic long-range operators with analytic potentials and Diophantine frequencies—in particular, the \textit{first} proof of the \textbf{all-phases} persistence for finite-range perturbations of subcritical almost Mathieu operators—among other advances in spectral theory of long-range operators.

The key novelty of our approach lies in the unanticipated connection between stable/vertical bundle intersections in geodesic flows—where they detect conjugate points—and their equally fundamental role in governing (de-)localization for Schrödinger operators. The geometric insight, combined with a novel coordinate-free monotonicity theory for general bundles (including its preservation under center-bundle restrictions) and adapted analytic spectral and KAM techniques, enables our spectral analysis of long-range operators.



\end{abstract}

\maketitle

\section{Introduction}
In this work, we study the spectral properties of self-adjoint long-range operators on $\ell^2(\mathbb{Z})$ defined by
\begin{equation}\label{long-range-op}
    (L_{v,w}u)_n = \sum_{k=-\infty}^{\infty} w_k u_{n+k} + v_nu_n, \quad n \in \mathbb{Z},
\end{equation}
where $w = \{w_k\}_{k \in \mathbb{Z}}$ is a sequence of hopping amplitudes satisfying $\overline{w_{-k}} = w_k$, and $v = \{v_n\}_{n \in \mathbb{Z}}$ is the on-site potential, with $v_n\in \R \text{ and }\sup_n |v_n| < \infty$.
When $w_k = 0$ for $|k| \neq 1$, \eqref{long-range-op} reduces to a Jacobi operator. If additionally $w_{\pm 1} = 1$, it becomes the classical Schr\"odinger operator:
\begin{equation}\label{schrodinger}
    (H_{v}u)_n = u_{n+1} + u_{n-1} + v_nu_n,
\end{equation}
which has served as a foundational model for electronic structure calculations in solid-state physics \cite{Anderson,AA}. However, the short-range nature of \eqref{schrodinger} neglects critical long-range interactions present in real materials. The inclusion of non-zero $w_k$ in \eqref{long-range-op} provides a more physically realistic framework \cite{bid,Rod,Sar}.

A central object of study is the quasi-periodic Schr\"odinger operator
\begin{equation}\label{quasi-periodic}
    (H_{ v,\alpha,\theta}u)_n = u_{n+1} + u_{n-1} +  v(\theta + n\alpha)u_n, \quad n \in \mathbb{Z},
\end{equation}
where $v \in C^0(\mathbb{T}^d, \mathbb{R})$ is the potential, $\theta \in \mathbb{T}^d$ is the phase, and $\alpha \in \T^d$ is a rationally independent frequency vector. These operators have profound connections to condensed matter physics and dynamical systems. For comprehensive reviews, see \cite{DF,DF2,J,Y}.  Notably, \eqref{long-range-op} arises as the Aubry dual of \eqref{quasi-periodic}, i.e. $w_k$ in \eqref{long-range-op} is defined by the Fourier coefficients of $v(\cdot)$ in \eqref{quasi-periodic} and $v_n$ in \eqref{long-range-op} is defined by $2\cos(2\pi(\theta + n\alpha))$. Aubry duality has been instrumental in analyzing localization-delocalization transitions (\cite{AJ2,AYZ,BJ,GJ,JK,Puig1}), Cantor spectrum problem (\cite{AJ3,GJY,Puig1,Puig11}) and remains a pivotal tool in the spectral theory of quasi-periodic operators.

\subsection{Spectral Types and Subordinacy Theory}
All of these factors motivate the investigation of the spectral properties of long-range operators defined by \eqref{long-range-op}. In this paper, we concentrate on subordinacy theory and its applications.

A central problem in the spectral theory of Schr\"odinger operators is the classification of spectral types: pure point, absolutely continuous, and singular continuous spectra. This classification underpins our understanding of quantum dynamical behavior \cite{DF}. Pure point spectrum corresponds to localized eigenstates, characteristic of strongly disordered systems. Absolutely continuous spectrum reflects delocalized states with ballistic transport, as seen in periodic structures. Singular continuous spectrum, though rare, emerges in critical systems with anomalous diffusive properties.

A fundamental question in spectral theory of Schr\"odinger operators is whether spectral types can be characterized by the behavior of generalized eigenfunctions. This issue is elegantly addressed in one dimension through the \emph{Gilbert-Pearson subordinacy theory} \cite{GP}, which establishes a relationship between spectral measures and the growth of solutions: the absolutely continuous spectrum is characterized by the existence of bounded solutions at the corresponding energies, a version that is most frequently employed \cite{Simon}. Subsequently, Jitomirskaya and Last \cite{JL1,JL2} further connected subordinacy to the analytic properties of the Weyl $m$-function, i.e. they proved that 
$$ \frac{5-\sqrt{24}}{|m(E+i \epsilon)|}  \leq \frac{\|u_2\|_{L(\epsilon)}}{\|u_1\|_{L(\epsilon)}} \leq \frac{5+\sqrt{24}}{|m(E+i \epsilon)|},$$
where $u_1,u_2$ are two linearly independent solutions.
Subordinacy has numerous important applications, including the characterization of purely absolutely continuous spectra \cite{Avila: ac,WXYZZ}, Hausdorff-dimensional spectral analysis for discrete Schrödinger operators \cite{DKL,Z}, and quantum dynamical bounds via solution growth rates \cite{DT1,KKL}.

The framework has been extended to diverse settings: Jacobi operators \cite{KL}; whole-line continuum Schr\"odinger operators \cite{G}; CMV matrices (one- and two-sided) \cite{GDO,SO}; Jacobi matrices on specific graphs \cite{Le}.
Despite these advances, a \emph{long-range operator subordinacy theory} remained open.

\subsubsection{Subordinacy theory for finite-range operator}
In 2015, Jitomirskaya posed a foundational question to the authors developing \textit{the full version of} subordinacy theory for finite-range operators.
There have been some progresses in recent years \cite{MS,OC}; however, they are not satisfactory. 

To explain this, let's consider the Schr\"odinger operators on strips. Let $(\Omega,T)$ be a topological dynamical system, i.e. $\Omega$ is a compact metric space and $T:\Omega\to \Omega$ is a homeomorphism. Assume $V(\cdot) \in C^0(\Omega,\mathrm{Her}(m,\mathbb{C}))$, where $\mathrm{Her}(m,\mathbb{C})$ denotes the space of $m \times m$ Hermitian matrices. We study the 
Schr\"odinger operator $H_{V,T,\omega}$  with a dynamically-defined potential induced by $V$  acting on $\ell^2(\mathbb{Z},\mathbb{C}^m)$:
\begin{equation}\label{strip-operator}
(H_{V,T,\omega}\vec{u})_n = C\vec{u}_{n+1} + V(T^n(\omega))\vec{u}_n + C^*\vec{u}_{n-1}, \quad n \in \mathbb{Z},
\end{equation}
where $\omega \in \Omega$ is the phase and $C \in \mathrm{GL}(m,\mathbb{C})$. These operators arise naturally in:
analytic theory of matrix orthogonal polynomials \cite{DPS}, XY spin chain models \cite{HSS}, and 
Dirac-Harper models \cite{BGW}. In particular, one can rewrite the dynamical defined finite-range operators \begin{equation*}
    (L_{v,w}u)_n = \sum_{k=-m}^{m} w_k u_{n+k} + v_nu_n, \quad n \in \mathbb{Z},
\end{equation*}
as  Schr\"odinger operators on strips \cite{Puig}.

The Schr\"odinger operator on the strip \eqref{strip-operator} induces a \emph{Schr\"odinger cocycle} 
\begin{equation}\label{eqn: Sc ccy}
A_E(\omega) = \begin{pmatrix}
    C^{-1}(EI - V(\omega)) & -C^{-1}C^* \\
    I & 0
\end{pmatrix},
\end{equation}
defined via the skew-product:
\[
(T,A_E): \begin{cases}
    \Omega \times \mathbb{C}^{2m} \to \Omega \times \mathbb{C}^{2m} \\
    (\omega,v) \mapsto (T(\omega), A_E(\omega)v)
\end{cases}.
\] For $n \in \mathbb{Z}$, define the cocycle iterates (or the transfer matrix):
$A_{0}(\omega)=I$,
\begin{equation*}
	(A_E)_{n}(\omega)=\prod_{j=n-1}^{0}A_E(T^{j}(\omega)),\  \text{ and } A_{-n}(\omega)=A_{n}(T^{-n}(\omega))^{-1} \text{ for }\ n\ge1.
\end{equation*}
We denote the Lyapunov exponents\footnote{A precise definition can be found in Section \ref{com}.} of $(T, A_E)$ by the following sequence: $$L_1(E) \geq L_2(E) \geq \ldots \geq L_m(E) \geq 0 \geq -L_m(E) \geq \ldots  \geq -L_2(E) \geq -L_1(E),$$ where the exponents are repeated according to their multiplicities. These exponents occur in pairs, as $(T, A_E)$ is Hermitian-symplectic \cite{Puig}\footnote{More precisely, it is Hermitian-symplectic with respect to a Hermitian-symplectic structure defined by $C$, see Section \ref{subsec: H-S}.}.

As a generalization of Kotani's theory \cite{Ko} concerning the Schr\"odinger operator, Kotani and Simon \cite{KS} demonstrated that the set
\begin{equation}\label{kotani}
\{ E \mid \text{exactly } 2j \text{ Lyapunov exponents are } 0 \}
\end{equation}
represents the essential support of the absolutely continuous spectrum with multiplicity $2j$. While \cite{MS,OC} established that 
\begin{equation}\label{grow}
  \{ E \mid \liminf_{L\rightarrow \infty} \frac{1}{L} \sum_{n=1}^L\|(A_E)_{n}\|<\infty\}  
\end{equation}
is contained in the essential support of the absolutely continuous spectrum, thereby excluding hyperbolicity (which is associated with positive Lyapunov exponents) entirely. 
Meanwhile, recall a cocycle $(T,A_E)$ is called \emph{bounded} if
$
\sup_{n \in \mathbb{Z}} \|(A_E)_n(\omega)\| < \infty.
$
In the scalar case ($m=1$), subordinacy theory \cite{Simon} essentially asserts that the operator $H_{V,T,\omega}$ is purely absolutely continuous on the set
$
\{E \mid (T,A_E) \text{ is bounded}\},
$
or on the set \eqref{grow} \cite{LS}. 

These findings imply that Kotani's theory \cite{KS} has already provided insights into the essential support. Jitomirskaya's inquiry, however, seeks to characterize the set where $H_{V,T,\omega}$ is \textit{purely} absolutely continuous. The primary challenge lies in managing the hyperbolic components of the associated cocycles, a difficulty stemming from the interplay between multiple length scales and non-commuting matrix products, as we will explain more clearly later. 

The resolution to this challenge emerged from insights derived from smooth dynamical systems. In establishing quantitative versions of Avila's global theory for one-frequency cocycles, the work in \cite{GJYZ} reveals that the \emph{partial hyperbolicity} of the quasi-periodic cocycle plays a crucial role. Indeed, partial hyperbolicity naturally arises as the dual model of a quasi-periodic Schr\"odinger operator \cite{GJYZ} and in a near-constant Schr\"odinger operator on a strip \cite{WXYZ}. This framework has been successfully extended to address other spectral problems \cite{GJ,GJY}.

Recall $(T,A_E)$ is \emph{partially hyperbolic} if there exists an $A_E$-invariant dominated splitting
$$
\mathbb{C}^{2m} = E^u_{A_E} \oplus E^c_{A_E} \oplus E^s_{A_E},
$$
such that $A_E$ uniformly expands $E^u_{A_E}$ and uniformly contracts $E^s_{A_E}$ (see Section~\ref{secds} for definitions). Note that we allow $E^\ast$ for $\ast \in \{c,u,s\}$ to be trivial.

In higher dimensions ($m \geq 2$), the Weyl $m$-function generalizes to a $m \times m$ matrix-valued function \cite{KS}. While one might attempt to extend the Jitomirskaya-Last inequality by considering two linearly independent fundamental solution matrices $U_1, U_2: \mathbb{Z} \to \mathbb{C}^{m\times m}$ and controlling $|m(E+i\epsilon)|$ through the ratio $\|U_1\|_{L(\epsilon)}/\|U_2\|_{L(\epsilon)}$, this approach generally fails due to intrinsic matrix obstructions. The non-commutativity and non-conformality of matrices result in the norms $\|U_i\|$ detecting only the fastest-growing solutions while neglecting the slower-growing solutions that actually govern the behavior of the $m$-function—specifically, those solutions that determine the absolute continuity of the spectral measure. This obstruction becomes particularly evident when $E^u_{A_E} \oplus E^s_{A_E} \neq \{0\}$, complicating the development of a practical higher-dimensional generalization of the Jitomirskaya-Last inequality,  this is the primary reason why \cite{MS,OC} do not succeed in establishing the full version of Jitomirskaya's open problem.
Our key insight is that partial hyperbolicity isolates the central subspace $E^c_{A_E}$, whose boundedness governs the spectral type. This insight resolves Jitomirskaya's problem concerning finite-range operators:

\begin{theorem}\label{thm:main-spectral} Assume $(T,A_E)$ is partially hyperbolic. Then for all $\omega \in \Omega$, $H_{V,T,\omega}$ is purely absolutely continuous on
\begin{equation}\label{defb}
\mathcal{B}_\omega:= \Big\{ E :  \sup_{n \in \mathbb{Z}} \|(A_E)_n(\omega)|_{E^c_{A_E}}\| < \infty \Big\}.
\end{equation}
\end{theorem}

\begin{rem}
\begin{enumerate}
 \item Subordinacy theory (as in Theorem \ref{thm:main-spectral}) is a powerful tool, providing foundational techniques to establish purely absolutely continuous spectrum for all phases (for example, Theorem \ref{type1ac} and Corollary \ref{pamo}).
This ``all phases" result remains unattainable by other methods. For example, Kotani's gem \cite{kotani2}—a profound result \cite{Da} relating the density of the absolutely continuous part of the density of states measure to the spectral measure—can only ensure such spectrum for almost every phase \cite{AD}, not universally.
\item     From this aspect, we should mention two recent results \cite{GJY, GX} that established Kotani's theory for one-frequency Schr\"odinger operators on the strip, assuming the partial hyperbolicity of the cocycle. It is interesting to consider whether the purity of the spectral measure following scheme of  Kotani's gem. 
    However, applying partial reflectionlessness \cite{GJY, GX}  directly presents challenges due to several factors: the non-commutativity of matrices, the definition of partial hyperbolicity being limited to a local neighborhood, and the complexity of the Green’s function at the center, which complicates control.
\end{enumerate}

\end{rem}

\subsubsection{Subordinacy theory for infinite-range operator}
While cocycles play a crucial role in analyzing strip Schr\"odinger operators on $\ell^2(\mathbb{Z},\mathbb{C}^m)$ with $m < \infty$, their direct extension to $m = \infty$ faces a fundamental obstruction: transfer matrices become undefined due to the absence of infinite-dimensional cocycle theory. Notably, infinite-range operators share deeper structural similarities with Schr\"odinger operators on $\ell^2(\mathbb{Z}^d)$:
\begin{equation*}\label{schrodinger-zd}
    (H^d u)_{\mathbf{n}} = \sum_{\|\mathbf{m}-\mathbf{n}\|_{1}=1} u_\mathbf{m} + v_\mathbf{n}u_\mathbf{n},
\end{equation*}
where $\|\cdot\|_1$ denotes the $\ell^1$-norm. This connection suggests a potential strategy to study \(H^d\) via approximations by Schr\"odinger operators on the strip \cite{Dam}.

For $\mathbb{Z}^d$ operators, a pivotal insight characterizes the absolutely continuous spectrum through the existence of weakly bounded solutions — a criterion independent of cocycle. Indeed, Kislev and Last \cite{KL1} proved that $H^d$ is 
absolutely continuous on the set 
\begin{equation}\label{weak-bdd-sol}
    \mathcal{WB} = \left\{ E : \liminf_{R \to \infty} R^{-1} \sum_{\|\mathbf{n}\|_\infty < R} |u(\mathbf{n},E)|^2 < \infty \right\},
\end{equation}
where $u(\mathbf{n},E)$ solves $(H^d - E)u(\mathbf{n},E) = 0$. While \eqref{weak-bdd-sol} excludes bounded solutions in $\Z^d$ ($d\geq 2$), we extend this framework to infinite-range operators.

For any $\phi \in \ell^2(\mathbb{Z})$, let $\mu_\phi$ be its spectral measure. Define the \emph{upper $\alpha$-derivative}:
\[
D_{\mu_\phi}^{+,\alpha}(E) := \limsup_{\varepsilon \to 0} \frac{\mu_\phi\big(E-\varepsilon,E+\varepsilon\big)}{(2\varepsilon)^\alpha},
\]
which quantifies the continuity/singularity of $\mu_\phi$. Our main result is:
\begin{theorem}\label{sub}
Let $H$ be an infinite-range operator defined in \eqref{long-range-op}, satisfying $|w_k| < \frac{C}{k^3}$ for $k \in \mathbb{Z}$. 
If there exists a nontrivial solution $u = \{u_n\}$ satisfying
\[
\liminf_{R \to \infty} R^{-\alpha} \sum_{k=-R}^R |u_n|^2 < \infty,
\]
then for any compactly supported $\phi$ with $\sum_n \overline{\phi_n} u_n \neq 0$, we have $
D_{\mu_\phi}^{+,\alpha}(E) > 0.$
\end{theorem}

We do not pursue the optimal decay rate of $w_k$ in this context. Theorem \ref{sub} is interesting because it offers insights into the pointwise behavior of spectral measures based on straightforward and natural assumptions regarding the behavior of generalized eigenfunctions. Furthermore, it yields immediate corollaries about subordinacy by the standard argument (for details, see \cite{KL1}):

\begin{coro}\label{cor:subordinacy}
Under the assumptions of Theorem \ref{sub}, if there exists a positive measure set $S \subset \mathbb{R}$ such that every $E \in S$ admits a nontrivial bounded solution, then $S$ is contained in the essential support of the absolutely continuous spectrum.
\end{coro}

\begin{rem}
One should compare Corollary \ref{cor:subordinacy} with Theorem \ref{thm:main-spectral}, here we only assume one nontrivial bounded solution (now the solution space is infinite dimensional), thus, in general, one can't anticipate $H$ is purely absolutely continuous on $S$.
\end{rem}

\subsection{Applications: absolutely continuous spectrum for long-range operator}\label{app:ac}

Building on Theorems~\ref{thm:main-spectral}, \ref{sub},  one can also establish fundamental relations between spectral properties, generalized eigenfunctions, and quantum dynamics—particularly in bounding transport properties of quantum systems. For detailed proofs and related results, see \cite{DT1,KKL,KL1} and references therein.  

\subsubsection{Absolutely continuous spectrum for Schr\"odinger operator}
Our focus here is the absolutely continuous spectrum of quasi-periodic long-range operators:
\begin{equation}\label{long-range-qp}
    (L_{\varepsilon v,w,\alpha,\theta}u)_n = \sum_{k=-\infty}^\infty w_k u_{n+k} + \varepsilon v(\theta + n\alpha)u_n, \quad n \in \mathbb{Z},
\end{equation}
where $w_k$ are Fourier coefficients of $w(\cdot) \in C^\omega(\mathbb{T}^d, \mathbb{R})$, $v(\cdot) \in C^\omega(\mathbb{T}^d, \mathbb{R})$ is the potential, $\theta\in \T^d$ is the phase, and $\alpha \in \T^d$ is a rationally independent frequency vector.

We begin by reviewing key milestones in the study of absolutely continuous spectrum for quasi-periodic Schr\"odinger operators. Recall that  $\alpha\in\T^d$ satisfies the Diophantine condition $\mathrm{DC}(\gamma, \tau)$: 
    \begin{equation*}\label{diophantine}
        \inf_{j \in \mathbb{Z}} |\langle k, \alpha \rangle + j| \geq \frac{\gamma}{|k|^\tau} \quad \forall k \in \mathbb{Z}^d \setminus \{0\}
    \end{equation*}
    for some $\gamma > 0$, $\tau > 0$, and denote $DC=\cup_{\gamma,\tau}DC(\gamma,\tau)$. For Diophantine frequencies $\alpha$  and sufficiently small $\varepsilon > 0$ (depending on $\alpha$ and $V$), Dinaburg and Sinai \cite{DS} established the existence of absolutely continuous spectrum using KAM techniques. 
Eliasson \cite{Eli92} refined the KAM scheme, proving that under the same Diophantine conditions, $H_{\varepsilon v,\alpha,\theta}$ exhibits \emph{purely} absolutely continuous spectrum for all $\theta$. This marked the first complete characterization of spectral type in the small coupling regime.

Another advancement in establishing purely absolutely continuous spectrum is achieved through localization techniques. Utilizing non-perturbative localization and Aubry duality \cite{DF2,GJLS}, for Diophantine frequencies, Jitomirakaya  \cite{J1999} proved the almost Mathieu operator $H_{2\lambda \cos,\alpha,\theta}$ has purely absolutely continuous spectrum for $|\lambda|<1$ and a.e. $\theta$. Subsequently, Bourgain and Jitomirskaya \cite{BJ} extended this result to general analytic quasi-periodic potentials.

 In the one-frequency case, the breakthrough can be traced back to Avila \cite{avila}, who introduced a tripartite classification (subcritical/critical/supercritical) linking spectral types to Lyapunov exponents. The Almost Reducibility Conjecture (ARC) connects subcriticality to almost reducibility \cite{avila2010almost,avila-kam}, implying pure absolutely continuous spectrum \cite{arcl,Avila: ac}.

 \subsubsection{Absolutely continuous spectrum for finite-range operators}

When there exists a value of $ |k| > 1 $ such that $ w_k \neq 0 $, research on the absolutely continuous spectrum of the long-range operator \eqref{long-range-qp} is \text{rare}, and there is \text{no} result regarding the pure absolutely continuous spectrum.
 Instead of concentrating on the  absolutely continuous spectrum, Wang et al. \cite{WXYZ} studied the absolute continuity of the integrated density of states (IDS), which can be interpreted as the average spectral measures of an ergodic family of self-adjoint operators $ \{ L_{\varepsilon v, w, \alpha, \theta} \}_{\theta \in \mathbb{T}^d} $ over $ \theta $:
\begin{equation}\label{sids}
    \mathcal{N}(E) = \int_{\theta \in \mathbb{T}^d} \mu_\theta(-\infty, E] d\theta,
\end{equation}
where $ \mu_\theta $ denotes the associated spectral measure of $ L_{\varepsilon v, w, \alpha, \theta} $. They demonstrated that in the perturbative regime, $ \mathcal{N}(\cdot) $ is absolutely continuous when $ \alpha $ is Diophantine and $ w(\cdot) $ is a trigonometric polynomial \cite{WXYZ}. If $ \mu_\theta $ is absolutely continuous for almost every $ \theta $, it follows that $ \mathcal{N}(\cdot) $ is absolutely continuous. However, the converse does not hold: the average of the singular spectral measure may also lead to absolute continuity \cite{AD}. This raises a natural question: Is $ \mu_\theta $ absolutely continuous for almost every $ \theta $? In this paper, we address this question:

\begin{theorem}\label{ac}
    Let $\alpha \in \mathrm{DC}$, $w(\cdot)$ be a trigonometric polynomial, and $v(\cdot) \in C^\omega(\mathbb{T}^d, \mathbb{R})$. There exists $\varepsilon_0 = \varepsilon_0(\alpha, v, w) > 0$ such that for $|\varepsilon| < \varepsilon_0$, the operator $L_{\varepsilon v, w, \alpha, \theta}$ is purely absolutely continuous  for almost every $\theta$.  
    Furthermore, it has no point spectrum for all $\theta$.
\end{theorem}

\begin{rem}
Theorem~\ref{ac} aligns with longstanding expectations in the  community, though we present here its \textbf{first rigorous proof}. 
\end{rem}

 \subsubsection{All phases pure absolutely continuous spectrum}

 A more critical question arises: under the conditions of Theorem~\ref{ac}, does the operator \( L_{\varepsilon v, w, \alpha, \theta} \) exhibit purely absolutely continuous spectrum for  \textbf{all} \(\theta\)? To address this, we first consider cocycles with a two-dimensional center. This class of operators naturally arises as the dual model of the {\it type I operator}, a model that has received significant attention recently \cite{GJ,GJY,GJYZ,hs,HS}.
 
For one-frequency quasi-periodic Schr\"odinger operators \eqref{quasi-periodic}, the Lyapunov exponent of the complexified Schr\"odinger cocycle \( L_y(E) = L(\alpha, A_E(\cdot + iy)) \) is an even, convex, piecewise affine function with integer slopes \cite{avila}. This motivates the following definitions:
 
\begin{definition}\label{acc}\cite{avila,GJY}
Let \( y \leq \infty \) denote the natural boundary of analyticity for \( v(\cdot) \in C^\omega(\T,\R) \). The acceleration is defined by
\[
\omega(E) = \lim_{y \to 0^+} \frac{L_y(E) - L_0(E)}{2\pi y}.
\]
The {\it T-acceleration} is defined by
\[
\bar{\omega}(E) = \lim_{y \to y_1^+} \frac{L_y(E) - L_{y_1}(E)}{2\pi(y - y_1)},
\]
where \( 0 \leq y_1 < y \) is the first turning point for the piecewise affine function \( L_y(E) \). If no such turning point exists, we set \( \bar{\omega}(E) = 0 \).
\end{definition}
 
With this framework, we introduce:
\begin{definition}\label{typei}\cite{GJY}
We say \( E \) is a {\it type I energy} for the operator \( H_{v,\alpha,\theta} \) if \( \bar{\omega}(E) = 1 \). The operator \( H_{v,\alpha,\theta} \) is called a {\it type I operator} if every \( E \) in its spectrum is a type I energy.
\end{definition}
 
From the definition, it follows immediately that if \( E \in \Sigma \) with \( L(E) > 0 \) and \( \omega(E) = 1 \), then \( E \) is a type I energy. In the positive Lyapunov exponent regime, using nonperturbative localization techniques, Han and Schlag \cite{hs,HS} showed that type I operators with even potentials exhibit Anderson localization. Ge and Jitomirskaya \cite{GJ} later removed the evenness assumption by leveraging the reducibility approach of Avila-You-Zhou \cite{AYZ}.
 
These results naturally lead to the conjecture that the dual model \( L_{v,w,\alpha,\theta} \)  of a {\it type I operator} \( H_{v,\alpha,\theta} \) should exhibit purely absolutely continuous spectrum for \textbf{all phases}. Prior to this work, such all-phases results were only achievable for the unperturbed subcritical almost Mathieu operator by Avila \cite{Avila: ac}. We confirm this conjecture for $\alpha\in DC$.
 
\begin{theorem}\label{type1ac}
Let \( \alpha \in DC \), let \( H_{v,\alpha,\theta} \) be a {\it type I operator} with non-constant trigonometric potentials such that \( L(E) > 0 \) for all \( E \in \Sigma \). Then \( L_{v,w,\alpha,\theta} \) has purely absolutely continuous spectrum for all \( \theta \).
\end{theorem}

As a corollary, \textit{any finite-range perturbation of a subcritical almost Mathieu operator retains purely absolutely continuous spectrum for all phases}. This extends Avila's well-known result for the unperturbed case \cite{Avila: ac}.

\begin{coro}\label{pamo}
Let \( \alpha \in DC \), \( 0 < |\lambda| < 1 \), and \( w(\theta) = \sum_{k=-m}^m w_k e^{2\pi i k\theta} \). There exists \( \varepsilon_0 = \varepsilon_0(\alpha, w, \lambda) > 0 \) such that for \( |\varepsilon| < \varepsilon_0 \), the operator
\begin{equation}\label{long-range-qp-1}
(\tilde{L}_{w,\alpha,\theta}u)_n = \varepsilon \sum_{k=-m}^m w_k u_{n+k}+u_{n+1}+u_{n-1} + 2\lambda\cos 2\pi(\theta + n\alpha)u_n, \quad n \in \mathbb{Z},
\end{equation}
has purely absolutely continuous spectrum for all \( \theta \).
\end{coro}

\begin{rem}
However, it is still open whether the corresponding result holds if the center is not two dimensional. 
While subordinacy theory provides a critical framework for addressing this, current methodologies \cite{Avila: ac,Eli92} rely heavily on the relationship between resonant energies and sublinear growth of associated cocycles—a connection that fails for high-dimensional cocycles (\(m > 1\)) \cite{WXYZ}. 
 We contend that resolving this challenge demands a fundamentally new perspective on high-dimensional quasi-periodic cocycle dynamics, consult Section \ref{all-ac} for more discussions.
\end{rem}

 \subsubsection{Absolutely continuous spectrum for infinite-range operators}

For finite-range operators, we obtain pure absolutely continuous spectrum for almost all phases with complete exclusion of point spectrum. This result naturally raises the question of whether such spectral properties persists when considering infinite-range operators.

\begin{theorem}\label{acqq}
Let  $\alpha \in \mathrm{DC}$,  $w(\cdot), v(\cdot)\in C^\omega(\mathbb{T}^d, \mathbb{R})$. There exists $\varepsilon_1= \varepsilon_1(\alpha,v,w) > 0$ such that for $|\varepsilon| < \varepsilon_1$,  
 $L_{\varepsilon v, w, \alpha, \theta}$ has absolutely continuous spectrum for all $\theta$. Furthermore, it has no point spectrum for all $\theta$.
\end{theorem}

\begin{rem}
\begin{enumerate}
    \item To the best knowledge of the authors, Theorem \ref{acqq} provides the first absolute continuity result for quasi-periodic infinite-range operators. 
    \item While our methods confirm the absence of point spectrum, the infinite-range nature of the interaction poses significant technical challenges for establishing spectral purity (detailed in Section~\ref{proof of inf}). We nevertheless demonstrate that absolutely continuous spectrum persists under long-range coupling, extending key features of short-range behavior to this broader setting.
\end{enumerate}
    
\end{rem}

\subsection{Difficulty and methodology}\label{method}

Establishing spectral results under the partial hyperbolicity assumption for the cocycle \((T, A_E)\) requires addressing several key difficulties:
\begin{enumerate}
    \item The bundles \(E_{A_E}^{*}\) may be non-trivial, obstructing the selection of a global basis;
    \item Non-commuting matrix products complicate the isolation of hyperbolic components' influence;
    \item Most critically, the center bundle \(E_{A_E}^{c}\) lacks Schr\"odinger structure, necessitating recovery of its \(E\)-dependence.
\end{enumerate}
To overcome these challenges, we develop two principal innovations:

\subsubsection{Insights from hyperbolic geodesic flow}\label{hgf}
The core argument of this work establishes a fundamental correspondence between the spectral analysis of discrete Schr\"odinger operators on strips and the dynamical structure of geodesic flows on Riemannian manifolds. This analogy arises from their shared symplectic nature and is reflected in parallel intersection problems within their respective phase spaces.

\begin{itemize}
\item {\textbf{Geometric side} (Geodesic Flows)} 
In the context of geodesic flows, a key question concerns the trivial intersection of the stable subbundle with the vertical subbundle $\mathcal{V}$ of the double tangent bundle (consult Section \ref{ver-geo} for more explanations), a property closely tied to the existence of conjugate points. For example, 
Klingenberg's classical result \cite{Kl74} establishes a hyperbolic geodesic 
\footnote{where the flow restricted to its orbit closure is uniformly hyperbolic.} segment $c([0,t_1])$ contains conjugate points \textit{if and only if} the stable bundle $E^s(t_0)$ intersects the vertical bundle $\mathcal{V}$ non-trivially at some $t_0 \in (0,t_1)$–manifested through vanishing property of Jacobi fields (\cite{Kl74}, Section 6).

\item{\textbf{Spectral side} (Schr\"odinger Operators on the strip)}
In the setting of Schr\"odinger operators, we define an analogous vertical subbundle $\mathcal{V}$ outlined in Section \ref{ver-sch}, inspired by the vertical bundle in Riemannian geometry. A crucial observation emerges: 
The analogous condition $E^s_{A_E} \cap \mathcal{V} \neq \{0\}$ generates exponential localized eigenfunction for the half-line operator, 
when $E^s_{A_E} \cap \mathcal{V} = \{0\}$,
the spectral analysis becomes more manageable due to a uniform bounded angle estimate (Proposition \ref{lem: key ang est1}). This allows us to combine geometric methods (Lemma  \ref{lem: ang imp slow1}) with a non-stationary generalization of telescoping argument (Section \ref{telescoping}), effectively eliminating the influence of hyperbolic component of $A_E$ on estimates of 
 $\Im M$ and yielding the desired spectral measure bounds (Corollary \ref{JL}). 
\end{itemize}

This correspondence identifies $\mathcal{V}$ as a \textit{singularity detector}, a role we summarize into Table 1. This observation illuminates a profound structural correspondence between geometric and spectral theories—one that we anticipate will provide a foundational framework for future research, either from the geodesic flows side or from the spectral theory side.

\begin{table}[htb]   
	\begin{center}   
		\label{table:1} 
		\begin{tabular}{|c|c|c|}   
			\hline   \textbf{} & \textbf{\makecell[c]{Riemannian manifold \\with hyperbolic geodesic flows} } & \textbf{\makecell[c]{Half line\\ Schr\"odinger operators}} \\   
				\hline   Dominated splitting & $TT^1M=E^s\oplus E^c \oplus E^u$& $\C^{2m}=E^s\oplus E^c\oplus E^u$  \\ 
                \hline Vertical bundle $\mathcal{V}$ & $\mathcal{V} = \{0\} \times \{w \in T_xM \mid w \perp v \}$ & $\mathcal{V}=\{0\} \times \mathbb{C}^m$ \\
			\hline  $E^s \cap \mathcal{V} \neq \{0\}$  & existence of conjugate points & point spectral measure  \\ 
			\hline $E^s \cap \mathcal{V} = \{0\}$  & non-existence of conjugate points & continuous  spectral measure \\  
			\hline   
		\end{tabular}     
	\end{center}   
    \caption{Correspondence between Geodesic Flows and Schr\"odinger Operators.}  
\end{table}

\subsubsection{Monotonicity argument}
Monotonic cocycles were first introduced by Avila and Krikorian \cite{AK} 
for $\SL(2,\R)$ cocycles, and  further extend to  symplectic cocycles in 
\cite{LiuXu,Xu}. This framework provides geometric insights into the Schr\"odinger operator through dynamical methods and can be applied to more general cocycles that lack the Schr\"odinger operator's specific structure.

Monotonicity theory serves as a crucial tool for extending Kotani theory \cite{AK}---fundamental to the study of Schr\"odinger operators---to broader classes of cocycles. It is particularly effective for analyzing rotational reducibility of cocycles not homotopic to the identity, as well as relationships between rotation number (or IDS) and Lyapunov exponents \cite{AK}. However, existing monotonicity theories rely on specific coordinate systems, and no general framework exists for cocycles defined on non-trivial bundles. Consequently, there has been no monotonicity theory for cocycles restricted to dynamically defined invariant sub-bundles (which are typically non-trivial).

This limitation arises naturally when studying dual cocycles of Schr\"odinger cocycles---for instance, in establishing Kotani theory. Here, partially hyperbolic cocycles require monotonicity theory for restrictions to their center bundles. Resolving these challenges demands a \textbf{general monotonicity theory for arbitrary bundles}. In this paper, we:
\begin{enumerate}
    \item Introduce a coordinate-independent definition of monotonicity (Definition \ref{def: mono HSp cccle}) and establish a \textbf{comprehensive monotonicity theory} for cocycles on general bundles.
    \item Using differential geometry of principal bundles (Theorem \ref{thm: pure geo}), prove a key result (Theorem \ref{mon1}) critical for spectral theory of dual operators:
\end{enumerate}
\begin{quote}
    \textit{For any monotonic family of partially hyperbolic Hermitian symplectic cocycles, their restriction to the center bundle induces (up to coordinate change) a monotonic family of lower-dimensional Hermitian symplectic cocycles.}
\end{quote}

In the following, we state the ideas of the proof of subordinacy and its application to absolutely continuous spectrum: 

\subsection{Ideas of the proof: subordinacy}

\subsubsection{Finite-range subordinacy:}\label{idea1}
The case when $m=1$ in Theorem \ref{thm:main-spectral} is classical, with multiple existing proofs. One approach, following Simon's ingenious argument \cite{Simon}, employs a complex deformation of cocycles combined with telescoping estimates to derive bounds on the imaginary part of the Weyl $M$-function. 
However, for $m>1$, a direct generalization of Simon's method to estimate the imaginary part of the matrix-valued Weyl $M$-function encounters a fundamental obstruction: it would require control over the growth of the cocycle $A_E$ (not only that of $A_E|_{E^c_{A_E}}$!), which typically exhibits exponential growth (due to interference from the behavior of $A_E$ in both $E^u_{A_E}$ and $E^s_{A_E}$ directions). This explains why Theorem \ref{thm:main-spectral}, though widely believed to be true, has remained unproven.

As discussed in Section \ref{hgf}, insights from hyperbolic geodesic flow show that the condition 
$E^s_{A_E} \cap \mathcal{V} = \{0\}$ enables us to effectively relate the spectral measure 
to the growth of the cocycle restricted to the center bundle (Corollary \ref{JL}). 
Therefore, the key step reduces to analyzing the exceptional energies where 
$E^s_{A_E} \cap \mathcal{V} \neq \{0\}$. In the Schr\"odinger case, localized eigenfunctions emerge, leading to a loss of control over $\Im M^+_{E+i\epsilon}$. To address this challenge, we employ a key insight connecting spectral theory and dynamical systems—generalizing the classical Wronskian argument for scalar Schr\"odinger operators (which plays a pivotal role in Cantor spectrum problems \cite{AJ3,GJY,Puig11}). 
Although it has long been recognized that the Wronskian argument does not extend directly to higher-dimensional matrices due to the coexistence of localized eigenfunctions, we demonstrate that within the symplectic framework, point spectrum is precluded in the central direction. This implies $\mu_\omega(\mathcal{B}_\omega \setminus \mathcal{G}_\omega^s) = 0$, where 
\[
\mathcal{G}^s_\omega = \left\{ E \in \Sigma : E^s_{A_E}(\omega) \cap \mathcal{V}_\omega = \{0\} \right\}.
\]
See Section \ref{comproof} for complete details.

From this aspect, we should mention in the study of hyperbolic geodesic flows and conjugate points, it is also crucial to rule out the possibility that $E^s$ (or more generally, any invariant isotropic subspace) intersects the vertical bundle non-trivially. A classical result by Ma\~n\'e (\cite{Ma}, Lemma III.2) states:
\begin{quote}
\textit{If $\theta \in SM$ and $E \subset S(\theta)$ is a Lagrangian subspace, then the set of $t \in \mathbb{R}$ such that $d\phi_t^\theta (E) \cap \V(\phi_t(\theta)) = \{0\}$ is discrete.}
\end{quote}
In attempts to prove Ma\~n\'e's conjecture, that hyperbolic geodesic flows with curvature bounded below have no conjugate points, establishing the triviality of the intersection $E^s \cap \V$ is a very natural approach, as seen, for instance, in a recent try \cite{Se}.

\subsubsection{Infinite-range subordinacy:}

The analysis of infinite-range operators presents a significant challenge due to the lack of robust tools from dynamical systems. To overcome this challenge,
 we  focus on the primary objective: characterizing the absolutely continuous spectrum in terms of generalized eigenfunctions. As a result, we do not require the dynamically defined potential as presented in Theorem \ref{thm:main-spectral}.

We extend the methodology of Kislev and Last \cite{KL1} for analyzing essential supports of absolutely continuous spectrum in $\mathbb{Z}^d$ (respectively $\mathbb{R}^d$) operators. The approach centers on the \emph{Lagrange bilinear form}:
\begin{equation*}\label{eq:lagrange-form}
    W_{[-r,r]}(f,g) := \langle Hf, g \rangle_r - \langle f, Hg \rangle_r = \sum_{n=-r}^{r} \left[ (Hf)_n \overline{g_n} - f_n \overline{(Hg)_n} \right].
\end{equation*}
The estimation of $ W_{[-r,r]}(f,g) $ is more complex than the one-dimensional case ($ m=1 $) analyzed in \cite{KL1}. The analogous analysis can be generalized to finite-dimensional cases ($ m < \infty $), but it fails in the infinite-dimensional case ($ m = \infty $). A key observation is that if $ w_k $ exhibits moderate decay, then $ W_{[-r,r]}(f,g) $ for the infinite-range operator can be well approximated by the finite-range operator. Specifically,
for $v_\epsilon(\cdot) = (H - E - i\epsilon)^{-1}\phi$ with $\phi$ is chosen to satisfy the conditions of Theorem  \ref{sub}, we establish that
\begin{equation*}
   \bigg|\sum_{r=0}^R \langle \phi,u\rangle_r\bigg|-  \|v_\epsilon\|_{\ell^2({2R})} \|u\|_{\ell^2({2R})} \lesssim \left|\sum_{r=0}^R W_{[-r,r]}(v_\epsilon, u)\right| \lesssim \|v_\epsilon\|_{\ell^2(\mathbb{Z})} + \|v_\epsilon\|_{\ell^2({2R})} \|u\|_{\ell^2({2R})}.
\end{equation*}
This inequality chain enables Green's function estimates  through lower bounds for $\|v_\epsilon\|_{\ell^2(\mathbb{Z})}$.

In contrast to the finite-dimensional case ($m<\infty$), one advantage of the Theorem \ref{sub} and Corollary \ref{cor:subordinacy} is that it requires the construction of only one bounded solution, rather than necessitating that all solutions be bounded, as stated in Theorem \ref{thm:main-spectral}. However, this approach results in the limitation of obtaining only the existence of the absolutely continuous spectrum, thereby sacrificing purity.

Another key observation is that 
while $\mathcal{WB}$ in \eqref{weak-bdd-sol} excludes bounded solutions for $\mathbb{Z}^d$ ($d \geq 2$), it \emph{does} admit such solutions in $\mathbb{Z}^1$. This dichotomy reveals:
subordinacy theory is inherently a one-dimensional phenomenon. 

\subsection{Ideas of the proof: Absolutely continuous spectrum}\label{idea3}
The spectral analysis of finite/infinite-range operators presents significant challenges.  The previous proof of  purely absolutely continuous spectrum relies on the quantitative almost reducibility, which offers a precise estimate of the spectral measure \cite{Avila: ac, Eli92}. However, current KAM techniques do not yield an exact quantitative estimate in the high-dimensional case ($m>1$) and seem powerless in the infinite-dimensional case. An alternative approach is derived from dual arguments. A common belief posits that if $H_\theta$ has pure point spectrum for almost every $\theta$, its dual operator $\widehat{H}_x$ should exhibit purely absolutely continuous spectrum. However, this duality argument crucially depends on the structure of $\widehat{H}_x$ \cite{BJ,DF2,J1999}, which is a scalar Schr\"odinger operator ($m=1$), resulting in the two-dimensional nature of solutions in key cases. 

\subsubsection{Pure AC spectrum of finite-range operators}
For finite-range operators, we circumvent duality arguments by directly analyzing the Schr\"odinger cocycle \((\alpha, A_E)\) through the following framework:  
\begin{itemize}
    \item Apply Theorem~\ref{thm:main-spectral} via KAM (Kolmogorov-Arnold-Moser) methods to prove that \((\alpha, A_E)\) is partially hyperbolic for a full-measure set of energies \(E\);
    \item Establish uniform boundedness of transfer matrices when restricted to the center subspace \(E^c_{A_E}\);
    \item Exclude the singular continuous spectrum using the absolute continuity of the IDS.
\end{itemize}

\subsubsection{All phases AC spectrum of finite-range operators}\label{all-ac}
To establish purely absolutely continuous spectrum for the dual operator of \textit{type I operators} for all phases, our proof relies on three important parts:  our newly developed subordinacy theory (Theorem \ref{thm:main-spectral}) and its proof framework,  Theorem \ref{mon1} in monotonicity theory, an adaptation of the fundamental approach from \cite{Avila: ac}.

Let's explain in details. Main argument in \cite{Avila: ac} relies on three key components:
\begin{enumerate}
    \item Quantitative almost reducibility estimates;
    \item Quantitative subordinacy theory;
    \item Lower bounds for the H\"older exponent of the integrated density of states (IDS).
\end{enumerate}
We address these components as follows, aiming to illustrate how the methodology developed in Section \ref{method} comes into play:
\begin{enumerate}
\item \textbf{Quantitative almost reducibility:} Leveraging the monotonicity argument (Proposition \ref{cr}), we first establish the monotonicity of the center bundle. We then apply the recently developed quantitative version of Avila's global theory \cite{GJYZ} to reduce the problem to a two-dimensional setting. This reduction enables us to utilize well-established results on quantitative almost reducibility \cite{Avila: ac, Eli92, LYZZ,WXYZZ}.
\item \textbf{Quantitative subordinacy theory:} For exceptional energies where \( E^s_{A_E} \cap \mathcal{V} \neq \{0\} \), the existence of a quantitative subordinacy theory is precluded (as detailed in Corollary \ref{JL}). However, Proposition \ref{lem: key ang est1} demonstrates that \( E^s_{A_E} \cap \mathcal{V} = \{0\} \) is an open condition, which supports a local quantitative subordinacy theory. Moreover, since these exceptional energies are countable, standard methods (Section \ref{ab-point}) allow us to eliminate point spectrum.  
\item \textbf{Lower bound for Hölder exponent (Proposition \ref{holderl}):} Two key observations are critical: 
(a) Dominated splitting guarantees that the sum of Lyapunov exponents restricted to \( E^s_{A_E} \) are harmonic in a neighborhood of \( E \), exhibiting minimal variation; consequently, the dominant contribution arises from the center bundle.  
(b) Traditional methods (e.g., Deift-Simon \cite{DeS}), which rely on the Thouless formula and harmonic analysis, fail here because the center cocycle lacks a Schrödinger structure. Notably, the Thouless formula does not yield useful information when restricted to the center bundle. Instead, we adapt techniques from Avila-Krikorian \cite{AK} to prove the monotonicity of a two-dimensional cocycle induced from the center dynamics of our dual Schr\"odinger cocycle (Proposition \ref{cr}), which facilitates estimates of the Lyapunov exponents.  
\end{enumerate}

\subsubsection{Ac spectrum for infinite-range operators:}
For infinite-range operators, Theorem~\ref{sub} necessitates the demonstration of bounded solutions supported on sets of positive measure. While duality principles suggest that $\ell^1(\mathbb{Z}^d)$ eigenfunctions of the dual operator $\widehat{H}_x$ could generate such solutions, a central challenge emerges: the countable cardinality of eigenfunctions per phase obstructs the construction of eigenfunction families with \textbf{positive measure}.  
In the Schr\"odinger operator settings, the regularity of IDS in zero Lyapunov exponent regimes \cite{DeS, Puig1} ensures full-measure spectral conclusions. However, this argument remain inherently two-dimensional in scope. To transcend this limitation, we reinterpret the problem through duality: pure point spectrum corresponds to the diagonalization of an infinite-dimensional matrix. We augment this perspective with refined eigenvalue analysis of the operator \( E_\infty(\cdot ) \).  Crucially, KAM techniques really enable the explicit construction of \textbf{positive-measure} eigenfunctions for dual operators  \cite{E}. This construction yields corresponding bounded solutions for the original infinite-range operators.

\subsection{Outline of the paper}

The remainder of this work is structured as follows. Section \ref{pre} presents preliminary definitions and foundational results. Section \ref{ideah} develops geometric insights from hyperbolic geodesic flows. Section \ref{idea2} establishes monotonicity theory for center-bundle cocycles. Subordinacy theory is developed for finite-range operators in Section \ref{subor} and extended to infinite-range cases in Section \ref{proof of inf}. Section \ref{ab-point} proves the absence of point spectrum for long-range quasi-periodic operators. Finally, we demonstrate three spectral results: Theorem \ref{ac} (absolutely continuous spectrum for finite-range quasi-periodic operators) in Section \ref{proof-finite};
Theorem \ref{type1ac} (pure absolutely continuous spectrum for all phases) in Section \ref{all};
 Theorem \ref{acqq} (infinite-range cases) in Section \ref{proof-infinite}.

\section{Preliminaries}\label{pre}

\subsection{Grassmannian }
In this subsection, we  briefly recall the holomorphic structure of the Grassmannians. The set of $k$-dimensional subspaces of $\C^d$ is a compact Grassmannian manifold with a holomorphic
structure and will be denoted by $G(k,d)$.
Let  $$M_k(d)=\{M\in \mathbb{C}^{d\times k}:\operatorname{rank}(M)=k\}.$$

\begin{theorem} \label{holomorphic}\cite{AJS}The following results hold:
\begin{enumerate}[{\rm (i)}]
\item There is a natural projection $\tilde p:\MM_k(d)\to G(k,d)$ which is also a holomorphic submersion.
\item Locally, for each  $M \in \MM_k(d)$ 
there exists neighborhood $U_M$ of $\tilde p(M)$ and
holomorphic injections $i_M: U_M\to \MM_k(d)$
such that $\tilde p \circ i_M=\id|U_M$.
\item Any $u \in C^0(\mathbb{T},G(k,d))$ can be lifted to a one-periodic function $\Tilde{u}: \T \rightarrow \MM_k(d)$ such that for any $ \theta \in \mathbb{T}, u(\theta)$ is spanned on $\mathbb{C}$ by the $n$ column vectors of $\Tilde{u}(\theta)$
         In other words, there exists $u_1,u_2,\dots,u_k \in C^0(\mathbb{T},\mathbb{C}^d) $ such that for any 
         $\theta \in \mathbb{T}$,
         $u(\theta)=\operatorname{span}_\mathbb{C}\{u_1(\theta),u_2(\theta),\dots,u_k(\theta)\}$.
         Furthermore, if $u \in C^\omega(\mathbb{T},G(k,d))$, be lifted to a one-periodic holomorphic function $\Tilde{u}: \T_\delta \rightarrow \MM_k(d)$.
\end{enumerate}
\end{theorem}
Though the holomorphic dependent of Grassmannian, we have the following Lemma:   

\begin{lemma}\label{parbasis}
    Suppose \( E_t(\theta) \in C^\omega(\I \times \mathbb{T}, G(k,d)) \), where \( \I \) is a neighborhood of \( 0 \) in \( \mathbb{R} \). Then there exists \( \varepsilon > 0 \) and analytic mappings \( u_{1,t}, u_{2,t}, \dots, u_{k,t} \in C^\omega((-\varepsilon,\varepsilon) \times \mathbb{T}, \mathbb{C}^d) \) such that for any \( \theta \in \mathbb{T} \) and \( t \in (-\varepsilon,\varepsilon) \), the subspace \( u_t(\theta) \) is exactly the complex span of these mappings:
    \[
    E_t(\theta) = \operatorname{span}_\mathbb{C}\{u_{1,t}(\theta), u_{2,t}(\theta), \dots, u_{k,t}(\theta)\}.
    \]
\end{lemma}

\begin{proof}
    Let \( P_t(\theta) \) denote the orthonormal projection onto the subspace \( u_t(\theta) \). Since \( u_t(\theta) \) is a holomorphic family of subspaces, the projection operator \( P_t(\theta) \) is also holomorphic in \( (t,\theta) \). By Theorem \ref{holomorphic}, there exist initial analytic basis vectors \( u_{1,0}, u_{2,0}, \dots, u_{k,0} \in C^\omega(\mathbb{T}, \mathbb{C}^d) \) such that for each \( \theta \in \mathbb{T} \),
    \[
    u_0(\theta) = \operatorname{span}_\mathbb{C}\{u_{1,0}(\theta), u_{2,0}(\theta), \dots, u_{k,0}(\theta)\}.
    \]
    Define the time-dependent vectors by \( u_{i,t}(\theta) := P_t(\theta)u_{i,0}(\theta) \). We claim these vectors form an analytic basis for \( u_t(\theta) \).

    Since \( P_t(\theta) \) is holomorphic and \( u_{i,0}(\theta) \) is analytic, the composition \( u_{i,t}(\theta) \) is analytic in \( (t,\theta) \). By construction, each \( u_{i,t}(\theta) \) lies in \( u_t(\theta) \) (as the image of the projection). To see they span \( E_t(\theta) \), note that for \( t=0 \) they recover the original basis, and the holomorphic dependence ensures the span remains \( k \)-dimensional (hence equal to \( u_t(\theta) \)) for sufficiently small \( |t| < \varepsilon \).
\end{proof}

\subsection{Complex  cocycles}\label{com}
Let $T:\Omega\to \Omega$ be a continuous map,  $A\in C^0(\Omega,{\rm M}(m,\C))$, a cocycle $(T, A)$ is a linear skew product:
$$
(T,A)\colon \left\{
\begin{array}{rcl}
	\Omega \times \CC^{m} &\to& \Omega \times \CC^{m}\\[1mm]
	(\omega,v) &\mapsto& (T\omega,A(\omega)v)
\end{array}
\right.  .
$$
For $n\in\mathbb{Z}$, $A_n$ is defined by $(T,A)^n=(T^n,A_n).$ Thus $A_{0}(\omega)=id$,
\begin{equation*}
	A_{n}(x)=\prod_{j=n-1}^{0}A(T^{j}\omega)=A(T^{n-1}\omega)\cdots A(T\omega)A(\omega),\ for\ n\ge1,
\end{equation*}
and $A_{-n}(\omega)=A_{n}(T^{-n}\omega)^{-1}$.

We denote by $L_1(A)\geq L_2(A)\geq...\geq L_m(A)$ the Lyapunov exponents of $(\alpha,A)$ repeated according to their multiplicities, i.e.,
$$
L_k(A)=\lim\limits_{n\rightarrow\infty}\frac{1}{n}\int_{\Omega}\ln(\sigma_k(A_n(x)))d\nu,
$$
where for any matrix $B\in {\rm M}(m,\C)$, we denote by
$\sigma_1(B)\geq...\geq \sigma_m(B)$ its singular values (eigenvalues
of $\sqrt{B^*B}$).  Since the k-th exterior product $\Lambda^k A_n$ satisfies $\sigma_1(\Lambda^k A_n)=\|\Lambda^k A_n\|$, $L^k(A)=\sum_{j=1}^kL_j(A)$ satisfies
$$
L^k(A)=\lim\limits_{n\rightarrow \infty}\frac{1}{n}\int_{\Omega}\ln\|\Lambda^kA_n(x)\|d\nu.
$$

\subsection{Uniform hyperbolicity and dominated splitting}\label{secds}
 Recall that
for complex cocycles $(T,A)\in C^0(\Omega,{\rm M}(m,\C))$, Oseledets theorem provides us with  strictly decreasing
sequence of Lyapunov exponents $L_j \in [-\infty,\infty)$ of multiplicity $m_j$, $1\leq j \leq \ell$ with $\sum_{j}m_j=m$, and for $a.e.$ $\omega$, there exists
a measurable invariant decomposition $$\C^m=E^{1}(\omega)\oplus E^2(\omega)\oplus\cdots\oplus E^{\ell}(\omega)$$ with $\dim E_\omega^j=m_j$ for $1\leq j\leq \ell$ such that $$
\lim\limits_{n\rightarrow\infty}\frac{1}{n}\ln\|A_n(\omega)v\|=L_j,\ \  \forall v\in E_\omega^j\backslash\{0\}.
$$
An invariant decomposition $\C^m=E^{1}(\omega)\oplus E^2(\omega)\oplus\cdots\oplus E^{\ell}{(\omega)}$  is   dominated if for any unit vector $v_j\in E^j(\omega)\backslash \{0\}$, we have $$\|A_n(\omega)v_j\|>\|A_n(\omega)v_{j+1}\|.$$
Oseledets decomposition is a priori measurable,  if an invariant decomposition $\CC^m=E^{1}(\omega)\oplus E^2(\omega)\oplus\cdots\oplus E^{\ell}(\omega)$  is   dominated, then $E^j(\omega)$ depends  continuously on $\omega$.

We also recall that  $(T,A)$  is called $k$-dominated (for some $1\leq k\leq m$) if there exists a dominated decomposition $\CC^m=E^+(\omega) \oplus E^-(\omega) $    with $\dim E^+_\omega =k.$  It follows from the definitions that the Oseledets splitting is dominated if and only if $(T,A)$ is $k$-dominated for each $k$ such that  $L_k(T,A)> L_{k+1}(T,A)$.

\begin{proposition}\label{lem: Hol dep AJS}\cite{AJS}If 
$(T,A)$ is dominated, then for any $1\leq j\leq \ell$,   $E^j(\omega)$ depend holomorphically on $A$.
\end{proposition}

In this paper, we focus on the case that the associated Schr\"odinger cocycle (on the strip) $A_E$ is {\it partially hyperbolic}, which means that there is an $A_E$-invariant dominated splitting $E^u_{A_E}\oplus E^c_{A_E}\oplus E^s_{A_E}$ of $\CC^{2m}$ everywhere, and  there exist  some constants $C>0,c>0$, and for every $n\geqslant 0$,
$$
\begin{aligned}
	\| A_n(\omega)v\| \leqslant Ce^{-cn}\| v\|, \quad & v\in E^s(\omega),\\
	\| A_n(\omega)^{-1}v\| \leqslant Ce^{-cn}\|v\|,  \quad & v\in E^u(T^n\omega).
\end{aligned}
$$ 
By classical cocycle theory, such splitting persists and varies continuously under small perturbation of $A_E$, hence partial hyperbolicity of $A_E$ is a robust property for both $V$ and $E$. 
In particularly, if $\CC^{2m}=E^u_{A_E}\oplus E^s_{A_E}$, we say the  cocycle $(T, A_E)$ is {\it uniformly hyperbolic}.

\subsection{Basic properties of Hermitian symplectic matrices}\label{subsec: H-S}

Recall that \( \Sp(2d,\mathbb{R}) \) (resp. \(\HSp(2d) \)) denotes the set of symplectic (resp. Hermitian symplectic) matrices, defined by
\[
\Sp(2d,\mathbb{R}) = \left\{ A \in \mathbb{R}^{2d \times 2d} \,\big|\, A^T J A = J \right\}, \quad \HSp(2d) = \left\{ A \in \mathbb{C}^{2d \times 2d} \,\big|\, A^* J A = J \right\},
\]
where \( J = \begin{pmatrix} O & -I_d \\ I_d & O \end{pmatrix} \) denotes the standard symplectic structure on \( \mathbb{R}^{2d} \) (resp. \( \mathbb{C}^{2d} \)). 

We now recall the direct sum operation for Hermitian symplectic groups. Let
\[
S_1 = \begin{pmatrix} A_1 & A_2 \\ A_3 & A_4 \end{pmatrix} \in \HSp(2n_1) \quad \text{and} \quad S_2 = \begin{pmatrix} B_1 & B_2 \\ B_3 & B_4 \end{pmatrix} \in \HSp(2n_2),
\]
then their direct sum is defined as
\[
S_1 \diamond S_2 = \begin{pmatrix}
A_1 & O & A_2 & O \\
O & B_1 & O & B_2 \\
A_3 & O & A_4 & O \\
O & B_3 & O & B_4
\end{pmatrix} \in \HSp(2n_1 + 2n_2).
\]
Note that while we state this action for Hermitian symplectic groups, an analogous construction holds for symplectic groups.

We now extend the standard symplectic structure to a more general context.

\begin{definition}
A Hermitian symplectic structure on \( \mathbb{C}^{2m} \) is given by an antisymmetric non-degenerate skew-linear 2-form \( \psi(\cdot, \cdot) \), which is conjugate-linear in the first argument and linear in the second argument, satisfying:
\begin{enumerate}
\item Skew-Hermitian property: \( \psi(X, Y) = -\overline{\psi(Y, X)} \) for all \( X, Y \in \mathbb{C}^{2m} \);
\item Non-degeneracy: If \( \psi(X, Y) = 0 \) for all \( Y \in \mathbb{C}^{2m} \), then \( X = 0 \).
\end{enumerate}
\end{definition}
\begin{definition}
    A vector $v\in \C^{2m}$ is called isotropic if $\psi(v,v)=0$.
\end{definition}

\begin{definition}[Hermitian Symplectic Subspaces]
A Hermitian symplectic subspace of \( \mathbb{C}^{2m} \) is a subspace \( V \subseteq \mathbb{C}^{2m} \) such that the restriction of the symplectic form \( \psi \) to \( V \times V \) is non-degenerate. The Grassmannian of all \( 2k \)-dimensional Hermitian symplectic subspaces of \( \mathbb{C}^{2m} \) is denoted by \( G_{\HSp}(2k, 2m) \).
\end{definition}
\begin{definition}[Signature of a Hermitian symplectic space]
    let $V=\operatorname{span}_{\C}(v_1,v_2,\cdots,v_{2n})$ be a basis of 
the Hermitian symplectic subspace $V \subset \C^{2d}$. We define the  Krein matrix by  
\begin{equation*}
G(v_1,v_2,\cdots,v_{2n})=i(\psi(v_i,v_j))_{1\leq i,j\leq2n} .
\end{equation*}
 Its congruence normal form is given by $ \operatorname{diag}(I_p, -I_q) $, where $ p $ represents the positive inertia index and $ q $ the negative inertia index of $ G $, with the condition that $ p + q = 2n $. It is clear that $ p - q $ are uniquely determined by $V $, referred to as the signature  of $ V $ and denoted by $ \operatorname{sign}(V) $.
\end{definition}

In this paper, unless otherwise specified, we always equipped $\C^{2m}$ the Hermitian symplectic structure introduced by $(T,A_E)$, i.e. 
\begin{equation}\label{syms}
    \psi(X,Y)=X^*SY, \text{ where } S=\begin{pmatrix}
	0&-C^*\\
	C&0
\end{pmatrix}.
\end{equation}
\begin{lemma}\label{wron11}\cite{GJ,WXZ}
   Suppose $(T, A_E)$ is partially hyperbolic, then $E^c_{A_E}$ is a Hermitian-symplectic subspace.    
\end{lemma}

\subsection{Quasi-periodic cocycle, Rotation number}

If \( T: \mathbb{T} \to \mathbb{T} \) is defined by \( T\theta = \theta + \alpha \), where \( \alpha \in \mathbb{R} \setminus \mathbb{Q} \), we denote this cocycle by \( (\alpha, A) \) and refer to it as a quasi-periodic cocycle.

In particular, if $A\in C^\omega(\T,{\rm M}(m,\C)),$ the Lyapunov exponent is  continuous with respect to $(\alpha,A).$
\begin{theorem}\cite{AJS}\label{Lecon}
	The function $\R \times \C^{\omega}(\T, {\rm M}(m,\C)\ni
	(\alpha,A)\mapsto L_k(\alpha,A)$
	are continuous at any $(\alpha',A')$ with $\alpha'\in \R\setminus\Q$.
\end{theorem}

Now suppose $ A\in C^0(\T,\mathrm{HSp}(2m)) $ is {\it homotopic to the identity}, there exists a continuous map $ \tilde{F}_{T,A}$ acting on the covering space $\T\times \R\times \mathrm{SU}(m)  $, of the form $ \tilde{F}_{T,A}(\theta,x,S)=(\theta+\alpha,x+f(\theta,x,S),*) $, such that $ f(\theta,x,S)=f(\theta,x',S') $ whenever $ (x,S) $, $ (x',S')\in \R\times \mathrm{SU}(m) $ projects to the same point $ W_{\Lambda}\in \mathrm{U}(m) $. In order to simplify the terminology we shall say that $  \tilde{F}_{\alpha,A} $ is a {\it lift} for $ (\alpha,A) $. The map $ f $ is then descends to a map $ \T\times \mathrm{U}(m) \to \R $ and is independent of the choice of the lift, up to the integer.
   Then the limit 
  \[
	\lim_{n\to +\infty}\frac{1}{n}\sum_{k=0}^{n-1} f\left({(T,A)^k(\theta,W_{\Lambda})}\right) \mod \Z
  \] 
 is uniform in all $ (\theta,W_{\Lambda})\in \T\times\mathbf{U}(m) $, and coincides with 
  \[
	\rho(\alpha,A)=\int_{\T\times \mathrm{U}(m)} f(\theta,W_{\Lambda})d\nu \mod \Z
  \] 
  where $ \nu $ is any probability measure which is invariant under $ (\alpha,A) $ and which projects to Lebesgue measure on $ \T$. We call  $ \rho(\alpha,A) $ the {\it fibered rotation number}, which is independent of the choice of the lift.
  \begin{proposition}\cite{LW}\label{rotids}
	Let $ \mathcal{N}(E) $ denotes the {\it integrated density of states} of $ H_{V,\alpha,\theta} $. Then we have
	\[
		m(1-\mathcal{N}(E))=\rho(\alpha,A_E) \mod \Z.
	\] 
  \end{proposition}

\begin{proposition}\label{rho}\cite{LW}
	If \(A:\T\to \mathrm{HSP}(2m)\) is continuous and homotopic to the identity and if \(B:\T\to \mathrm{HSP}(2m)\) is continuous, then there exists $r\in\Z$, such that
	\[
	\rho\Bigl((0,B)^{-1}\circ (\alpha,A)\circ (0,B)\Bigr)=\rho\bigl((\alpha,A)\bigr)- r\,\alpha \mod \Z.
	\]If \(B(\cdot)\) is only defined on \(2\T\), then there exists  \(r\in \Z\), such that 
	\[
	\rho\Bigl((0,B)^{-1}\circ (\alpha,A)\circ (0,B)\Bigr)=\rho\bigl((\alpha,A)\bigr)-\frac{r\,\alpha}{2} \mod \Z.
	\]
\end{proposition}

\subsection{Aubry Duality}
Suppose there exists \( E \) such that the operator \( L_{\varepsilon v,w,\alpha,\theta} \) admits a solution \( u = (u_n)_{n \in \mathbb{Z}} \in \ell^2(\mathbb{Z}, \mathbb{C}) \). Define the Fourier transform \( \hat{u}(x) = \sum_{n \in \mathbb{Z}} u_n e^{inx} \). Then for almost every \( x \in \mathbb{T} \), the sequence \( \tilde{u} \) defined by
\[
\tilde{u}(n) = \hat{u}\left(x + \langle n, \alpha \rangle\right) e^{2\pi i \langle n, \theta \rangle}, \quad n \in \mathbb{Z}^d,
\]
serves as a solution to the dual operator \( \widehat{L}_{\varepsilon v,w,\alpha,x} \), which is defined by
\begin{equation}\label{dualo}
\left( \widehat{L}_{\varepsilon v,w,\alpha,x} u \right)_n := \varepsilon \sum_{k \in \mathbb{Z}^d} v_k u_{n+k} + w\left(x + \langle n, \alpha \rangle\right) u_n, \quad n \in \mathbb{Z}^d,
\end{equation}
where \( v_k \) denotes the Fourier coefficient of \( v(\cdot) \), and \( w(\theta) = \sum_{k} w_k e^{2\pi i k \theta} \).

Conversely, if \( \widehat{L}_{\varepsilon v,w,\alpha,x} \) has a solution \( u = (u_n)_{n \in \mathbb{Z}^d} \in \ell^1(\mathbb{Z}^d, \mathbb{C}) \), define the Fourier transform \( \hat{u}(\theta) = \sum_{n \in \mathbb{Z}^d} u_n e^{i\langle n, \theta \rangle} \). Then for any \( \theta \in \mathbb{T}^d \), the sequence \( \tilde{u} \) defined by
\[
\tilde{u}(n) = \hat{u}\left(\theta + n\alpha\right) e^{2\pi i n x}, \quad n \in \mathbb{Z},
\]
is a solution to \( L_{\varepsilon v,w,\alpha,\theta} \).

\subsection{Analytic set}
In this subsection, we recall some conclusion in classical descriptive set theory.
\begin{definition}\cite{Ke}
    Let $X$ be a Polish space. A set $A\subset X$ is called analytic if there is a Polish space $Y$ and Borel set $B\subset X\times Y$ with $A=\pi_X(B),$ where $\pi_X$ is the projection to $X.$
\end{definition}
\begin{proposition}\cite{Ke}\label{ana}
    All analytic subsets of a measurable space are universally measurable. In particular, if the measurable space is $\R$, then all analytic subsets are Lebesgue measurable.
\end{proposition}
\begin{lemma}\label{leb}
    Let $f:\T^d\rightarrow\R$ be a Borel measurable function, then for any Borel set $A\in\T^d$, $f(A)$ is the Lebesgue measurable set in $\R.$
\end{lemma}
\begin{proof}
    Let $g(x,y)=f(x)-y$, it is a Borel measurable function in $\T^d\times\R$. Consider $$g^{-1}(0):=\{(x,y)|x\in\T^d,y=f(x)\},$$
    it is a Borel measurable set in $\T^d\times\R$ by the definition of measurable function.
   Therefore $f(A)=\pi_\R(g^{-1}(0))$ is Lebesgue measurable by Proposition \ref{ana}.
\end{proof}

\section{Ideas from hyperbolic geodesic flow}\label{ideah}
\subsection{Vertical bundle and Uniform angle estimates}\label{ver-sch}

Recall that to study the spectral properties of the dynamically defined  half-line restrictions \( H^{\pm}_{V,T,\omega} \) acting on \( \ell^2(\mathbb{Z}^{\pm}, \mathbb{C}^m) \), we employ a (Hermitian-)symplectic formalism. The transfer matrices \( A_E(\omega) \) generate a cocycle structure:
$$
A_E: \Omega \times \mathbb{C}^{2m} \to \Omega \times \mathbb{C}^{2m}, \quad (\omega, v) \mapsto (T\omega, A_E(\omega)v), 
$$
with fiber decomposition at \( \omega \in \Omega \):
$
\mathbb{C}^{2m}_\omega = \mathbb{C}^m(0,\omega) \times \mathbb{C}^m(1,\omega),$
where:
\begin{itemize}
    \item \( \mathbb{C}^m(0,\omega) \) is the \textbf{position fiber} 
    (value space at the $0$-th coordinate of \( H_{V,T,\omega} \)).
    \item \( \mathbb{C}^m(1,\omega) \)  represents the \textbf{momentum fiber} (the value space at the $1$-st coordinate of \( H_{V,T,\omega} \), equivalent to the 0th coordinate of \( H_{V,T,T\omega} \)).
\end{itemize}
Once we have this, define the \textit{vertical bundle}: 
    \[
    \mathcal{V}_\omega := \ker \pi_\omega = \{0\} \times \mathbb{C}^m(1,\omega),
    \]
    where \( \pi_\omega: \mathbb{C}^{2m}_\omega \to \mathbb{C}^m(0,\omega) \) is the canonical projection to its first factor.
    
On the other hand, we will assume that $(T, A_E)$ is partially hyperbolic. This means that the fiber can be decomposed into
$$
\mathbb{C}^{2m}_\omega = E^s_{A_E}(\omega) \oplus E^c_{A_E}(\omega) \oplus E^u_{A_E}(\omega),
$$
where, due to the stability of dominated splitting, for sufficiently small $\epsilon$, the dominated splitting of $E^s_{A_E} \oplus E^c_{A_E} \oplus E^u_{A_E}$ persists if we replace $E$ with $E + i\epsilon$. In this case, we have
$$
\dim E^{\ast}_{A_{E+i\epsilon}} = \dim E^\ast_{A_E}, \quad \ast \in \{s, c, u\},
$$
and $E^\ast_{A_{E+i\epsilon}}$ depends holomorphically on $E + i\epsilon$ (see Appendix \ref{ds-con} for details). It is important to note that if $\epsilon \neq 0$, then $(T, A_{E+i\epsilon})$ is always uniformly hyperbolic \cite{Puig}. Consequently, the fiber can be decomposed into
$$
\mathbb{C}^{2m}_\omega = \E^s_{A_{E+i\epsilon}}(\omega) \oplus \E^u_{A_{E+i\epsilon}}(\omega).
$$
However, in this paper, $E^{\ast}_{A_{E+i\epsilon}}$  always denotes the analytic extension of $E^{\ast}_{A_{E}}$.

Our main goal in this section is to study the geometric consequences if
$
E^s_{A_E}(\omega) \cap \mathcal{V}_\omega = \{0\}.
$
Before delving into the results, let's discuss the origins and motivations behind these terminologies. As mentioned in the introduction, our method was inspired by hyperbolic geodesic flow.

\subsubsection{Vertical bundles and comparison to geodesic flows}\label{ver-geo}

Let $SM$ be the unit tangent bundle of a Riemannian manifold $(M,g)$. At each point $\theta = (x,v) \in SM$, the \text{vertical bundle} $\mathcal{V} \subset T_\theta(SM)$ is defined as the kernel of the restricted projection:
\[
\mathcal{V}_\theta := \{ \xi \in T_\theta(SM) \mid d\pi(\xi) = 0 \text{ and } \xi \perp \dot{\gamma}_\theta \},
\]
where $\pi : SM \to M$ is the base projection, $\dot{\gamma}_\theta$ is the geodesic flow direction at $\theta$, and the perpendicularity condition $\xi \perp \dot{\gamma}_\theta$ is with respect to the Sasaki metric.

Concretely, in local coordinates, this becomes:
\[
\mathcal{V}_{(x,v)} = \{0\} \times \{w \in T_xM \mid w \perp v \}.
\]
This bundle captures purely directional variations while fixing both the base point and the speed of geodesics, and encodes the Jacobi field initial condition 
$J(0)=0$. For those $\theta$ such that $E^s_\theta\cap \V_\theta\neq \{0\}$, along the orbit of $\theta$ one can detect conjugate points \cite{Kl74}.

To clarify the terminology, consider the continuum Schr\"odinger operator acting on $L^2(\mathbb{R}, \mathbb{C}^m)$:
\begin{equation}\label{sch}
H \mathbf{x} = \frac{d^2 \mathbf{x}}{dt^2} + V(t) \mathbf{x} = E \mathbf{x},
\end{equation}
where $V(t)$ is a bounded $m \times m$ real symmetric matrix function on $\mathbb{R}$. Let $\mathbf{v} = \frac{d\mathbf{x}}{dt}$. Then, equation \eqref{sch} can be rewritten as
$$
\frac{d}{dt} \begin{pmatrix} \mathbf{x} \\ \mathbf{v} \end{pmatrix} = \begin{pmatrix} 0 & I \\ E-V(t) & 0 \end{pmatrix} \begin{pmatrix} \mathbf{x} \\ \mathbf{v} \end{pmatrix},
$$
where $I$ is the identity matrix. This formulation clarifies the origin of the terms ``position fiber" and ``momentum fiber", and illustrates their relationship to the vertical bundle.

Then analogue to the study of conjugate points through geodesic flows, for \( E^{s}_{A_E} \), we define the (stable) critical set
\begin{align*}
\mathcal{G}^s_\omega &:= \left\{ E \in \Sigma \,:\, E^s_{A_E}(\omega) \cap \mathcal{V}_\omega = \{0\} \right\}.
\end{align*}
The critical set characterizes energies where stable directions avoid vertical bundles (note actually in our case the vertical bundle $\mathcal{V}_{\omega} = \{0\} \times \mathbb{C}^m$). The trivial intersection condition \( E^s_{A_E}(\omega) \cap \mathcal{V}_\omega = \{0\} \) excludes exponential decay of solutions, an evidence of absolutely continuous spectrum.

\subsubsection{Uniform angle estimates and its consequence}\label{sec: ang est1}

To further elucidate the geometric implications, we demonstrate that for any energy $ E \in \mathcal{G}^s_\omega $, it is possible to select $ m $-linearly independent vectors that maintain uniform angle bounds with $ E^s_{A_{E+i\epsilon}}(\omega) $:

\begin{proposition}\label{lem: key ang est1}
Assume that \( (T,A_E) \) is partially hyperbolic, and the energy \( E \in \mathcal{G}^s_\omega \). There exists an orthonormal basis \( \{\mathbf{u}_j(\omega)\}_{j=1}^m \) of \( \mathbb{C}^m(0,\omega) \) such that for any \( \{\mathbf{v}_{j,E+i\epsilon}(\omega)\}_{j=1}^m \in \mathbb{C}^m(1,\omega) \) that continuously depends on \( E+i\epsilon \) when \( \epsilon \neq 0 \), there exists a constant \( \gamma = \gamma(E,\omega,\mathbf{v}) > 0 \) satisfying the angular lower bound:
\begin{equation}\label{angle1}
\angle\left((\mathbf{u}_j(\omega), \mathbf{v}_{j,E+i\epsilon}(\omega)), E^s_{A_{E+i\epsilon}}(\omega)\right) = \gamma>0,
\end{equation}
for any \( 1 \leq j \leq m \) and \( \epsilon > 0 \).
Furthermore, $\gamma$ is continuously dependent on $E.$
\end{proposition}

\begin{proof} 
Suppose $E_0\in\mathcal{G}^s_\omega$, which means $\angle (E^s_{A_{E_0}}(\omega) , \mathcal{V}_\omega)=\gamma_1(E_0).$
Since $E^s_{A_E}$ is holomorphic on $E$ by Appendix \ref{ds-con}, we have $\gamma_1(E)$ is continuously dependent on $E$. In particularly, there exists a neighborhood of $E_0$, such that $\angle (E^s_{A_{E}}(\omega) , \mathcal{V}_\omega)=\gamma_1(E)>0$ for any $E$ close to $E_0.$

The restriction $ \pi_\omega|_{E^s_{A_E}(\omega)} $ is injective. Suppose otherwise: there would exist a non-zero vector $ \mathbf{v} = (0, \mathbf{u}) \in E^s_{A_E}(\omega) $, contradicting $ E \in \mathcal{G}^s_\omega $.

If $ \dim E^s_{A_E} = m $, it is well-known that $ E $ is not in the spectrum \cite{Puig}. Assume $ \dim E^s_{A_E} < m $ without loss of generality. Since $ \pi_{\omega}(E^{s}_{A_E}(\omega)) \subsetneq \mathbb{C}^{m} $, we can select $ \mathbf{u}_1(\omega) \notin \pi_{\omega}(E^{s}_{A_E}(\omega)) \cup (\pi_{\omega}(E^{s}_{A_E}(\omega)))^{\perp} $. Choosing $ \mathbf{u}_2(\omega) \in \left(\operatorname{span}\{\mathbf{u}_1(\omega)\}\right)^{\perp} $, it follows that $ \mathbf{u}_2(\omega) \notin \pi_{\omega}(E^{s}_{A_E}(\omega)) $ by the construction of $ \mathbf{u}_1(\omega) $. Assuming we have an orthonormal set $ \{\mathbf{u}_i(\omega)\}_{i=1}^{j} $, select $ \mathbf{u}_{j+1}(\omega) \in \left(\operatorname{span}\{\mathbf{u}_i(\omega)\}_{i=1}^{j}\right)^{\perp} $, which ensures $ \mathbf{u}_{j+1}(\omega) \notin \pi_{\omega}(E^{s}_{A_E}(\omega)) $ by the initial choice. Through induction, we obtain an orthonormal basis $ \{\mathbf{u}_i(\omega)\}_{i=1}^{m} $ for $ \mathbb{C}^{m}(0, \omega) $ which can be chosen to be continuously dependent on $E$, satisfying: \begin{equation}\label{angle-11}
\min_{1 \leq i \leq m} \angle\left(\mathbf{u}_i(\omega), \pi_{\omega}(E^{s}_{A_E}(\omega))\right)= \tilde{\gamma}_1(E) > 0,
\end{equation} where $ \tilde{\gamma}_1(E) $ depends only on the geometry of $ \pi_\omega(E^s_{A_{E}}(\omega)) $, hence it is continuously dependent on $E$.

Suppose the angle condition fails: there exist $ \epsilon_n \to 0 $, $1\leq j\leq m $ and $E$ in the neighborhood of  $E_0$ such that \begin{equation*}
\lim_{n\to\infty} \angle\left( (\mathbf{u}_j(\omega), \mathbf{v}_{j,E+i\epsilon_n}(\omega)), E^s_{A_{E+i\epsilon_n}}(\omega) \right) = 0,
\end{equation*} which implies that any limit point $ \frac{ (\mathbf{u}_j(\omega), \mathbf{v}_{j,E+i\epsilon_n}(\omega))}{\| (\mathbf{u}_j(\omega), \mathbf{v}_{j,E+i\epsilon_n}(\omega))\|} $ is in $ E^s_{A_E}(\omega) $. Assume $ \mathbf{v}_{j,E+i\epsilon_n}(\omega) \rightarrow b \in \mathbb{C}^m \cup \{\infty\} $. If $ b \neq \infty $, we have $ (\mathbf{u}_j(\omega), b) \in E^s_{A_E}(\omega) $, which contradicts \eqref{angle-11}. Otherwise, normalizing the vectors, we obtain a limit point in $ E^s_{A_E}(\omega) $ of the form $ (0, a) $: \begin{equation*}
\frac{ (\mathbf{u}_j(\omega), \mathbf{v}_{j,E+i\epsilon_n}(\omega))}{\| (\mathbf{u}_j(\omega), \mathbf{v}_{j,E+i\epsilon_n}(\omega))\|} \xrightarrow[n\to\infty]{} (0, a) \in E^s_{A_E}(\omega).
\end{equation*} Therefore $ E^s_{A_E} \cap \mathbb{C}^m(0, \omega) \neq \{0\} $, contradicting the assumption that $ E \in \mathcal{G}^s_{\omega} $. Thus, the uniform angle bound $ \gamma(E)> 0 $ must hold for sufficiently small $ \epsilon $ and continuously depends on $E$. \end{proof}

If furthermore these vector pairs belong to $ \mathcal{E}^s_{A_{E+i\epsilon}}(\omega) $, then we have the following:

\begin{lemma}\label{lem: ang imp slow1}
Suppose that the uniform angle condition \eqref{angle1} holds, and assume furthermore
\begin{equation*}
\left(\mathbf{u}_j(\omega), \mathbf{v}_{j,E+i\epsilon}(\omega)\right) \in \mathcal{E}^s_{A_{E+i\epsilon}}(\omega).
\end{equation*}
There exists \( {C}_1 = {C}_1(E,\omega,\mathbf{v}) > 0 \) such that for all \( k \in \mathbb{Z} \):
\begin{equation*}
\|(A_{E+i\epsilon})_k(\omega)(\mathbf{u}_j(\omega), \mathbf{v}_{j,E+i\epsilon}(\omega))^{T}\|
\geq {C}_1\|((A_{E+i\epsilon})_k(\omega)|_{E^c_{A_{E+i\epsilon}}})^{-1}\|^{-1}\|(\mathbf{u}_j(\omega), \mathbf{v}_{j,E+i\epsilon}(\omega))^{T}\|.
\end{equation*}
Furthermore, $C_1$ is continuously dependent on $E.$
\end{lemma}

\begin{proof} By definition, the vector $ (\mathbf{u}_j(\omega), \mathbf{v}_{j,E+i\epsilon}(\omega))^{T} \in E^s_{A_{E+i\epsilon}}(\omega) \oplus E^c_{A_{E+i\epsilon}}(\omega) $ can be decomposed into central and stable components: $$
(\mathbf{u}_j(\omega), \mathbf{v}_{j,E+i\epsilon}(\omega))^{T} = \mathbf{u}^c_{E+i\epsilon} + \mathbf{u}^s_{E+i\epsilon},
$$ where $ \mathbf{u}^c_{E+i\epsilon} \in E^c_{A_{E+i\epsilon}}(\omega) $ and $ \mathbf{u}^s_{E+i\epsilon} \in E^s_{A_{E+i\epsilon}}(\omega) $. From the uniform angle condition \eqref{angle1}, it follows that $$
\|\mathbf{u}^c_{E+i\epsilon}\| \geq \tilde{C}_1 \left\| (\mathbf{u}_j(\omega), \mathbf{v}_{j,E+i\epsilon}(\omega))^{T} \right\|.
$$ And $\tilde{C}_1$ is continuously dependent on $E$ since the angle $\gamma(E)$ in \eqref{angle1} is continuous dependent on $E$. The evolution of the central component under the cocycle restriction is given by: $$
\mathbf{u}^c_{E+i\epsilon} = (A_{E+i\epsilon})_k(\omega)|_{E^c_{A_{E+i\epsilon}}}^{-1} \cdot (A_{E+i\epsilon})_k(\omega)\mathbf{u}^c_{E+i\epsilon},
$$ which implies $$
\|\mathbf{u}^c_{E+i\epsilon}\| \leq \| (A_{E+i\epsilon})_k|_{E^c_{A_{E+i\epsilon}}}^{-1} \| \cdot \| (A_{E+i\epsilon})_k\mathbf{u}^c_{E+i\epsilon} \|.
$$

On the other hand, the partial hyperbolicity of $ (T, A_{E}) $ implies that for small $ \epsilon $ and large $ k $, $$
\| (A_{E+i\epsilon})_k\mathbf{u}^c_{E+i\epsilon} \| \geq 2 \| (A_{E+i\epsilon})_k\mathbf{u}^s_{E+i\epsilon} \|,
$$ thus we have $$
\| (A_{E+i\epsilon})_k\mathbf{u}^c_{E+i\epsilon} \| \leq \tilde{C} \|(A_{E+i\epsilon})_k(\omega)(\mathbf{u}_j(\omega), \mathbf{v}_{j,E+i\epsilon}(\omega))^{T}\|,
$$ which yields the required propagation estimate. \end{proof}

\subsection{Non-stationary telescoping-type inequality}\label{telescoping}
Another ingredient is  Proposition \ref{prop: variation center}, which provides control of $\|((A_{E+i\epsilon})_k(\omega)|_{E^c_{A_{E+i\epsilon}}})^{-1}\|$. 
 In the spirit of \cite{Simon}, we develop a novel non-stationary framework, where the crucial distinction lies in the domain of cocycles is not fixed. First, we need the following elementary but important observation:

\begin{proposition}\label{lem: est reverse norm}
Let $W$ be a compact subset of $G_{\mathrm{HSp}}(2k,2m)$ for $1 \leq k \leq m$. There exists $c = c(W) \geq 1$ such that  for any Hermitian symplectic matrix $A$, and $V \in W$ with $AV \in W$, we have 
\begin{equation*}\label{eqn: est reverse norm}
c^{-1} \|A^{-1}|_{AV}\| \leq \|A|_V\| \leq c \|A^{-1}|_{AV}\|.
\end{equation*}
\end{proposition}

\begin{proof}

We first need to construct a \textbf{continuous local symplectic basis} \((\xi_i(V))_{i=1}^{2k}\) near \( V \) . To prove this, we need the following continuous version of Sylvester Inertia Theorem: 
\begin{lemma}\label{ht}
   Suppose $G(\cdot):G(k,d)\rightarrow \mathrm{GL}(m, \mathbb{C})\cap \mathrm{Her}(m,\C)$ is a continuous function on the neighborhood $U_V$ of $V\in G(k,d)$. Then, there exists $p\geq 0$, a neighborhood $\tilde{U}_V$ of $V$ and continuous function $N(\cdot): \tilde{U}_V\rightarrow \mathrm{GL}(m, \mathbb{C}) $,  such that $$N(V)^*G(V)N(V)=\text{diag}(I_p, -I_{m-p}).$$
\end{lemma}
\begin{proof}
    Since $G(V)$ is Hermitian, it admits continuous eigenvalues \(\lambda_i(V)\) (\(1 \leq i \leq m\)) in a neighborhood of \( V \). By possibly shrinking the neighborhood, we may assume \(\lambda_i(V) > 0\) for \(1 \leq i \leq p\) and \(\lambda_i(V) < 0\) for \(p+1 \leq i \leq m\).

Let $\Gamma_1$ be the circle that encloses all positive eigenvalues, while $\Gamma_2$ is the circle that encloses all negative eigenvalues. Define
			$$
			P_1(V) = \frac{1}{2\pi i}\int_{\Gamma_1}(zI-G(V))^{-1} dz, \quad P_2(V) = \frac{1}{2\pi i}\int_{\Gamma_2}(zI-G(V))^{-1} dz.
			$$
			Then, $P_1(V)$ and $P_2(V)$ are continuous projection operators.

			Define
			$$
			Q_1(V) = \text{Range}(P_1(V)) \quad \text{and} \quad Q_2(V) = \text{Range}(P_2(V)),
			$$
			which correspond to continuous $p$-dimensional and $(m-p)$-dimensional invariant subspaces, respectively.
		By Lemma \ref{holomorphic}, by possibly shrinking the neighborhood, there exist $\{q^1_i(V)\}_{i=1}^p$ be a continuous basis for $Q_1(V)$, and $\{q^2_i(V)\}_{i=1}^{m-p}$ be a continuous basis for $Q_2(V)$. It is straightforward to verify that $(u,v)_G := \langle u,G(V)v \rangle$ defines an inner product on $Q_1(V)$, while $(u,v)_G := -\langle u,G(V)v \rangle$ defines an inner product on $Q_2(V)$.
			
			Applying the standard Gram–Schmidt process, we can obtain continuous bases $\{\tilde{q}^1_i(V)\}_{i=1}^p$ and $\{\tilde{q}^2_i(V)\}_{i=1}^{m-p}$ such that
			
			$$
			\big((\tilde{q}^1_i(V),\tilde{q}^1_j(V))_G\big) = I_p \quad \text{and} \quad \big((\tilde{q}^2_i(V),\tilde{q}^2_j(V))_G\big) = I_{m-p}.
			$$
			Let
			$$
			N(V) = \begin{pmatrix}
				\tilde{q}^1_1(V), \cdots, \tilde{q}^1_p(V), \tilde{q}^2_1(V), \cdots, \tilde{q}^2_{m-p}(V)
			\end{pmatrix}.
			$$
			Since $Q_1(V)$ and $Q_2(V)$ are orthonormal, we have
			$$
			N(V)^* G(V) N(V) = \text{diag}(I_p, -I_{m-p}).
			$$
    
\end{proof}

By Theorem \ref{holomorphic}, we choose a continuous local basis \((g_i(V))_{i=1}^{2k}\) near \( V \). Define the \textit{Krein matrix} 
\[
G(V) = i \big( \psi(g_i, g_j) \big)_{1 \leq i,j \leq 2k} \in \mathrm{GL}(2k, \mathbb{C})\cap \mathrm{Her}(2k,\C).
\]
By Lemma \ref{ht}, there exists $p\geq 0$, a neighborhood $\tilde{U}_V$ of $V$ and continuous function $N(\cdot): \tilde{U}_V\rightarrow \mathrm{GL}(2k, \mathbb{C}) $,  such that $N(V)^*G(V)N(V)=\text{diag}(I_p, -I_{2k-p}).$
Let \(M \in \mathrm{GL}(2k, \mathbb{C})\) satisfy
\[
M^* \mathrm{diag}(I_p, -I_{2k-p}) M = \begin{pmatrix}
0 & -iI_p \\
iI_{2k-p} & 0
\end{pmatrix}.
\]
Define the transformed basis vectors:
\[
\big( \xi_1(V), \ldots, \xi_{2k}(V) \big) = \big( g_1(V), \ldots, g_{2k}(V) \big) N(V) M.
\]
Then \(\{\xi_i(V)\}_{i=1}^{2k}\) forms a basis for \( V \), and the symplectic form \(\psi\) satisfies:
\[
\psi(\xi_i(V), \xi_j(V)) = \begin{cases}
1, & j - i = 2k - p, \\
-1, & i - j = p, \\
0, & \text{otherwise}.
\end{cases}
\]

 Given that $ A $ is a symplectic map, the set $\{A\xi_i\}_{i=1}^{2k}$ also forms a  basis.
For any vector $ v \in V $, we can express it as $ v = \sum_{i=1}^{2k} v_i \xi_i $, and consequently, $ Av = \sum_{i=1}^{2k} v_i A\xi_i $. It follows that:
\begin{equation*}\label{sbasis}
    \psi(Av, A\xi_j) = (Av)^* S A\xi_j = \sum_{i=1}^{2k} v_i \psi(A\xi_i, A\xi_j) = \begin{cases}
	-v_{j+p}, & 1 \leq j \leq 2k-p, \\
	v_{j-2k+p}, & 2k-p+1 \leq j \leq 2k.
\end{cases}
\end{equation*}

Thus, when $\|Av\| = 1$, we can conclude that:
$$
|v_{j}| \leq \begin{cases}
	\|S\| \|A|_V\| \|\xi_{j+2k-p}\| , & 1 \leq j \leq p, \\
	\|S\| \|A|_V\| \|\xi_{j-p}\|, & p+1 \leq j \leq 2k.
\end{cases} 
$$
It follows that 
$
\|v\| \leq \|S\| \|A|_V\| \left( \sum_{i=1}^{p} \|\xi_{i+2k-p}\| \|\xi_i\| + \sum_{i=p+1}^{2k} \|\xi_{i-p}\| \|\xi_i\| \right).
$
Therefore, we have:
$$
\|A^{-1}|_{AV}\| = \sup_{\|Av\|=1, v \in V} \|v\| \leq \|S\|\left(   \sum_{i=1}^{p} \|\xi_{i+2k-p}\| \|\xi_i\| + \sum_{i=p+1}^{2k} \|\xi_{i-p}\| \|\xi_i\|  \right)  \|A|_V\|  := c(V) \|A|_V\|.
$$
Clearly $c(V)$ defined above is a continuous function on $ V $ since $\{\xi_i(V)\}_{i=1}^{2k}$ depends continuously on $ V $. We can attain the maximal $c(V)$ in the small neighborhood of $V$. Given that $ W $ is compact, by the standard compact argument, we have $\|A^{-1}|_{AV}\| \leq c(W) \|A|_V\|$. The reverse inequality can be established by considering $ A^{-1} $ instead of $ A $.
\end{proof}

\begin{rem}
\begin{enumerate}
    \item The proof was motivated by the Analytic Sylvester Inertia Theorem \cite[Theorem 1.3]{WXZ}. It is important to note that, in general, one cannot expect to find a canonical basis\footnote{A canonical basis is defined as a set $\{v_1, \cdots, v_k, v_{-1}, \cdots, v_{-k}\} \subset \mathbb{C}^{2m}$ that spans the subspace and satisfies $\psi(v_i, v_{-j}) = \delta_{ij}$ for all $i, j > 0$.} for Hermitian symplectic subspaces of $\mathbb{C}^{2m}$, as demonstrated by Harmer \cite{Ha}. 
\item While Lemma \ref{ht} is of local nature, if 
$Q_1,Q_2$  correspond to trivial bundles, then one obtains the global Sylvester Inertia Theorem, as stated in Lemma \ref{hetong}.
\end{enumerate}
\end{rem}

Once we have Proposition \ref{lem: est reverse norm}, we establish the following telescoping-type inequality:

\begin{lemma}\label{lem: key esti}
Let $V_n \in W \subset G_{HSp}(2k, 2m)$, where $n \geq 1$, $1 \leq k \leq m$, and $W$ is a compact set. There exists a constant $c > 0$, which depends only on $W$, and satisfies the following properties: if there is a family of linear transformations $f_t(n): V_n \to V_{n+1}$ for $t \in [0, 1]$ and $n \in \mathbb{Z}^+$ such that for all $n \geq 1$,
\begin{enumerate}
    \item $f_0(n)$ maps $V_n$ to $V_{n+1}$ and preserves the Hermitian symplectic structure introduced by \eqref{syms}.
\item 
There exists $L > 0$ such that for any $t, s \in [0, 1]$, $$\|f_t(n) - f_s(n)|_{V_n}\| \leq L|t - s|, \quad \|f_t^{-1}(n) - f_s^{-1}(n)|_{V_{n+1}}\| \leq L|t - s|.$$
\end{enumerate}
Then for any $t \in [0, 1]$ and $n \geq 1$,
\begin{equation}\label{eqn: key variation}
\|f_t(n) \cdots f_t(1)|_{V_1}\|, \quad \|(f_t(n) \cdots f_t(1)|_{V_1})^{-1}\| \leq cC(n)\exp(cC(n)Ltn),
\end{equation}where $C(n):=\max(\max_{1\leq s\leq n}\|f_0(s)\cdots f_0(1)|_{V_1}\|^2,1)$, $C(0):=1$.
\end{lemma}
\begin{proof}
First by Proposition \ref{lem: est reverse norm}, there exists a constant $c_1\geq 1$ only depends on $W$ such that  for any $1\leq j\leq n$, 
\begin{equation}\label{eqn: f_0 seg est}
\|f_0(n)\cdots f_0(j)|_{V_j}\|\leq c_1 \|f_0(n)\cdots f_0(1)|_{V_1}\|\cdot \|(f_0(j-1)\cdot f_0(1)|_{V_1})\|\leq  c_1 C(n).   \end{equation}
Here we use $C(j)\leq C(n)$ for $j\leq n$. 

We now prove \eqref{eqn: key variation} by induction, focusing on the inequality for $\|f_t(n) \cdots f_t(1)\|$ for simplicity. Let $c \geq c_1$. Clearly, the inequality holds for $n = 0$. Assume it holds for all $j \in [0, n)$. For $n$, we consider the following elementary estimate:
$$\begin{aligned}
    &\quad\| f_t(n)\cdots f_t(1)|_{V_1}\|\\ 
 &\leq \|f_0(n)\cdots f_0(1)|_{V_1}\|+\sum_{j=0}^{n-1}\|f_0(n)\cdots f_0(j+2)f_t(j+1)\cdots f_t(1)|_{V_1}-f_0(n)\cdots f_0(j+1)f_t(j)\cdots f_t(1)|_{V_1}\|\\
 &\leq \|f_0(n)\cdots f_0(1)|_{V_1}\|+\sum_{j=0}^{n-1}\|f_0(n)\cdots f_0(j+2)|_{V_{j+2}}\|\|f_t(j+1)-f_0(j+1)|_{V_{j+1}}\|\|f_t(j)\cdots f_t(1)|_{V_1}\|.\\
\end{aligned}
$$

By combining the uniform Lipschitz estimate for $f_t(n)$, $t \in I$, the definition of $C(n)$, \eqref{eqn: f_0 seg est}, and the induction hypothesis, we obtain:
\begin{eqnarray*}
\| f_t(n) \cdots f_t(1)|_{V_1}\| &\leq& c_1 C(n) + \sum_{j=0}^{n-1} c_1 C(n) \cdot (Lt) \cdot c C(j) e^{c C(j) Lt j} \\
&\leq& c C(n) \left(1 + (c C(n) \cdot Lt) \sum_{j=0}^{n-1} e^{c C(n) Lt \cdot j} \right) \\
&=& c C(n) \left(1 + (c C(n) \cdot Lt) \frac{e^{c C(n) Lt n} - 1}{e^{c C(n) Lt} - 1}\right) \\
&\leq& c C(n) e^{c C(n) Lt n}.
\end{eqnarray*}
Thus, the induction is complete.

By a similar argument (reverse the order for telescoping steps), we could show the corresponding inequality for $\|(f_t(n)\cdots f_t(1)|_{V_1})^{-1}\|$
(If necessary we could take $c\geq c_2$, where $c_2$ only depends on $W$ and satisfies that for any $j\leq n$, $\|f_0(j)^{-1}\cdots f_0(n)^{-1}|_{V_n}\|\leq c_2C(n)$).
\end{proof}

\begin{prop}\label{prop: variation center}Suppose that the  Schr\"odinger cocycle  $(T, A_E)$ is partially hyperbolic for some $E\in \RR$. Then there exists constant $C_2=C_2(E)>1$ 
such that for any $\epsilon\in \RR$ sufficiently close to $0$, any $\omega\in \Omega$, any $n\in \ZZ^+$ we have 
\begin{equation*}\label{eqn: est center variation}
\| (A_{E+i\epsilon})_{ n} (\omega)|_{E^c_{A_{E+i\epsilon}}}\|, \| (A_{E+i\epsilon})_{ n} (\omega)|_{E^c_{A_{E+i\epsilon}}})^{-1}\|  \leq  C_1C(n)\exp (C_1C(n)\epsilon n),
\end{equation*}
 where  $C(n):=\max_{0\leq s\leq n}\|(A_E)_{s}(\omega)|_{E^c_{A_E}}\|^2$. Furthermore, $C_2$ is continuously dependent  on $E.$
\end{prop}
\begin{proof}
By classical cone arguments in hyperbolic dynamics, the family $E^c_{A_{E+i\epsilon}}(\omega)$ continuously depends on $(E+i\epsilon,\omega)$ and holomorphically depends on $E+i\epsilon$ (refer to Appendix \ref{ds-con} for the whole proof). For sufficiently small $\epsilon$, the linear projection $P_{E+i\epsilon}(\omega)$ from $E^c_{A_{E+i\epsilon}}(\omega)$ to $E^c_{A_E}(\omega)$ along $E^{u}_{A_E}(\omega)\oplus E^s_{A_E}(\omega)$ is well-defined, holomorphically dependent on $E+i\epsilon$, and continuously dependent on $(E+i\epsilon,\omega)$. Moreover, $P_E(\omega)$ is the identity for all $\omega$.

Let $B_{E+i\epsilon}:=(P_{E+i\epsilon})^{-1}$. Then the map $$
\tilde A_{E+i\epsilon}(\omega): E^c_{A_E}(\omega)\to E^c_{A_E}(T\omega)
$$ defined by $$
\tilde A_{E+i\epsilon}: \omega\mapsto (B_{E+i\epsilon})^{-1}(T\omega)\circ A_{E+i\epsilon}(\omega)\circ B_{E+i\epsilon}(\omega)
$$ can be viewed as a one-parameter family of cocycles that leaves $E^c_{A_E}$ invariant and holomorphically depends on $E+i\epsilon$. Additionally, we have $\tilde A_E=A_E$. The key observation is the following:
\begin{lemma}\label{lip1}
For any \(\omega\in\Omega\), \(\|\tilde A_{E+i\epsilon}(\omega)-\tilde A_{E}(\omega)|_{E^c_{A_E}(\omega)}\|\leq C_2\epsilon\), where \(C_3\) is a constant dependent only on \(E\). Futhermore, $C_3$ is continuous dependent on $E.$
\end{lemma}
\begin{proof}
By the mean-value theorem and Cauchy's integral formula, we have
$$
 	\begin{aligned}
 		\|\tilde A_{E+i\epsilon}(\omega)-\tilde A_{E}(\omega)|_{E^c_{A_E}(\omega)}\|&\leq \epsilon \sup_{0\leq\tilde\epsilon\leq\epsilon} \|d_{\tilde\epsilon} \tilde A_{E+i\tilde\epsilon}(\omega)\|\leq \epsilon \sup_{0\leq\tilde\epsilon\leq\epsilon} \|\frac{1}{2\pi i}\int_{\gamma}\frac{ \tilde A_{z}(\omega)}{(z-i\tilde\epsilon)^2}dz
 		 \|\\&\leq \epsilon C\sup_{z,\omega}  \|\tilde A_{z}(\omega)\|\leq C_3\epsilon ,
 	\end{aligned}
 $$
 where  $\gamma=\{|z-i\tilde\epsilon|=r\}$ with sufficiently small $r$, and  we used  $\tilde A_{E+i\epsilon}(\omega)$ is a continuous function on $(E+i\epsilon,\omega)$ and $\Omega$ is a compact set. 
 \end{proof}
 
By Lemma \ref{wron11}, $E^c_{A_E}(\omega)$ is a Hermitian-symplectic subspace for any $\omega\in\Omega$. We assume that $\dim E^c_{A_E}(\omega)=2k$ for some $1\leq k\leq m$. Since $E^c_{A_E}(\omega)$ continuously depends on $(\omega,E)\in\Omega\times \I$, where $\I$ is a small closed interval of $E$. Since $\I$ and $\Omega$ are compact sets, the sequence of subspaces $W=\{E^c_{A_E}(\omega)\}_{\omega\in\Omega,E\in \I}\subset G_{HSp}(2k,2m)$ forms a compact family of symplectic subspaces in $\mathbb{C}^{2m}$.

We apply Lemma \ref{lem: key esti} to the sequence of one-parameter linear maps $$
\tilde A_{E+i\epsilon}(T^{n-1}(\omega)): E^c_{A_E}(T^{n-1}(\omega))\to E^c_{A_E}(T^{n}(\omega)),
$$ where $\epsilon$ is $t$, and $E^c_{A_E}(T^{n-1}(\omega))$ is $V_n$ in Lemma \ref{lem: key esti}. By Lemma \ref{lip1}, we obtain a uniform Lipschitz estimate for the family $\epsilon\to \tilde A_{E+i\epsilon}|_{E^c_{A_E}}$.

Hence, we can apply Lemma \ref{lem: key esti} and obtain telescoping estimates of $$
\|\tilde A_{E+i\epsilon}(T^{n-1}(\omega))\cdots \tilde A_{E+i\epsilon}(\omega)|_{E^c_{A_E}}\|,~~ \|(\tilde A_{E+i\epsilon}(T^{n-1}(\omega))\cdots \tilde A_{E+i\epsilon}(\omega)|_{E^c_{A_E}})^{-1}\|.
$$ Since $\tilde A_{E+i\epsilon}$ is conjugate to $A_{E+i\epsilon}$ through $B_{E+i\epsilon}$ and $B_{E+i\epsilon}$ is the identity when $\epsilon=0$, we conclude that $B_{E+i\epsilon}$ is uniformly close to the identity when $\epsilon$ is sufficiently small. Therefore, the estimates of $$
\|\tilde A_{E+i\epsilon}(T^{n-1}(\omega))\cdots \tilde A_{E+i\epsilon}(\omega)|_{E^c_{A_E}}\|,~~ \|(\tilde A_{E+i\epsilon}(T^{n-1}(\omega))\cdots \tilde A_{E+i\epsilon}(\omega)|_{E^c_{A_E}})^{-1}\|
$$ imply the corresponding estimates of $$
\|(A_{E+i\epsilon})_{n}(\omega)|_{E^c_{A_{E+i\epsilon}}}\|, \|(A_{E+i\epsilon})_{n}(\omega)|_{E^c_{A_{E+i\epsilon}}})^{-1}\|
$$ (up to a constant close to 1), which completes the proof of Proposition \ref{prop: variation center}.
\end{proof}

\section{Monotonic Hermitian symplectic cocycles}\label{idea2}

 \subsection{Monotonicity of Hermitian symplectic bundles}       
We first introduce the concept of monotonicity for a curve taking values in the set of Hermitian symplectic mappings. 
        
\begin{definition}\label{def: mono HSp mapping}
Let $(V,\psi)$, $(V', \psi')$ be $2d$-dimensional Hermitian symplectic spaces with signature $0$.
\begin{enumerate}
     \item  A $C^l (l=1,\cdots,\infty,\omega)$ path $\gamma: I \rightarrow\left(V\backslash\{0\}, \psi\right)$ is called monotonic increasing (or decreasing), if $\psi\left(\gamma(t), \gamma^{\prime}(t)\right)$ is always positive (or negative) for all $t \in I$.
     \item A $C^l$ one parameter Hermitian symplectic mappings $$A_t: (V,\psi)\to (V',\psi'), t \in I, \psi'(A_t(v),A_t(w))=\psi(v,w)\text{ for any }v,w \in V,$$ is called monotonic, if for any \textit{isotropic vector} $v \in V$, the curve $A_t(v)$ is monotonic in $(V',\psi')$.
 \end{enumerate}
\end{definition}
Similarly, we can extend the definition of monotonicity to general Hermitian symplectic cocycles:
\begin{definition}\label{def: mono HSp cccle}
Let $\mathcal V=\{V_\omega, \psi_\omega\}_{\omega\in \Omega}$ be a continuous complex vector bundle (could be non-trivial) over a compact metric space $\Omega$ such that on each fiber $V_\omega$ we equip a Hermitian symplectic structure $\psi_\omega$ (continuous in $\omega$). Let $T:\Omega\to \Omega$ be a continuous map and $I\subset \R$ be an interval. 

A family of Hermitian symplectic cocycles $A_t:\V\to \V, t\in I$ over $T:\Omega\to \Omega$ such that 
\begin{itemize}
    \item $\partial_t A_t$ exists.
    \item $A_t, \partial_t A_t$ are continuous in $(t,\omega)\in I\times \Omega$.
\end{itemize}
 is called monotonic if for any $t\in I, \omega\in \Omega$, $A_t: (V_\omega,\psi_\omega)\to (V_{T\omega},\psi_{T\omega})$ is a $C^1$ monotonic family of Hermitian symplectic mappings in the sense of Definition \ref{def: mono HSp mapping} (2).

\end{definition}

Our definition for monotonic cocycles coincides with the definition in \cite{AK, Xu} when the cocycle taking values in $\SL(2,\R)$, $\Sp(2d,\R)$.
\begin{definition}\cite{AK, Xu}\label{olddef}
\begin{enumerate}
\item Let $t\mapsto A_t\in C(X,\mathrm{SL(2,\R)})$ be a one-parameter family cocycles, $C^1$ in $t$. We say it is monotonic, 
if for any $t_0\in I$, $\frac{d}{dt}\arg(A_tv)|_{t=t_0}<0$ for any $v\in\R^2\backslash\{0\}.$
\item Let $t\mapsto A_t\in C(X,\mathrm{Sp(2d,\R)}) $  be a one-parameter family cocycles, $C^1$ in $t$. We say it is monotonic, 
if for any $t_0\in I$, $$J\partial_t|_{t=t_0} A_t\cdot A_{t_0}^{-1} \text{ is positive definite}.$$
\end{enumerate}

\end{definition}

Notice that any real vector is actually isotropic, for the case that $\V$ is a trivial bundle with constant real-symplectic structure, our definition coincides with Definition \ref{olddef}.

\subsection{Local Trivialization  for Families of Hermitian Symplectic Forms on Complex Vector Bundles}

Next theorem concerns the local trivialization of a family of Hermitian symplectic forms on a complex vector bundle, it shows that  a family of fiberwise non-degenerate Hermitian symplectic forms can be "trivialized" to the initial form via a smooth family of bundle isomorphisms under mild derivative conditions. This result will play a crucial role in Theorem \ref{mon1}.

\begin{theorem}\label{thm: pure geo}
Let $\V \to X$ be a rank-$2d$ complex vector bundle over a smooth manifold $X$, and let $\omega(t)$ for $t \in I$ (a neighborhood of 0) be a family of fiberwise non-degenerate Hermitian symplectic forms. Assume:
\begin{enumerate}
    \item $\omega(t)$ depends continuously on $(x,t) \in X\times I$ and uniformly $C^l$ in $t$ ($l\geq 2$),
    \item $\omega(0) = \omega_0$ and $\omega'(0) = 0$ (i.e. for any $u,v$, $\frac{d}{dt}|_{t=0}  \omega(t)(u,v)=0$).
\end{enumerate}
Then there exists a family of fiber-preserving bundle isomorphisms $F_t: \V \to \V$ (covering $id_X$) depending continuously in $(x,t)\in X\times I$ and uniformly $C^{l-1}$ in $t$ such that:
\[
F_t^* \omega(t) = \omega_0, \quad F_0 = id, \quad \partial_t|_{t=0} F_t = 0.
\]
\end{theorem}
\begin{proof}
The proof uses the geometry of principal bundles \cite{SS}, specifically the existence of an invariant connection. This allows for the construction of canonical horizontal lifts. 
\\
\\
\textbf{Step 1 (Preparations using principal bundles): }

For our $\mathcal{V} \to X$, 
we consider its \textit{principal} $G=\GL(2d, \mathbb{C})$\textit{-bundle of frames}, denoted by $\mathcal{W}$. A point in $\mathcal{W}$ above $x \in X$ is a choice of frame for the fiber $\mathcal{V}_x$. The projection map is $\pi_X: \mathcal{W} \to X$.

For the initial Hermitian symplectic form $\omega_0|_x$ on each fiber $\mathcal{V}_x$, 
the group that preserves $\omega_0|_x$ is isomorphic to the single, abstract Lie group $H=\HSp(2d)$. 

We consider the associated bundle $\overline{\mathcal{W}} = \mathcal{W} / H$ with fibers $\GL(2d,\C)/H$. The fibers are precisely the set of all possible Hermitian symplectic forms (with the same signature as $\omega_0$) on $\V_x$. 
The family of Hermitian symplectic forms $\omega(t)$ corresponds to a family of continuous sections $\sigma(t): X \to \overline{\mathcal{W}}$ of the associated bundle $\overline{\W}$ which are uniformly $C^l$ in $t$ and stationary at $t=0$. 
\\
\\ \textbf{Step 2 (Reductive pair): } The pair $(G, H) = (\text{GL}(2d, \mathbb{C}), \text{HSp}(2d, \mathbb{C}))$ forms a \textit{reductive pair}, meaning the Lie algebra $\mathfrak{g} = \mathfrak{gl}(2d, \mathbb{C})$ admits a decomposition:
\[
\mathfrak{g} = \mathfrak{h} \oplus \mathfrak{m},
\]
where $\mathfrak{h} = \mathfrak{hsp}(2d, \mathbb{C})$ is the Lie algebra of $H$, and $\mathfrak{m}$ is a vector subspace such that $\text{Ad}(h) \mathfrak{m} \subseteq \mathfrak{m}$ for all $h \in H$. This property is fundamental for our next step, defining a canonical connection.
\\
\\ \textbf{Step 3 (Vertical, horizontal spaces and canonical projections):} We use the decomposition in Step 2 to define an $H$-equivariant principal connection on $\mathcal{W}_x$ for each $x$. For any $p \in \mathcal{W}_x$, we define the ``horizontal space" $\text{Hor}_p$ as the image of $\mathfrak{m}$ under the right translation $(R_g)_*$ (where $p = u \cdot g$ for some local frame $u$ and $g \in G$). More precisely, identifying the tangent space to the fiber at $p$ with $\mathfrak{g}$ (via right trivialization), $\text{Hor}_p$ is the part corresponding to $\mathfrak{m}$. The vertical space $\text{Vert}_p$ corresponds to $\mathfrak{h}$.
\begin{itemize}
    \item $\text{Vert}_p = T_p (p \cdot H) \cong \mathfrak{h}$.
    \item $\text{Hor}_p = \{ (R_g)_* X \mid X \in \mathfrak{m} \}$, where $p$ is associated with $g \in G$.
\end{itemize}
The $\text{Ad}(H)$-invariance of $\mathfrak{m}$ guarantees that the horizontal distribution is $H$-equivariant:
\[
(R_h)_* \text{Hor}_p = \text{Hor}_{p \cdot h}, \quad \forall h \in H.
\]
This $H$-equivariance is essential for ensuring consistency across the fibers of the quotient bundle. Apparently the quotient projection $\pi_x: \mathcal{W}_x \to \overline{\mathcal{W}}_x = \mathcal{W}_x/H_x$  that $d\pi_p$ maps $\text{Hor}_p$ isomorphically onto $T_{\pi(p)}\overline{\mathcal{W}}_x$. 
\\
\\\textbf{ Step 4 (Construct $F_t$ through principal connections): }
For each $x \in X$, we fix an arbitrary $p_0(x)\in \W_x$ such that it is a canonical frame for the Hermitian symplectic structure $\omega_0$. Then we could define a curve $\tilde F_t(x)$ uniquely that satisfies
\begin{enumerate}
\item  $\tilde F_0(x)=p_0(x)$;
\item $\pi_x( \tilde F_t(x))=\sigma(t;x)$;
\item  $\frac{d}{dt}\tilde F_t(x)\in \mathrm{Hor}_{\tilde F_t(x)}$. The curve  $\tilde F_t(x)$ is a horizontal lift of $\sigma(t;x)$. The existence and uniqueness of such a lift for a given initial condition is a standard result for principal connections compatible with the quotient structure.
\end{enumerate}

Now we can define our bundle maps $F_t$ as follows. For any $p\in \W_x$,  write $p=p_0(x)\cdot g$ for a unique $g\in G=\GL(2d,\C)$. Then define:
$$F_t(p):=\tilde F_t(x)\cdot g.$$
It is not hard to see that the definition of $F_t$ is independent of the choice of $p_0(x)$. Moreover this family of maps $F_t$ satisfies:
\begin{itemize}
\item $F_0(p)=p$: Since $\tilde F_0(x)=p_0(x)$, $F_0(p)=p_0(x)\cdot g=p$.
\item $F_t$ is a principal bundle automorphism: It covers the identity on $X$ and is $G-$equivariant just by its definition.
\item $\pi_x(F_t(p)) = \sigma(t;x)$ by definition of $\tilde F_t(x)$ and $F_t$.
\item Therefore $F_t$ induces a family of linear maps on $\V$: Due to $F_t$ being a principal bundle automorphism, it naturally induces a family of fiber-preserving linear isomorphisms $F_t:\V_x\to \V_x$.
\end{itemize}

\textbf{ Step 5 (Complete the proof):}
We have constructed a family of principal bundle automorphisms $F_t: \mathcal{W} \to \mathcal{W}$ defined by $F_t(p) := \tilde{F}_t(x) \cdot g$ for $p = p_0(x) \cdot g \in \mathcal{W}_x$, where $\tilde{F}_t(x)$ is the horizontal lift of $\sigma(t;x)$ starting from a canonical frame $p_0(x)$. This family $F_t$ induces a family of fiber-preserving linear isomorphisms $F_t: \mathcal{V} \to \mathcal{V}$. We now verify that this constructed $F_t$ satisfies the remaining required properties.
\begin{enumerate}
\item $F_t^* \omega(t) = \omega_0$: it is a corollary of $\pi_x(\tilde F_t(x))=\sigma(t,x)$.
\item $F_0 = id$: this follows that $ F_0(p)=p$.
\item $F_t'(0) = 0$ (the Stationary Condition): it is a corollary of the fact $\pi_x$ is a linear isomorphism (in Step 3). Here we use the definition of $\mathfrak m$.
\item By the way we define $F_t$ and classical regularity result for ODE, $F_t$  is continuous in $(x,t)\in X\times I$ and uniformly $C^{l-1}$ in $t$, (we solve $\tilde F_t$ using $\frac{d}{dt}\sigma(t,x)$ which is uniformly $C^{l-1}$ in $t$, so the solution of ODE also uniformly $C^{l-1}$ in $t$). 
\end{enumerate}
	\end{proof}

\subsection{Inheritance of monotonicity}

The following theorem is crucial for studying Hermitian symplectic cocycles that arise from the duals of Schrödinger cocycles.
Roughly speaking, if we have a monotonic family of cocycles that is also partially hyperbolic, we may deduce that the restriction of these cocycles to the center bundle remains monotonic—at least over some smaller parameter interval.
	\begin{theorem}\label{mon1}
Let $\V=\{V_\omega,\psi_\omega\}_{\omega\in \Omega}$ be a continuous vector bundle over a compact metric space $\Omega$, with a continuous family of Hermitian symplectic fibers $(\V_\omega,\psi_\omega)$.
Let $A_t: \V\to \V, t\in (-\epsilon,\epsilon)$ be a monotonic family of Hermitian symplectic cocycle  over $T$ (in the sense of Definition \ref{def: mono HSp cccle}) and uniformly $C^l, l=2,\dots,\infty, \omega$ in $t$. Morerover we assume $A_t, t\in (-\epsilon, \epsilon)$ is partially hyperbolic (with constant center dimension). 

Then there exists a small neighborhood $(-\epsilon',\epsilon')$ of $0$ and a family of Hermitian symplectic mappings $B_t: E_0^c\rightarrow E_t^c$ such that
\begin{enumerate}
    \item $B_t$ is uniformly $C^{l-1}$ in $t$;
    \item $\partial_tB_t, B_t$ are continuous in $(t,\omega)$;
    \item $C_t: E^c_0\to E^c_0, C_t|_{E^c_0(\omega)}:=B_t^{-1}(T\omega) \circ A_t|_{E_t^c(\omega)} \circ B_t(\omega)|_{E^c_0(\omega)}$ is a monotonic family of Hermitian cocycles on the bundles $E^c_0$ over $T$. Here $E^c_t$ is short for $E^c_{A_t}$.
\end{enumerate}

	\end{theorem}
	\begin{proof}
Without loss of generality, we may assume $\epsilon$ is sufficiently small so that $E_t^c(\omega)$ remains close to $E_0^c$, and $E_t^c(\omega)$ is uniformly $C^l$-dependent on $t$ and continuous in $(t,\omega)$ (Appendix \ref{ds-con}).  
 
Let $P_t(\omega)$ denote the linear projection from $E_t^c$ onto $E_0^c$ along $E_0^u \oplus E_0^s$. This projection is well-defined, uniformly $C^l$ in $t$, and both $P_t$ and $\partial_t P_t$ are continuous in $(t,\omega)$. We define $\B_t(\omega) = P_t(\omega)^{-1}$. Consequently, $E_t^c$ may be viewed as the graph of a linear map $\Phi_t: E_0^c \to E_0^s \oplus E_0^u$. Furthermore, $\B_t$ admits the explicit representation:  
\[
\B_t: E_0^c \to E_t^c,\quad v \mapsto v + \Phi_t(v).
\]  
Finally, we define $\cC_t: E_0^c \to E_0^c$ by the relation  
\[
\cC_t(\omega) = \B_t^{-1}(T\omega) \circ \left. A_t \right|_{E_t^c(\omega)} \circ \left. \B_t(\omega) \right|_{E_0^c(\omega)}.
\]
We then distinguish the proof into following two steps:\\

\textbf{Step 1:} We show that for any isotropic vector \( v \in E^c_0 \), 
\[
\psi\left( \B_0(v), \left. \frac{d}{dt} \right|_{t=0} \B_t(v) \right) > 0.
\]
   
Without loss of generality we assume $A_t$ is monotonic increasing. We consider the curve   
$\gamma(t)=\cC_t(v)=\B_t^{-1}(T\omega)$ $\left.\circ A_t\right|_{E_t^c} \circ \B_t(\omega)(v)$ for $t$ near $0$ for an arbitrary isotropic vector $v\in E^c_0(\omega)-\{0\}$. We claim that there exists $\epsilon'>0$ such that regardless the choice of $v$, 
\begin{equation*}
			\psi(\gamma(t), \gamma^{\prime}(t))>0\text { for all } t \in (-\epsilon',\epsilon').
\end{equation*}
In fact, by continuity of $P_t, \partial_t P_t$ and compacticity of $\Omega$ and the space of isotropic vectors with unit length (with respect to any fixed Hermitian metric), it suffices to show for any non-zero isotropic vector $v$, $
 \psi(\gamma(0), \gamma^{\prime}(0))>0$. In fact,

		$$
		\begin{aligned}
			& \psi\left(\gamma(0), \gamma^{\prime}(0)\right) \\
			= & \psi\left(A_0(v),\left.\left.\frac{d}{d t}\right|_{t=0} \B_t^{-1}(T\omega) \circ A_t\right|_{E_t^c} \circ \B_t(\omega) \cdot v\right) \\
			= & \psi\left(A_0(v),\left.\left.P_0 \circ \frac{d}{d t}\right|_{t=0} A_t\right|_{E_t^c} \circ \B_t(\omega) \cdot v\right)+  \psi\left(A_0(v),\left.  \frac{d}{d t}\right|_{t=0}P_t(T\omega)\circ A_0 \cdot v\right) \\
				= & \psi\left(A_0(v),\left.\left.P_0 \circ \frac{d}{d t}\right|_{t=0} A_t\right|_{E_t^c} \circ \B_t(\omega) \cdot v\right)+  \psi\left(A_0(v),\left.  \frac{d}{d t}\right|_{t=0}A_0 \cdot v\right) \text { by definition of } P_t  \\
			= & \psi\left(A_0(v),\left.\left.P_0 \circ \frac{d}{d t}\right|_{t=0} A_t \circ \B_t(\omega)\right|_{E_0^c} \cdot v\right) \text { by definition of } B_t \\
			= & \psi\left(A_0(v),\left.\left.\frac{d}{d t}\right|_{t=0} A_t \circ \B_t(\omega)\right|_{E_0^c} \cdot v\right) \text { using symplectic orthogonal property } \\
			= & \psi\left(A_0(v), A_0^{\prime}(v)+\left.A_0 \cdot \frac{d}{d t}\right|_{t=0}\left(\Phi_t \cdot v\right)\right) \text { by definition of } \Phi_t \\
			= & \psi\left(A_0(v), A_0^{\prime}(v)\right)+\psi\left(v,\left.\frac{d}{d t}\right|_{t=0}\left(\Phi_t \cdot v\right)\right) \text { by symplecticity of } A_0\\
			=& \psi\left(A_0(v), A_0^{\prime}(v)\right)>0.
		\end{aligned}
		$$
		
		For the fourth equality and the last equality we use $\Phi_t,\left.\frac{d}{d t}\right|_{t=0}\left(\Phi_t(v)\right) \in E_0^s \oplus E_0^u$ and $E_0^c$,  $E_0^s \oplus E_0^u$ are Hermitian symplectic orthogonal. Therefore we complete the step 1.\\

\textbf{Step 2:} Therefore, if \( \B_t \) is already a Hermitian symplectic transform from \( E^c_0 \) to \( E^c_t \), then \( \cC_t \) satisfies all assumptions of Theorem \ref{mon1} at least for \( t \) sufficiently close to \( 0 \). The problem is that we can only show \( \B_t \) is an \textit{asymptotic} Hermitian symplectic near \( 0 \) up to an error \( O(t^2) \). We will replace \( \B_t \) by a genuinely Hermitian symplectic \( B_t \) such that \( \B_t \) and \( B_t \) are the same up to an error \( O(t^2) \) near \( 0 \). The replacing procedure is completed by a differential geometric argument (Theorem \ref{thm: pure geo}).  
Since monotonicity only involves \( B_t \) and \( \frac{d}{dt} B_t \), our newly-defined \( B_t \) and associated \( C_t \) actually satisfy all assumptions of Theorem \ref{mon1} near \( 0 \).

Now we construct $B_t$ from $\B_t$. Our $\B_t$ in general fail to be Hermitian symplectic mappings. We denote the restriction of Hermitian symplectic form $\psi$ to $E^c_t$ by $\psi^c(t)$. Consider $(\B_t)^\ast(\psi^c(t))$, the pull back of $\psi^c(t)$ from $E^c_t$ to $E^c_0$ which we denote by $\omega(t)$, a one parameter family of Hermitian symplectic forms on $E^c_0$.

We plan to construct a family of fiber fixing linear mappings $F_t: E^c_0\to E^c_0$ such that $F_0=id, F'(0)=0, F(t)^{\ast}\omega(t)=\omega(0)$.
The claim of Theorem \ref{mon1} 
follows from Theorem \ref{thm: pure geo}, by taking $\tilde B_t=B_t\circ F_t$, then $\tilde B_t$ is monotonic in a neighborhood of $0$ and Hermitian symplectic for $t\in I$.

    \end{proof}

\section{Subordinacy theory for operator on strip}\label{subor}

Let us first review the foundational setup of spectral theory for operators on a strip. We begin by recalling the definition of the Weyl matrix.

\begin{lemma}[\cite{KS},\cite{Puig}]\label{M}
	For any $z\in\mathbb{C}\backslash\mathbb{R}$, there exist unique sequences of $d\times d$ matrix
	valued functions $\{\mathbf{F}^{\pm}_z(k)\}_{k\in\mathbb{Z}}$ that
	satisfy the following properties:
	\begin{enumerate}
		\item $\mathbf{F}^\pm_z(0)=I_d,$
		\item  $$
		C^*\mathbf{F}^{\pm}_z(k-1)+C\mathbf{F}^\pm_z(k+1)+V(k)\mathbf{F}^\pm_z(k)=z \mathbf{F}^\pm_z(k),
		$$
		\item
		$$
		\sum\limits_{k=0}^\infty\|\mathbf{F}^+_z(k)\|^2<\infty, \ \ \sum\limits_{k=-\infty}^{0}\|\mathbf{F}^-_z(k)\|^2<\infty.
		$$
	\end{enumerate}
\end{lemma}

Once we have $\mathbf{F}^{\pm}_z(k)$, we can define the $M$-matrices,
$$
M^+_z=-C\mathbf{F}^+_z(1), \qquad 
M^-_z=C\mathbf{F}^-_z(1).
$$
It is easy to check for any $z\in\mathbb{C}\backslash\mathbb{R}$, $\Im M_z^\pm$ is positive definite.
For any  $z\notin \Sigma(H)$,  and for $p,q\in \ZZ$, we can define the Green matrix: $$G_z(p,q):=\begin{pmatrix}
	\langle \vec{\delta}_{i,p},(H-z)^{-1}\vec{\delta}_{j,q}\rangle
\end{pmatrix}_{1\leq i,j\leq d},$$  where $\{\vec{\delta}_{in}; i=1,\ldots,d, n\in\mathbb{Z}\}$ denote the standard basis of $\ell^2(\mathbb{Z},\C^m)$.
For any $z\in\mathbb{C}^+$, it will be useful to define $2m\times2m$ matrices $M_z$ by
\begin{equation*}
  M_z=  \begin{pmatrix}
    G_z(0,0)&G_z(0,1)\\
    G_z(1,0)&G_z(1,1)
\end{pmatrix}.
\end{equation*}
Direct computation shows that:
\begin{lemma}\label{spm}\cite{KS}We have 
\begin{enumerate}
\item  $\tr\Im M_{E+i\epsilon} \geq \epsilon^{-1} \mu(E-\epsilon, E+\epsilon),$  where  $\mu$ is the canonical spectral measure. 
\item
\begin{equation*}
  M_z=\begin{pmatrix}
    -(M_z^++M_z^-)^{-1}&\left(M_z^{-} + M_z^{+}\right)^{-1} M_z^{+}(C^*)^{-1}\\
    C^{-1} M_z^{+} \left(M_z^{-} + M_z^{+}\right)^{-1}&C^{-1}M_z^+\big(M_z^-+M_z^+\big)^{-1}M_z^-(C^*)^{-1}
\end{pmatrix}.
\end{equation*}
\end{enumerate}
\end{lemma}

This means that to estimate the spectral measure, it is enough to estimate 
$\tr\Im M_{E+i\epsilon}.$

\begin{lemma}\label{est imM} We have the following expression:
\begin{eqnarray}\label{im m}
  \tr\Im M_{E+i\epsilon}&=&\tr C^{-1}Y(X+Y)^{-1} \Im (X)(X^\ast+Y^\ast)^{-1}Y^\ast(C^*)^{-1}\\\nonumber
  &&+\tr(X+Y)^{-1}\Im (X)(X^\ast+Y^\ast)^{-1}\\\nonumber
  &&+\tr C^{-1}X(X+Y)^{-1}\Im (Y)(X^\ast+Y^\ast)^{-1}X^\ast(C^*)^{-1}\\\nonumber
  &&+\tr(X+Y)^{-1}\Im (Y)(X^\ast+Y^\ast)^{-1},
 \end{eqnarray}where $X,Y=M^{\pm}_{E+i\epsilon}$.
Moreover, if we define the conformal coefficient $\kappa(N)$ for an invertible matrix $N$ by 
\begin{equation*}
\kappa(N):=\|N\|\cdot \|N^{-1}\|.
\end{equation*}
Then we have 
\begin{equation}\label{im M contr} \tr\Im M_{E+i\epsilon}
\leq (\max(\kappa(\Im X),\kappa(\Im Y))^3 \bigg((m+\tr C^{-1} XX^\ast(C^{-1})^*) \|\Im  X\|^{-1}+(m+\tr C^{-1} YY^\ast(C^{-1})^*) \|\Im  Y\|^{-1}\bigg).
\end{equation}
\end{lemma}

\begin{proof}
Here and in the after, for any $A\in{\rm M}(m,\C)$, we define
$\Re A:=\frac{A+A^*}{2}$,$\Im A:=\frac{A-A^*}{2i}$.
Direct calculation shows that
\begin{align*}
-\Im\left((X+Y)^{-1}\right) 
&= -\frac{1}{2i}\left[(X+Y)^{-1} - (X^*+Y^*)^{-1}\right] \\
&= (X+Y)^{-1}\left[\Im(X) + \Im(Y)\right](X+Y)^{*-1} \\
&= (X+Y)^{*-1}\left[\Im(X) + \Im(Y)\right](X+Y)^{-1}.
\end{align*}
Using the identity \( X(X+Y)^{-1}Y = (X^{-1} + Y^{-1})^{-1} \), we derive:
\begin{align*}
\Im\left(X(X+Y)^{-1}Y\right) 
&= \Im\left((X^{-1} + Y^{-1})^{-1}\right) \\
&= \frac{1}{2i}(X^{-1} + Y^{-1})^{-1}\left(X^{*-1} - X^{-1} + Y^{*-1} - Y^{-1}\right)(X^{*-1} + Y^{*-1})^{-1} \\
&= (X^{-1} + Y^{-1})^{-1} (X^{-1}\Im (X) X^{*-1}+ Y^{-1}\Im (Y) Y^{*-1})(X^{*-1} + Y^{*-1})^{-1}\\
&= Y(X+Y)^{-1}\Im(X)(X^*+Y^*)^{-1}Y^* \\
&\quad + X(X+Y)^{-1}\Im(Y)(X^*+Y^*)^{-1}X^*,
\end{align*}
where the third equality follows from the identity $\Im (X^{*-1})= X^{-1}\Im (X) X^{*-1}.$ Moreover, given the transformation property:
\begin{equation*}
\Im\left(C^{-1}X(X+Y)^{-1}Y(C^*)^{-1}\right) = C^{-1}\Im\left(X(X+Y)^{-1}Y\right)(C^*)^{-1},
\end{equation*}
the result \eqref{im m} follows directly from Lemma~\ref{spm}.

For two positive definite matrices \( M_1 \) and \( M_2 \), we write \( M_1 \succ (\succcurlyeq) M_2 \) if \( M_1 - M_2 \) is positive (semi)definite. Then we have the following:
\begin{lemma}\label{later use}
For a complex $m\times m$ matrix $Z$ such that $\Im  Z$ positive definite, we have $Z$ is invertible and 
\begin{eqnarray}\label{imx-1}
\|Z^{-1}\| &\leq &  \|(\Im  Z)^{-1}\|,\\
\label{imx}  \|(\Im Z)^{-1}\|^{-1} &\preccurlyeq  \Im Z &\preccurlyeq \|\Im(Z)\| .
\end{eqnarray}
\end{lemma}

\begin{proof}
For any $A\in{\rm Her}(m,\C)$, denote the eigenvalues of $A$ by $\lambda_1(A) \geq \lambda_2(A)\geq \cdots \geq  \lambda_m(A)$. For any $B\in{\rm M}(m,\C)$, denote the singular values of 
$B$ by $s_1(B) \geq s_2(B)\geq \cdots \geq  s_m(B)$,
by the result of Fan-Hoffman \cite[Remark 1]{Fan}, $\lambda_m(\mathrm{Re}A)\leq s_m(A)$. Given this, let $\tilde{Z}=-iZ$, then  $\mathrm{Re}(\tilde{Z})=\Im  Z$. By assumption, $\mathrm{Re}(\tilde{Z})$ is positive definite, so its smallest singular value coincides with its smallest eigenvalue: $s_m(\mathrm{Re}(\tilde{Z}))=\lambda_m(\mathrm{Re}(\tilde{Z}))\leq s_m(\tilde{Z})$. 
Using the identity $s_k(A)=1/s_{m-k+1}(A)$ for singular values of inverses, we conclude  $\|(\Im  Z)^{-1}\|\geq\|Z^{-1}\|.$

Since $\Im Z$ is positive definite, it follows that $\|(\Im Z)^{-1}\|^{-1} = \lambda_m(\Im(Z))$. Consequently, $ \|\Im(Z)^{-1}\|^{-1} \I = \lambda_m(\Im(Z)) \I \preccurlyeq   \Im Z $.
\end{proof}

For any matrix $ B \in \mathrm{M}(m, \mathbb{C}) $, if $ M_1 \preccurlyeq M_2 $, then it follows that $ B M_1 B^{*} \preccurlyeq B M_2 B^{*} $. This leads us to the following conclusion:
\begin{eqnarray*}
 && (X+Y)^{-1}\Im(X)(X^*+Y^*)^{-1} \\
&  \preccurlyeq& (X+Y)^{-1}(\Im X + \Im Y)\|(\Im Y)^{-1}\|^{-1}\|(\Im Y)^{-1}\|(X^*+Y^*)^{-1} 
  \quad \text{(by \( \Im X \preccurlyeq \Im X + \Im Y \))} \\
  & =&(X+Y)^{-1}(\Im X + \Im Y)^{1/2}\|(\Im Y)^{-1}\|^{-1}(\Im X + \Im Y)^{1/2}\|(\Im Y)^{-1}\|(X^*+Y^*)^{-1} \\
& \preccurlyeq & (X+Y)^{-1}(\Im X + \Im Y)^2\|(\Im Y)^{-1}\|(X^*+Y^*)^{-1} 
\quad \text{(by \eqref{imx})}\\
& \preccurlyeq &  \|\Im X + \Im Y\|^2\|(X+Y)^{-1}\|^2\|(\Im Y)^{-1}\|  \\
&  \preccurlyeq & \|\Im X + \Im Y\|^2\|(\Im X + \Im Y)^{-1}\|^2\|(\Im Y)^{-1}\| \quad \text{(by  \eqref{imx-1})} \\
&  =& \kappa(\Im X + \Im Y)^2\|(\Im Y)^{-1}\| 
 \quad \text{(where \( \kappa(A) := \|A\|\|A^{-1}\| \))} \\
&\preccurlyeq  & \kappa(\Im X + \Im Y)^2\kappa(\Im Y)\|\Im Y\|^{-1}.
\end{eqnarray*}

Meanwhile, it holds that $\kappa(\Im X + \Im Y) \leq \max(\kappa(\Im X), \kappa(\Im Y))$.  By applying \eqref{im m} and a similar inequality for $Y$, we establish the proof of \eqref{im M contr}. 
\end{proof}

Therefore, to prove Theorem \ref{thm:main-spectral}, the key is to estimate $\Im M_{E+i\epsilon}^{\pm}$. We will decompose the proof into several steps.

\subsection{Estimate for \texorpdfstring{$\Im M_{E+i\epsilon}$}{Im M}}\label{estimate m1}

First we estimate $\Im M_{E+i\epsilon}^{+}$:
\begin{lemma}\label{imm}
    Let $E\in \mathcal{G}_\omega^s$, then 
    \begin{equation}\label{eqn: criterion key ineq}
     \tr\Im M_{E+i\epsilon}^+(\omega)\geq 
	 C_5^{-1}(m+\tr C^{-1} M_{E+i\epsilon}^+(\omega)(M_{E+i\epsilon}^+(\omega))^\ast(C^{-1})^*),   
    \end{equation}
	 where $C_5:=C_4\max_{0\leq s\leq 3[\epsilon^{-1}]}\|(A_E)_{s}(\omega)|_{E^c_{A_E}}\|^6$, $C_4$ is a constant dependent only on $E,\omega$, and is continuously dependent on $E.$
\end{lemma}
\begin{proof}

For each basis vector $\mathbf{u}_j(\omega)$ given by Proposition \ref{lem: key ang est1}, consider the unique $\ell^2(\mathbb{Z}_+,\mathbb{C}^m)$ solution:
\begin{equation*}
	H_{V,T,\omega}\mathbf{u}_{E+i\epsilon}(n,\omega) = (E+i\epsilon)\cdot \mathbf{u}_{E+i\epsilon}(n,\omega)
\end{equation*}
satisfying the boundary condition  $\mathbf{u}_{E+i\epsilon}(0,\omega)=\mathbf{u}_j(\omega)$, such $\mathbf{u}_{E+i\epsilon}$ is unique and by the definition of $M^+_{E+i\epsilon}$, we have that $\mathbf{u}_{E+i\epsilon}(1,\omega)=-C^{-1}M_{E+i\epsilon}^+\mathbf{u}_j(\omega)$ which holomorphically depends on $E+i\epsilon$ when $\epsilon\neq 0$.

By the selection, 
$\left(\mathbf{u}_{E+i\epsilon}(0,\omega), \mathbf{u}_{E+i\epsilon}(1,\omega)\right)^{T} \in \E^s_{A_{E+i\epsilon}}(\omega),$
then by 
Lemma \ref{lem: ang imp slow1} and  Proposition \ref{prop: variation center}, we get that there exists $C_6=C_6(E,\omega,V)$ which is continuously dependent on $E$, such that 
\begin{eqnarray*}
&&\|(\mathbf{u}_{E+i\epsilon}(2k+1,\omega), \mathbf{u}_{E+i\epsilon}(2k+2,\omega))\|^2\\
&\geq& (C_5C(2k+1)\exp(C_5C(2k+1)\epsilon (2k+1))^{-2}\cdot \|(\mathbf{u}_{E+i\epsilon}(0,\omega), \mathbf{u}_{E+i\epsilon}(1,\omega))\|  \\
&\geq &(C_5C(2k+1)\exp(C_5C(2k+1)\epsilon (2k+1))^{-2} \cdot (1+\|C^{-1}M_{E+i\epsilon}^+(\omega)\cdot \mathbf{u}_j(\omega)\|^2).
\end{eqnarray*}
Hence using the increasing property of $C(n)$, we get for any $N$, 
\begin{eqnarray}
&&\langle\mathbf{u}_j(\omega), \Im M_{E+i\epsilon}^+(\omega)\cdot \mathbf{u}_j(\omega)\rangle \nonumber \\
&=&-\Im  \langle\mathbf{u}_{E+i\epsilon}(0,\omega),  C\mathbf{u}_{E+i\epsilon}(1,\omega)\rangle \nonumber \\
&=&\epsilon\cdot \sum_{k=0}^\infty \|(\mathbf{u}_{E+i\epsilon}(2k+1,\omega), \mathbf{u}_{E+i\epsilon}(2k+2,\omega))\|^2  \nonumber \\
&\geq &\epsilon\cdot \sum_{k=0}^N \|(\mathbf{u}_{E+i\epsilon}(2k+1,\omega), \mathbf{u}_{E+i\epsilon}(2k+2,\omega))\|^2 \nonumber\\
&\geq & \sum_{k=0}^N (C_5C(2N+1)\exp(C_5C(2N+1)\epsilon (2k+1))^{-2} \cdot (1+\|C^{-1}M_{E+i\epsilon}^+(\omega)\cdot \mathbf{u}_j(\omega)\|^2)\nonumber\\
&\geq & C_5^{-2}C(2N+1)^{-2}\cdot(\epsilon\exp(-2C_5C(2N+1)\epsilon)\cdot \frac{1-\exp(-4N C_5 C(2N+1)\epsilon)}{1-\exp(-4C_5C(2N+1)\epsilon)}\nonumber\\ && \cdot(1+\|C^{-1}M_{E+i\epsilon}^+(\omega)\cdot \mathbf{u}_j(\omega)\|^2). \label{eqn: est Im M M2 Simon}
\end{eqnarray}
 We know on $[0,1]$, $2xe^{-x}/(1-e^{-2x})$ has a lower bound $C>0$. Denote by $x(N):=2C_6C(2N+1)\epsilon$, then we have that the right hand side of \eqref{eqn: est Im M M2 Simon} is greater than 
\begin{eqnarray*}
\frac{1}{2}  C_6^{-3}C(2N+1)^{-3} x\frac{e^{-x}(1-e^{-2Nx})}{1-e^{-2x}} \cdot (1+\|C^{-1}M_{E+i\epsilon}^+(\omega)\cdot \mathbf{u}_j(\omega)\|^2) .
\end{eqnarray*}
Let $N_0=N_0(\epsilon)$ be the largest $N>0$ such that $x(N_0)\leq 1$, then when $N_0<\infty$, we have that $x(N_0)\geq x(N_0+1)\cdot \|A\|_{C^0}^{-4}\geq \|A\|_{C^0}^{-4}$. Moreover if we pick $\epsilon$ small enough (only depend on $\|A\|_{C^0}$), we can assume 
\begin{equation*}
N_0> \|A\|_{C^0}^4.
\end{equation*}
Then if $N_0<\infty$, we have 
\begin{equation}\label{eqn: est const}
 x(N_0)N_0\geq 1.
\end{equation}
If $N_0\leq \epsilon^{-1}$, take $N$ be $N_0$ in the right hand side of \eqref{eqn: est Im M M2 Simon}, by \eqref{eqn: est const} and $x(N)\leq 1$, we have that the right hand side of \eqref{eqn: est Im M M2 Simon} is greater than 
\begin{eqnarray}
&& \frac{1}{2}  C_6^{-3}C(2N_0+1)^{-3} \cdot (1+\|C^{-1}M_{E+i\epsilon}^+(\omega)\cdot \mathbf{u}_j(\omega)\|^2)\cdot x\frac{e^{-x}(1-e^{-2Nx})}{1-e^{-2x}} \nonumber\\
&\geq& \frac{1}{2}  C_6^{-3}C(2N_0+1)^{-3} \cdot (1+\|C^{-1}M_{E+i\epsilon}^+(\omega)\cdot \mathbf{u}_j(\omega)\|^2)\cdot 
\inf_{x\in[0,1]}\frac{xe^{-x}}{1-e^{-2x}}(1-e^{-2})\nonumber\\
&\geq &C_6 C(2N_0+1)^{-3}(1+\|C^{-1}M_{E+i\epsilon}^+(\omega)\cdot \mathbf{u}_j(\omega)\|^2)\ \nonumber\\
&\geq & C_6 C(3[\epsilon^{-1}])^{-3}(1+\|C^{-1}M_{E+i\epsilon}^+(\omega)\cdot \mathbf{u}_j(\omega)\|^2) .\label{eqn: last step simon est}
\end{eqnarray}
If $N_0>\epsilon^{-1}$, we take $N$ be $[\epsilon^{-1}]+1$, then $x(N)\leq 1$ and $x(N)N\geq 1$ (We can choose $C_6$ to be large, depending only on $E, \omega$, and continuously depending on $E$), we have that the right hand side of \eqref{eqn: est Im M M2 Simon} is greater than
$$\frac{1}{2}  C_6^{-3}C(2N+1)^{-3} \cdot (1+\|C^{-1}M_{E+i\epsilon}^+(\omega)\cdot \mathbf{u}_j(\omega)\|^2)\cdot 
\inf_{x\in[0,1]}\frac{xe^{-x}}{1-e^{-2x}}(1-e^{-2})\nonumber,$$
hence we have a similar estimate as \eqref{eqn: last step simon est}. In summary, we have 
\begin{equation}\label{eqn: est spec H+}
\langle\mathbf{u}_j(\omega), \Im M_{E+i\epsilon}^+(\omega)\cdot \mathbf{u}_j(\omega)\rangle \geq C_4 C(3[\epsilon^{-1}])^{-3}(1+\|C^{-1}M_{E+i\epsilon}^+(\omega)\cdot \mathbf{u}_j(\omega)\|^2).
\end{equation}

Let $\mathbf{u}_j$ run over the base we chose,  \eqref{eqn: est spec H+} implies the following inequality for positive definite matrices
\begin{equation*}
\begin{aligned}
	\tr\Im M_{E+i\epsilon}^+(\omega)&\geq  C_4C(3[\epsilon^{-1}])^3(m+\Tr(M_{E+i\epsilon}^+(\omega))^\ast(C^{-1})^*C^{-1} M_{E+i\epsilon}^+(\omega))\\
	&= C_4C(3[\epsilon^{-1}])^3(m+\tr C^{-1} M_{E+i\epsilon}^+(\omega)(M_{E+i\epsilon}^+(\omega))^\ast(C^{-1})^*),
\end{aligned}
\end{equation*}
where $C_5:=C_4C(3[\epsilon^{-1}])^3$.
\end{proof}

As a consequence, we give estimate of $\Im M_{E+i\epsilon}$:
\begin{coro}
We have the following estimates:
\begin{subequations}\label{coro:key-estimates}
\begin{align}
\label{key-eq-1}
\left(m + \tr\left[C^{-1} M_{E+i\epsilon}^+(\theta) \left(M_{E+i\epsilon}^+(\theta)\right)^\ast (C^{-1})^\ast\right]\right) \left\|\Im M^+_{E+i\epsilon}\right\|^{-1} &\leq m C_5, \\
\label{key-eq-2}
\left\|\Im M^+_{E+i\epsilon}\right\| &\leq 2m^2 C_5 \|C\|^2, \\
\label{key-eq-3}
\left\|\left(\Im M^+_{E+i\epsilon}\right)^{-1}\right\| &\leq C_5, \\
\label{key-eq-4}
\kappa\left(\Im M^+_{E+i\epsilon}\right) &\leq 2m^2 C_5^2 \|C\|^2.
\end{align}
\end{subequations}
\end{coro}

\begin{proof}
Equation \eqref{key-eq-1} is a corollary of \eqref{eqn: criterion key ineq}, while \eqref{key-eq-4} directly follows from \eqref{key-eq-2} and \eqref{key-eq-3}.  Denote $\|\cdot\|_{\text{HS}}$ as the Hilbert-Schmidt norm of matrices. Recall the inequalities $\|\cdot\| \leq \|\cdot\|_{\text{HS}} \leq \sqrt{2m} \|\cdot\|$, $|\tr(\cdot)| \leq m \|\cdot\|$, and $\left\|M_{E+i\varepsilon}^+\right\|_{\text{HS}} \geq \left\|\Im M_{E+i\varepsilon}^+\right\|_{\text{HS}}$. By \eqref{eqn: criterion key ineq}, we derive
\begin{align*}
C_5 m \left\|\Im \left(M^+_{E+i\epsilon}\right)\right\| &\geq m + \left\|C^{-1} M^+_{E+i\epsilon}\right\|_{\text{HS}}^2 \nonumber \\
&\geq \|C\|_{\text{HS}}^{-2} \left\|M^+_{E+i\epsilon}\right\|_{\text{HS}}^2 \nonumber \\
&\geq \|C\|_{\text{HS}}^{-2} \left\|\Im \left(M^+_{E+i\epsilon}\right)\right\|_{\text{HS}}^2 \nonumber \\
&\geq \frac{1}{2m} \|C\|^{-2} \left\|\Im \left(M^+_{E+i\epsilon}\right)\right\|^2,
\end{align*}
which implies \eqref{key-eq-2}.

For two positive definite matrices $A$ and $B$, if $A \leq B$, then $A^{-1} \geq B^{-1}$. Applying this to \eqref{eqn: criterion key ineq}, we obtain
\begin{align*}
\left(\Im M^+_{E+i\epsilon}\right)^{-1} &\leq m \left(\tr \Im M^+_{E+i\epsilon}\right)^{-1} 
\leq m C_5 \left(m + \tr\left[C^{-1} M_{E+i\epsilon}^+ \left(M_{E+i\epsilon}^+\right)^\ast (C^{-1})^\ast\right]\right)^{-1} 
\leq C_5,
\end{align*}
which implies \eqref{key-eq-3}.
\end{proof}

Similarly, we can define the inverse vertical bundle 
    $
    \mathcal{U}_\omega := \ker \rho_\omega = \mathbb{C}^m(0,\omega) \oplus \{0\},
    $
    where \( \rho_\omega: \mathbb{C}^{2m}_\omega \to \mathbb{C}^m(1,\omega) \) is the canonical projection to its second factor. And let $\mathcal{G}^u_\omega :=  \left\{ E \in \Sigma \,:\, E^u_{A_E}(\omega) \cap \mathcal{U}_\omega = \{0\} \right\}.$ If $E\in \mathcal{G}^u_\omega,$    analogous inequalities hold for $M^-_{E+i\epsilon}$.
  We can derive the estimate of the spectral measure on $\mathcal{G}_\omega^s\cap \mathcal{G}_\omega^u$:
\begin{coro}\label{JL} Let   $E\in\mathcal{G}_\omega^s\cap \mathcal{G}_\omega^u$. Then
    $$\mu_\omega(E-\epsilon, E+\epsilon)\leq  \epsilon C(E,\omega) \sup_{|s|\leq 3 \epsilon^{-1}}\|(A_E)_{s}(\omega)|_{E^c_{A_E}}\|^{42}.$$
    Furthermore, $C(E)$ is continuously dependent on $E.$
\end{coro}
\begin{proof}
    By \eqref{key-eq-1}, \eqref{key-eq-4} and similar inequalities for $M_-$, we have (let $X,Y=M^{\pm}_{E+i\epsilon}$)
\begin{eqnarray*}
 \tr\Im M_{E+i\epsilon}
&\leq& (\max(\kappa(\Im X),\kappa(\Im Y))^3 ((m  +  \tr C^{-1} XX^\ast(C^{-1})^*) \|\Im  X\|^{-1} \\&+&(m  + \tr C^{-1} YY^\ast(C^{-1})^*) \|\Im  Y\|^{-1})\nonumber\\
&\leq & C\sup_{|s|\leq 3 \epsilon^{-1}}\|(A_E)_{s}(\omega)|_{E^c_{A_E}}\|^{42} .
\end{eqnarray*}
Combined with  Lemma \ref{spm}, we get the desired proof.
\end{proof}

\subsection{Wronskian argument}\label{comproof}
To eliminate the impact of $E^s_{A_E}(\omega) \cap \mathcal{V}_\omega \neq \{0\}$, we primarily utilize the Wronskian. For any Schr\"odinger operator $H_V$ defined on a strip, we consider the equations $H_{V}\mathbf{u}=E\mathbf{u}$ and $H_{V}\mathbf{v}=E\mathbf{v}$. We define the Wronskian of the solution pair $(\mathbf{u},\mathbf{v})$ as follows:
$$
W(\mathbf{u},\mathbf{v})(n):=\begin{pmatrix}
\mathbf{u}(n+1) & \mathbf{u}(n)
\end{pmatrix} S \begin{pmatrix}
\mathbf{v}(n+1) \\ \mathbf{v}(n)
\end{pmatrix} =\mathbf{u}^*(n)C\mathbf{v}(n+1) - \mathbf{u}^*(n+1)C^*\mathbf{v}(n).
$$
It is noteworthy that
$$
W(\mathbf{u},\mathbf{v})(n)=\psi\left(\begin{pmatrix}
\mathbf{u}(n+1)\\\mathbf{u}(n)
\end{pmatrix},\begin{pmatrix}
\mathbf{v}(n+1)\\\mathbf{v}(n)
\end{pmatrix}\right).
$$
Since $A_E$ is a Hermitian symplectic matrix, we have
\begin{equation}\label{wron1}
W(\mathbf{u},\mathbf{v})(n)=\psi\left((A_E)_n\begin{pmatrix}
\mathbf{u}(1)\\\mathbf{u}(0)
\end{pmatrix},(A_E)_n\begin{pmatrix}
\mathbf{v}(1)\\\mathbf{v}(0)
\end{pmatrix}\right)=\psi\left(\begin{pmatrix}
\mathbf{u}(1)\\\mathbf{u}(0)
\end{pmatrix},\begin{pmatrix}
\mathbf{v}(1)\\\mathbf{v}(0)
\end{pmatrix}\right),
\end{equation}
which is independent of $n\in\mathbb{Z}$.

This concept clearly generalizes the Wronskian determinant for one-dimensional Schr\"odinger operators $H_v$ defined in \eqref{schrodinger}. We consider the equations $H_vu=Eu$ and $H_v\tilde{u}=E\tilde{u}$. The Wronskian determinant is defined as
$$
W(u,\tilde{u})(n):=u(n)\tilde{u}(n+1)-\tilde{u}(n)u(n+1).
$$
The Liouville Theorem demonstrates that $W(u,\tilde{u})(n)$ is independent of $n\in\mathbb{Z}$ \cite{DF}. However, $W(\mathbf{u},\mathbf{v})$ is not a determinant in the traditional sense, and there is no Liouville Theorem applicable in this case. Equation \eqref{wron1} provides a new interpretation of the Liouville Theorem, namely that the preservation of the symplectic form leads to the constancy of the Wronskian.

Once we have this, if we denote $\sigma_{\omega,p} = \{E \in \Sigma \mid E \text{ is an eigenvalue of } H_{V,T,\omega}\},$ we have the following: 
\begin{lemma}\label{wron}
Suppose that $(T,A_E)$ is partially hyperbolic, and \(E \in \mathcal{B}_\omega\). Then \(E \notin \sigma_{\omega,p} \).
\end{lemma}
\begin{proof} Suppose there exists $\mathbf{u} \in \ell^2(\mathbb{Z}, \mathbb{C}^m)$ such that $H_{V,\omega,\alpha} \mathbf{u} = E \mathbf{u}$. Then, we have $\begin{pmatrix}
\mathbf{u}_1 \\
\mathbf{u}_0
\end{pmatrix} \in E_{A_E}^c(\omega).$ Otherwise, 
 $\begin{pmatrix}
\mathbf{u}_1 \\
\mathbf{u}_0
\end{pmatrix}$  has nonzero components in $ E^u_{A_E} \oplus E^s_{A_E}$, it follows that 
we have $\mathbf{u}_n \to \infty$ as $n \to -\infty$ or $n \to \infty.$
For any $\mathbf{v} \in E_{A_E}^c(\omega)$, we define
$$
\begin{pmatrix}
\mathbf{v}_k \\
\mathbf{v}_{k-1}
\end{pmatrix} := (A_z)_{k}(\omega) \mathbf{v}.
$$
Note that $E_{A_E}^c(\omega)$ is an invariant subspace, and $(A_E)_{s}(\omega)|_{E_{A_E}^c}$ is bounded. Thus, we have
$
\sup_{k \in \mathbb{Z}} \bigg\|\begin{pmatrix}
\mathbf{v}_k \\
\mathbf{v}_{k-1}
\end{pmatrix}\bigg\| < \infty.
$
Therefore,
$$
\bigg|\psi\left(\mathbf{v}, \begin{pmatrix}
\mathbf{u}_1 \\
\mathbf{u}_0
\end{pmatrix}\right)\bigg| = |W(\tilde{\mathbf{v}}, \mathbf{u})(n)| \leq C \bigg\|\begin{pmatrix}
\mathbf{v}_k \\
\mathbf{v}_{k-1}
\end{pmatrix}\bigg\| \bigg\|\begin{pmatrix}
\mathbf{u}_k \\
\mathbf{u}_{k-1}
\end{pmatrix}\bigg\| \to 0
$$
by \eqref{wron1} and the definition of the Wronskian.

Since $E_{A_E}^c(\omega)$ is a symplectic subspace by Lemma \ref{wron11}, and due to the non-degeneracy of the symplectic form, we conclude that $\mathbf{u} = 0.$ It follows that \(E \notin \sigma_{\omega,p} \). \end{proof}

Moreover,
we have the following simple, but quite important observation:
\begin{lemma}\label{eig}
Denote $\mathcal{G}_{\omega}=\mathcal{G}_{\omega}^u\cap\mathcal{G}_{\omega}^s$, then  
 we have 
$\mu_\omega(\Sigma_{\omega} \backslash \mathcal{G}_{\omega}) = \mu_{\omega,pp}(\Sigma_{\omega}  \backslash \mathcal{G}_{\omega}).$ 
\end{lemma}
\begin{proof}
Note that if $E \in \Sigma_{\omega} \backslash \mathcal{G}_\omega^s$, then $E$ is an eigenvalue of $H^+_{V,T,\omega}$. Consequently, $\Sigma_{\omega} \backslash \mathcal{G}_\omega^s$ can contain only countably many points.  Then, we have
$$
\begin{aligned}
\mu_\omega(\Sigma_{\omega} \backslash \mathcal{G}_\omega^s) = \mu_{\omega,pp}(\Sigma_{\omega} \backslash \mathcal{G}_\omega^s) 
&= \mu_{\omega,pp}((\Sigma_{\omega} \backslash \mathcal{G}_\omega^s) \cap \sigma_{\omega,p}) + \mu_{\omega,pp}((\Sigma_{\omega} \backslash \mathcal{G}_\omega^s) \cap \sigma_{\omega,p}^c),
\end{aligned}
$$
the result follows. 
\end{proof}

  \textbf{Proof of Theorem \ref{thm:main-spectral}:}\label{proof1}
By Corollary \ref{JL}, we have $\mu_{\omega,s}(\mathcal{B}_{\omega}  \cap \mathcal{G}_{\omega})=0.$ By Lemma \ref{wron} and Lemma \ref{eig},  we have $\mu_{\omega}(\mathcal{B}_{\omega}\backslash \mathcal{G}_\omega)=0. $  The result follows.\qed

 	\section{Subordinacy theory for infinite-range operator}\label{proof of inf}
We start with the following simple observations:

\begin{lemma}\label{op-s}\cite{KL1}
Let \( H \) be a self-adjoint operator acting on \( \ell^2(\mathbb{Z}) \), and let \( \phi \in \ell^2(\mathbb{Z}) \) be a fixed vector. For \( z \in \mathbb{C} \setminus \mathbb{R} \), we have
 $$\Im z \|(H-z)^{-1}\phi\|^2_{\ell^{2}(\Z)}=\Im \langle (H-z)^{-1}\phi,\phi \rangle.$$
\end{lemma}

    \begin{lemma}\label{der}\cite{DF}
    For a finite and compactly support Borel measure $\mu$ on $\R$,
we denote its Stieltjes transform by 
$$F_\mu(z)=\int_\R \frac{d\mu(x)}{x-z},z\in\C\backslash\R.$$
    Then for every $E\in\R,$ we have 
        $$ C_1 D_\mu^{+,\alpha}(E)\leq\limsup_{\varepsilon\rightarrow0}\varepsilon^{1-\alpha}\Im F_\mu(E+i\varepsilon)\leq C_2 D_\mu^{+,\alpha}(E).$$
    \end{lemma}

To prove Theorem \ref{sub}, we will estimate from below $ \Im \langle (H-E-i\epsilon)^{-1}\phi, \phi \rangle $ as $ \epsilon \to 0 $. By applying Lemma \ref{op-s}, it suffices to estimate from below the norm of the function $$
v_\epsilon(\cdot) = (H-E-i\epsilon)^{-1}\phi.
$$ The key tool employed is the Lagrange bilinear form. For any two sequences $ \{f_n\}_{n\in \mathbb{Z}} $ and $ \{g_n\}_{n\in \mathbb{Z}} $, we define the Lagrange bilinear form as follows: $$
W_{[-r,r]}(f,g) := \langle Hf, g \rangle_r - \langle f, Hg \rangle_r = \sum_{n=-r}^{r} \left[ (Hf)_n \overline{g_n} - f_n \overline{(Hg)_n} \right].
$$ A straightforward calculation yields the following result:

\begin{lemma}\label{lower}
Let \( u = \{u_E(n)\}_{n \in \mathbb{Z}} \) be any generalized eigenfunction satisfying \( (H-E)u_E = 0 \). Then, for any \( 0 < R_0 < R \), 
we have estimate:
$$ \bigg|\sum_{r=R_0}^RW_{[-r,r]}(v_\epsilon, u)\bigg|\geq \bigg|\sum_{r=R_0}^R \langle \phi,u\rangle_r\bigg|- \epsilon R \|v_\epsilon\|_{\ell^2(2R)} \|u\|_{\ell^2(2R)}. $$
\end{lemma}
\begin{proof}
Note that 
$$\begin{aligned}
		W_{[-r,r]}(v_\epsilon, u)&=\langle Hv_\epsilon,u\rangle_r-\langle v_\epsilon,Hu\rangle_r\\
        &=(E+i\epsilon)\langle v_\epsilon,u\rangle_r+\langle \phi,u\rangle_r -E\langle v_\epsilon,u\rangle_r \\
        &= \langle \phi,u\rangle_r+i\epsilon\langle v_\epsilon,u\rangle_r.
	\end{aligned}$$
Taking sum from $R_0$ to $R$, then the result follows from Cauchy inequality. 

\end{proof}

This demonstrates that if $ \langle \phi, u \rangle_{\ell^2(\mathbb{Z})} \neq 0 $, a lower bound for $ \bigg|\sum_{r=R_0}^RW_{[-r,r]}(v_\epsilon, u)\bigg| $ can always be established.

Next, we consider the self-adjoint long-range operator on $ \ell^2(\mathbb{Z}) $: $$
(Hu)_n = \sum_{k \in \mathbb{Z}} w_k u_{n+k} + v_n u_n,
$$ where $ \overline{w_{-k}} = w_k $ for $ k \in \mathbb{Z} $, and $ \{v_n\}_{n \in \mathbb{Z}} $ is a bounded real sequence. The crucial observation is that if $ w_k $ decays rapidly, we can obtain a good upper bound estimate for $ \sum_{r=R_0}^RW_{[-r,r]}(v_\epsilon, u) $.

\begin{proposition}\label{green1}
If \( |w_k| < \frac{C}{k^3} \) and \( \|f\|_{\ell^2(\mathbb{Z})} < \infty \), \( \|g\|_{\infty} \leq 1 \), then we have
$$\bigg|\sum_{r=1}^{R}W_{[-r,r]}(f,g)\bigg|\leq  {6\|f\|_{\ell^2(\Z)}}+C\|f\|_{\ell^2(2R)}\|g\|_{\ell^2(2R)}.$$
\end{proposition}
\begin{proof} To estimate the Lagrange bilinear form for an infinite-range operator, we first provide an estimate for the finite-range case:
\begin{lemma}\label{finite}
If \( w_k = 0 \) for all \( |k| > K \), then
$$\bigg|\sum_{r=1}^{R}W_{[-r,r]}(f,g)\bigg|\leq 
 \big(\sum_{j=1}^{K}4j|w_j| \big) \|f\|_{\ell^2(R+K)}\|g\|_{\ell^2(R+K)}.$$
	 \end{lemma}
\begin{proof} By the definition of the Lagrange bilinear form, we have \begin{align*}
W_{[-r,r]}(f,g) &= \sum_{n=-r}^{r} \sum_{k=-K}^{K} w_k f_{n+k} \overline{g_n} - \sum_{n=-r}^{r} f_n \sum_{k=-K}^{K} \overline{w_k} \overline{g_{n+k}} \\
&= \sum_{k=-K}^{K} \sum_{n=-r}^{r} w_k f_{n+k} \overline{g_n} - \sum_{k=-K}^{K} \sum_{n=-r}^{r} f_n \overline{w_k} \overline{g_{n+k}}.
\end{align*}
Let $$
W_j := \sum_{|k|=j} \sum_{n=-r}^{r} w_k f_{n+k} \overline{g_n} - \sum_{|k|=j}\sum_{n=-r}^{r} f_n \overline{w_k} \overline{g_{n+k}}, \quad 1 \leq j \leq K.
$$ A direct calculation shows that \begin{align*}
W_j &= \sum_{n=-r}^{r} w_j f_{n+j} \overline{g_n} + \sum_{n=-r}^{r} w_{-j} f_{n-j} \overline{g_n} - \sum_{n=-r}^{r} f_n \overline{w_j g_{n+j}} - \sum_{n=-r}^{r} f_n \overline{w_{-j} g_{n-j}} \\
&= \sum_{n=-r+j}^{r+j} w_j f_n \overline{g_{n-j}} + \sum_{n=-r-j}^{r-j} w_{-j} f_n \overline{g_{n+j}} - \sum_{n=-r}^{r} f_n w_{-j} \overline{g_{n+j}} - \sum_{n=-r}^{r} f_n w_j \overline{g_{n-j}} \\
&= \sum_{n=r+1}^{r+j} w_j f_n \overline{g_{n-j}} - \sum_{n=-r}^{-r+j-1} f_n w_j \overline{g_{n-j}} + \sum_{n=-r-j}^{-r-1} w_{-j} f_n \overline{g_{n+j}} - \sum_{n=r-j+1}^{r} f_n w_{-j} \overline{g_{n+j}} \\
&= (I) + (II) + (III) + (IV),
\end{align*} where we used the fact that $ \overline{w_j} = w_{-j} $ in the second equality.

Thus, we can estimate \begin{align*}
\left| \sum_{r=1}^{R} (I) \right| &\leq |w_j| \sum_{r=1}^{R} \sum_{n=r+1}^{r+j} |f_n| |g_{n-j}| 
\leq |w_j| \sum_{r=1}^{R} \sum_{n=1}^{j} |f_{r+n}| |g_{r+n-j}| \\
&= |w_j| \sum_{n=1}^{j} \sum_{r=1}^{R} |f_{r+n}| |g_{r+n-j}| \leq j |w_j| \|f\|_{\ell^2(R+K)} \|g\|_{\ell^2(R+K)}.
\end{align*} The other terms can be estimated similarly. Therefore, we obtain the estimate $$
\left| \sum_{r=1}^{R} W_{[-r,r]}(f,g) \right| = \left| \sum_{r=1}^{R} \sum_{j=1}^{K} W_j \right| \leq \sum_{j=1}^{K} \left| \sum_{r=1}^{R} W_j \right| \leq \big(\sum_{j=1}^{K}4j|w_j| \big) \|f\|_{\ell^2(R+K)} \|g\|_{\ell^2(R+K)}.
$$ \end{proof}

Now suppose that $ H_K = \sum_{|k| \leq K} w_k u_{n+k} + u_n v_n $. By the definition of the Lagrange bilinear form, we can write $$
\begin{aligned}
\sum_{r=1}^{R} W_{[-r,r]}(f,g) &= \sum_{r=1}^{R} \big( \langle Hf, g \rangle_r - \langle f, Hg \rangle_r \big) \\
&= \sum_{r=1}^{R} \big( \langle (H-H_K)f, g \rangle_r + \langle H_K f, g \rangle_r - \langle f, (H-H_K)g \rangle_r - \langle f, H_K g \rangle_r \big).
\end{aligned}
$$ Since $ |w_k| < \frac{C}{k^3} $, a direct calculation shows that $$
\bigg| \langle (H-H_K)f, g \rangle_r \bigg| \leq \sum_{n=-r}^{r} \sum_{|m-n|>K} |w_{m-n} f_m| |g_n| \leq (2r+1) \frac{\|f\|_{\ell^2(\mathbb{Z})}}{K^2}.
$$ Similarly, one can estimate $$
\bigg| \langle f, (H-H_K)g \rangle_r \bigg| \leq (2r+1) \frac{\|f\|_{\ell^2(\mathbb{Z})}}{K^2}.
$$

By Lemma \ref{finite}, we have \begin{align*}
\left| \sum_{r=1}^{R} W_{[-r,r]}(f,g) \right| &\leq \frac{6R^2 \|f\|_{\ell^2(\mathbb{Z})}}{K^2} + \big(\sum_{k=1}^{K} 4k |w_k|\big) \|f\|_{\ell^2(R+K)} \|g\|_{\ell^2(R+K)}.
\end{align*} Choosing $ K = R $, the result follows. \end{proof}

These two results allow us to obtain a lower bound estimate for $ \|v_\epsilon\|_{\ell^2(2R)} $, thus providing a lower bound for the upper derivative of the spectral measure:

\begin{proof}[Proof of Theorem \ref{sub}] Let $ c := \sum_n \overline{\phi(n)} u_E(n) \neq 0 $. Take a sufficiently large $ R_0 $ such that $ \text{supp} \, \phi \subset B_{R_0} $, and let $ (H-E-i\varepsilon)^{-1} \phi(\cdot) = v_\epsilon(\cdot) $.

Within this context, take $ R $  larger than $ R_0 $. Lemma \ref{lower} gives $$
\sum_{r=R_0}^R \left| W_{[-r,r]}(v_\epsilon, u) \right| \geq |c| (R - R_0) - \epsilon R \|v_\epsilon\|_{\ell^2(2R)} \|u\|_{\ell^2(2R)}.
$$ On the other hand, Proposition \ref{green1} provides $$
\sum_{r=R_0}^R \left| W_{[-r,r]}(v_\epsilon, u) \right| \leq 6 \|v_\epsilon\|_{\ell^2(\mathbb{Z})} + C \|v_\epsilon\|_{\ell^2(2R)} \|u\|_{\ell^2(2R)}.
$$ Setting $ \epsilon = \frac{1}{R} $, we obtain $$
6 \|v_\epsilon\|_{\ell^2(\mathbb{Z})} + C \|v_\epsilon\|_{\ell^2(2R)} \|u\|_{\ell^2(2R)} \geq |c| (R - R_0).
$$ According to the assumption, there exists a sequence $R\rightarrow\infty$, such that $$
\|u\|_{\ell^2(2R)} \leq C R^{\alpha/2}.
$$ Thus, we find that $$
C \|v_\epsilon\|_{\ell^2(\mathbb{Z})} R^{\alpha/2} \geq |c| (R - R_0).
$$ Therefore, if $ R $ is sufficiently large, we conclude that $$
\|v_\epsilon\|_{\ell^2(\mathbb{Z})} \geq C R^{1-\alpha/2}.
$$

Now, we invoke Lemma \ref{op-s} and note that $$
\Im \langle (H-E-i\epsilon)^{-1} \phi, \phi \rangle = \epsilon \|(H-E-i\epsilon)^{-1} \phi\|^2 = \epsilon \|v_{\epsilon_n}\|_{\ell^2(\mathbb{Z})}^2 \geq C\epsilon^{\alpha-1}.
$$ By Lemma \ref{der}, we have $ D^{+,\alpha}_{\mu_{\phi}} \geq C. $ \end{proof}

\section{Absence of point spectrum for long-range operator}\label{ab-point}
 	In this section, we prove the absence of point spectrum of  $L_{\varepsilon v,w,\alpha,\theta}$. 
 			
\begin{proposition}\label{np}
		Suppose that $\alpha\in DC$, $w(\cdot)$ and $v(\cdot)$ are analytic function on $\T^d$. Then
	there exists $\varepsilon_2=\varepsilon_2(\alpha,v,w)>0$ such that if $|\varepsilon|<\varepsilon_2$, $L_{\varepsilon v,w,\alpha,\theta}$   has no point spectrum for any $\theta$.
\end{proposition}
\begin{proof}
Suppose that there exists $E$ such that $L_{\varepsilon v,w,\alpha,\theta}u=Eu $ has a solution $u=(u_n)_{n\in\Z}\in\ell^2(\mathbb{Z},\mathbb{C}),$ with $\|u\|_{\ell^2(\mathbb{Z})}=1$, which implies that if we 
define  $\hat{u}(x)=\sum_{n\in\mathbb{Z}}u_ne^{inx},$ then $\hat{u}\in L^2(\TT^d)$.
Thus for full measure $x \in\TT $,  the sequence $\tilde{u}$ defined by
$
\tilde{u}(n)=\hat{u}(x+\langle n, \alpha\rangle) e^{2 \pi i \langle n, \theta\rangle},n\in\Z^d
$
is a solution of the dual operator $\widehat{L}_{\varepsilon v,w,\alpha,x}\tilde{u}=E \tilde{u}$, where  $\widehat{L}_{\varepsilon v,w,\alpha,x}$ is defined in \eqref{dualo}.
Note that
$$
\begin{aligned}
	&\quad\int_{\TT^d } \sum_{n \in \mathbb{Z}^d}|\tilde{u}(n)|^2(1+|n|)^{-2d} d x =\sum_{n \in \mathbb{Z}^d}(\int_{\TT }|\hat{u}(x+\langle n, \alpha\rangle)|^2 d x)(1+|n|)^{-2d } \\
	&=\sum_{n \in \mathbb{Z}^d}\left(\int_{\TT }|\hat{u}(x)|^2 d x \right)(1+|n|)^{-2d} 
	=\sum_{n \in \mathbb{Z}^d}(1+|n|)^{-2d} 
	<\infty.
\end{aligned}
$$

As a consequence, we obtain  full measure set $\mathcal{A}$, such that for  almost every $x\in \mathcal{A}$, there is a constant $C(x)<\infty$ such that
$$
|\tilde{u}(n)| \leq C(x)(1+|n|)^d \quad \text { for every } n \in \mathbb{Z}^d .
$$
We need the following Lemma on Green's function estimate:
\begin{lemma}\cite{Bourgain,Liu}
	Suppose that $\alpha\in DC$, $0<\sigma<1$. For any $\epsilon>0$, there exists $\varepsilon_2=\varepsilon_2(w,v,\alpha,\sigma,\epsilon)$, the following holds if $|\varepsilon|<\varepsilon_2$:
	Let N be sufficiently large. There is a subset $\Omega=\Omega_N(E)\subset \Omega $ satisfying $|\Omega|<e^{-N^{(\sigma-\epsilon)}}$ such that if $x\notin\Omega$, we have the Green's function estimate 
    $$\|G_{[-N,N]^d}(E+i0,x)\|\leq e^{N^\sigma}$$ and
	$$|G_{[-N,N]^d}(E+i0,x)(n,n')|<e^{-\frac{\rho}{2}|n-n'|} \text{ for all } |n-n'|\geq\frac{1}{10} N,$$ where $\rho$ is analytic radius of $v(\cdot).$
\end{lemma}
By Borel-Cantelli lemma, $|\limsup_N\Omega_N(E)|=0.$ Choose $x\notin \limsup_N\Omega_N(E)\cup \mathcal{A}^c$, then  $x\notin\Omega_N(E)$ for sufficiently large $N.$
Let $\Lambda=[-N,N]^d$, it follows that
 $$\begin{aligned}
	&\qquad|\hat{u}(x)|=|\tilde{u}(0)|\\&\leq\sum_{n\in\Lambda,n'\notin\Lambda}|G_{\Lambda}(E)(0,n)| |v_{n-n'}| |\tilde{u}(n')|\\
    &=\sum_{|n|\leq\frac{N}{10},n'\notin\Lambda}|G_{\Lambda}(E)(0,n)| |v_{n-n'}| |\tilde{u}(n')|+\sum_{n\in\Lambda,|n|>\frac{N}{10},n'\notin\Lambda}|G_{\Lambda}(E)(0,n)| |v_{n-n'}| |\tilde{u}(n')|\\
	&<\sum_{n\in\Lambda,|n|\leq\frac{N}{10},n'\notin\Lambda}Ce^{N^\sigma-\rho|n-n'|}|n'|^d+\sum_{n\in\Lambda,|n|>\frac{N}{10},n'\notin\Lambda}Ce^{-\frac{\rho}{2}|n|-\rho|n-n'|}|n'|^d\\
	&<Ce^{-\frac{\rho N}{4}}.
\end{aligned}$$
Letting $N\rightarrow\infty,$ we conclude $\hat{u}(x)=0$ for a.e. $x\in\T$, which  contradicts to the assumption $u\neq 0$. Hence, $L_{\varepsilon v,w,\alpha,\theta}$   has no point spectrum for any $\theta$.	
\end{proof}

If we further have $\hat{L}_{\varepsilon v,w,\alpha,x}$ is a one-dimensional Schro\"odinger operator, we have the following strong conclusion, which is essentially proved in \cite{De}:
\begin{proposition}\label{apoint}
    Suppose $E\in \sigma(H_{\varepsilon v,\alpha,\theta})$ with $L(E)>0$, then $E$ is not the eigenvalue of  the dual operator ${L}_{2\cos,\varepsilon v,\alpha,x}$.
\end{proposition}
\begin{proof}
    Suppose  $E$ is the eigenvalue of ${L}_{2\cos,\varepsilon v,\alpha,x}$. By the same argument of Proposition \ref{np}, for almost every $\theta$, there is a constant $C(\theta)<\infty$ such that
$$
|\tilde{u}(n)| \leq C(\theta)(1+|n|) \quad \text { for every } n \in \mathbb{Z}.
$$
Since the set of $\theta$ for which $E$ is an eigenvalue has Lebesgue measure zero. Consequently, by the multiplicative ergodic theorem, $|\tilde{u}(n)|$ grows exponentially on at least one half-line for almost every $\theta,$ which is a contradiction.
\end{proof}

\section{Absolutely continuous spectrum for finite-range operator}\label{proof-finite}
 In this section, we prove 
absolutely continuous spectrum for quasi-periodic finite-range operator.		

\subsection{Absolutely continuity of IDS}

To apply Theorem \ref{thm:main-spectral}, we first view the 
finite-range operator as a Sch\"odinger operator on the strip, and establish the relation of ${L}_{v,w,\alpha,\theta}$ and ${H}_{V,\alpha,\theta}$.
 Suppose $w(\theta)=\sum_{k=-K}^Kw_ke^{2\pi ik\theta}$, let	$$
C=\begin{pmatrix}
	w_K&\cdots&w_1\\
	0&\ddots&\vdots\\
	0&0&w_K
\end{pmatrix}
\text{ and }  
	V(\theta)=\begin{pmatrix}
		\varepsilon v (\theta+(K-1)\alpha)&w_{-1}&\cdots&w_{-K+1}\\
		w_1&\ddots&\ddots&\vdots\\
		\vdots&\ddots&\varepsilon v (\theta+\alpha)&w_{-1}\\
		w_{K-1}&\cdots&w_1&\varepsilon v(\theta)
	\end{pmatrix}.
$$
We have the following simple observation:	
\begin{lemma}\label{uni}
For any $\theta\in\TT^d$, ${H}_{V,K\alpha,\theta}$ and ${L}_{\varepsilon v,w,\alpha,\theta}$ is unitary equivalent.
 \end{lemma}	
\begin{proof}
	Let $$\begin{aligned}
U:l^2(\mathbb{Z},\mathbb{C})&\rightarrow l^2(\mathbb{Z},\mathbb{C}^K)\\
\{u_n\}_{n\in\mathbb{Z}}&\mapsto\bigg\{\begin{pmatrix}
	u_{Kn+K-1}&\cdots&u_{Kn+1}&u_{Kn}
\end{pmatrix}^t\bigg\}_{n\in\mathbb{Z}}
	\end{aligned}.$$
	We can easily check that $U$ is a unitary transform and
	$
		U^{-1}{{H}_{V,K\alpha,\theta}}U
		=	{L}_{\varepsilon v,w,\alpha,\theta} 
	$.
	Therefore, ${H}_{V,K\alpha,\theta}$ and ${L}_{\varepsilon v,w,\alpha,\theta}$ is unitary equivalent.
\end{proof}
Let us recall the definition of IDS.
\begin{definition}[Integrated Density of States] 
	Let $\{{H}_{V,\alpha,\theta}\}_{\theta\in\T^d}$ defined in \eqref{strip-operator}. One defines the density of states of $\{{H}_{V,\alpha,\theta}\}_{\theta\in\T^d}$ by
	\begin{equation*}
		k(A) :=  \int_{\Omega} \tr(P_0\chi_A(H_{V,\alpha,\theta})P_0^*) d\nu(\theta),
	\end{equation*}
	where $A \subset \RR$ is the Borel set, $\chi_A(H_\theta)$ is the spectral projection,   and $P_0:\ell_2(\mathbb{Z},\mathbb{C}^m)\rightarrow \mathbb{C}^m$ is the projection $P_0u=u_0$.
\end{definition}

Moreover, we have the following well-known result:
\begin{theorem}[Thouless Formula for Schr\"odinger operator on the strip]\cite{CS,KS}
	\label{thouless}
	For  $z \in \mathbb{C}$, we have
	$$
	\sum_{j=1}^{K}L_j(z) = \int_{\mathbb{R}} \log \left|z - x\right| dk(x) - \log |\det C |.
	$$
\end{theorem}

The finite-range operator ${L}_{\varepsilon v,w,\alpha,\theta}$ can also induce a cocycle $(\alpha,L_{E}(\cdot))$, where $$
L_{E}(\theta)=\begin{pmatrix} -\frac{w_{K-1}}{w_{K}} & \cdots & -\frac{w_{1}}{w_{K}} & \frac{E-w_0-\varepsilon v(\theta)}{w_{K}} & -\frac{w_{-1}}{w_{K}} & \cdots & -\frac{w_{\ell+1}}{w_{K}} & -\frac{w_{-K}}{w_{K}} \\ 1 &&&&&&&
\\& \ddots &&&&&&
\\&& 1 &&&&&
\\&&& 1&&&&
\\&&&& 1&&&
\\&&&&& \ddots&&
\\&&&&&& 1& 0
\end{pmatrix}.
$$
And we denote the Lyapunov exponents as $\{\gamma_j\}_{j=1}^{2K}.$
We also have the Thouless Formula for finite-range operator.
\begin{theorem}[Thouless Formula for finite-range  operator]\cite{ Puig}
	\label{thouless1}
	For  $z \in \mathbb{C}$, we have
	$$
	\sum_{j=1}^{K}\gamma_j(z) = \int_{\mathbb{R}} \log \left|z - x\right| d\mathcal{N}(x) - \log |w_K|,
	$$where $\mathcal{N}(\cdot)$ is defined in \eqref{sids}.
\end{theorem}

Within these result, absolutely continuity of $\{H_{V, K\alpha, \theta}\}$ can inherit from ${L}_{\varepsilon v,w,\alpha,\theta}$ :

\begin{proposition}\label{idsac}
Suppose that $\alpha \in DC$. Then, there exists $\varepsilon_3 = \varepsilon_3(\alpha, v, w) > 0$ such that if $|\varepsilon| < \varepsilon_3$, the IDS $k(\cdot)$ of $\{H_{V, K\alpha, \theta}\}$ is absolutely continuous.
\end{proposition}

\begin{proof}
Through direct calculations, we have $(K\alpha, A_E(\theta)) = (\alpha, L_E(\theta))^K$. Thus, $\sum_{j=1}^{K} L_j(z) = K \sum_{j=1}^{K} \gamma_j$. By Theorem \ref{thouless} and Theorem \ref{thouless1}, it follows that     
	$$\int\log|z-z'|d\mathcal{N}(z')-\log|w_K|=\frac{\sum_{j=1}^{K}L_j(z)}{K}=\frac{\int\log|z-z'|dk(z')}{K}-\log|w_K|.$$
	Let us recall the following key result:

\begin{lemma}[{\cite[Theorem 3.2.3]{ransford}}]\label{ran}
    Let $\mu_1$ and $\mu_2$ be finite Borel measures on $\mathbb{C}$ with compact support. If 
    \[
    \int_{\mathbb{C}} \log|z - z'| \, d\mu_1(z') = \int_{\mathbb{C}} \log|z - z'| \, d\mu_2(z') + h(z)
    \]
    holds on an open set $U \subset \mathbb{C}$, where $h$ is harmonic on $U$, then $\mu_1|_U = \mu_2|_U$.
\end{lemma}

Since $\mathcal{N}(\cdot)$ is absolutely continuous (\cite[Theorem 1.2]{WXYZ}), Lemma~\ref{ran} implies that $k(\cdot)$ inherits absolute continuity.
\end{proof}

\subsection{Absolutely continuity of the spectral measure}

\begin{lemma}\label{iml}
    Suppose $\alpha \in \mathrm{DC}$, $w(\cdot)$ is a trigonometric polynomial, and $v(\cdot)$ is analytic on $\mathbb{T}^d$. Then there exists $\varepsilon_4 = \varepsilon_4(\alpha, v, w) > 0$ such that for $|\varepsilon| < \varepsilon_4$, the following holds for a.e. $E \in \Sigma$:
    \begin{enumerate}
        \item The cocycle $(K\alpha, A_E(\cdot))$ is partially hyperbolic.
        \item $\sup_s \|(A_E)_s(\theta)|_{E_{A_E}^c}\| < \infty$.
    \end{enumerate}
\end{lemma}

\begin{proof}
    We begin with a key result from \cite{WXYZ}:
    
    \begin{lemma}[{\cite[Theorem 2.3]{WXYZ}}]\label{lem:reduction}
        There exists $\tilde{\varepsilon}_1 = \tilde{\varepsilon}_1(\alpha, v, w) > 0$ such that for $|\varepsilon| \leq \tilde{\varepsilon}_1$, there exists a Lebesgue measure zero set $\mathcal{S} \subset \Sigma$ with the following property: For every $E \in \Sigma \setminus \mathcal{S}$, there exists $B_E \in C^\omega(\mathbb{T}^d, \mathrm{GL}(2K, \mathbb{C}))$ satisfying
        \[
        B_E^{-1}(\cdot + \alpha)L_E(\cdot)B_E(\cdot) = D_E,
        \]
        where $D_E = \mathrm{diag}\{\lambda_{1,E}, \ldots, \lambda_{2K,E}\}$ with $\lambda_{i,E} \neq \lambda_{j,E}$ for $i \neq j$.
    \end{lemma}

    Observe that $(K\alpha, A_E(\theta)) = (\alpha, L_E(\theta))^K$. For $E \in \Sigma \setminus \mathcal{S}$, conjugacy yields:
    \[
    B_E^{-1}(\cdot + K\alpha)A_E(\cdot)B_E(\cdot) = \mathrm{diag}\{\lambda_{1,E}^K, \ldots, \lambda_{2K,E}^K\}.
    \]
    Assume without loss of generality $|\lambda_{1,E}^K| \geq \cdots \geq |\lambda_{2K,E}^K|$.
    \begin{Claim}
        For $E \in \Sigma \cap \mathcal{S}^c$:
        \begin{enumerate}
            \item $(K\alpha, A_E(\cdot))$ is partially hyperbolic.
            \item $\dim E_{A_E}^c = 2m_E$ for some $m_E \in \mathbb{Z}^+$.
        \end{enumerate}
    \end{Claim}

    \begin{proof}[Proof of Claim]
        Since $(K\alpha, A_E(\cdot))$ is symplectic \cite{Puig}, its Lyapunov exponents occur in pairs $\pm \gamma_i$. Conjugacy invariance implies:
      $
        |\lambda_{i,E}| = |\lambda_{2K+1-i,E}| $ for $1 \leq i \leq K.
        $
        For $E \in \Sigma$, uniform hyperbolicity is precluded, so $|\lambda_{K,E}| = 1$. Thus, there exists $m_E\in \mathbb{Z}^+$ such that:
        \begin{equation*}\label{eq:eigen}
            |\lambda_{1,E}| \geq \cdots \geq |\lambda_{K-m_E,E}| > 1 = |\lambda_{K-m_E+1,E}| = \cdots = |\lambda_{K+m_E,E}| > |\lambda_{K+m_E+1,E}| \geq \cdots \geq |\lambda_{2K,E}|.
        \end{equation*}
        The analytic conjugacy $B_E$ decomposes the phase space into invariant bundles:

        \[
        \begin{aligned}
            E_{A_E}^u(\theta) &= \operatorname{span}\{b_1(\theta), \ldots, b_{K-m_E}(\theta)\}, \\
            E_{A_E}^c(\theta) &= \operatorname{span}\{b_{K-m_E+1}(\theta), \ldots, b_{K+m_E}(\theta)\}, \\
            E_{A_E}^s(\theta) &= \operatorname{span}\{b_{K+m_E+1}(\theta), \ldots, b_{2K}(\theta)\}.
        \end{aligned}
        \]
        This spectral gap confirms partial hyperbolicity.
    \end{proof}

    To establish $\sup_s \|(A_E)_s(\theta)|_{E_{A_E}^c}\| < \infty$, let $\psi \in E_{A_E}^c(\theta)$ with $\|\psi\| = 1$. Decompose $\psi = \sum_{j=K-m_E+1}^{K+m_E} c_j b_j(\theta),$
    This just means  $$B(\theta)\begin{pmatrix}
    0&\cdots&0&c_{K-m_E+1}&\cdots &c_{K+m_E}&0&\cdots&0
\end{pmatrix}^t=\psi,$$ it follows that
$\sqrt{|c_{K-m+1}|^2+\cdots+|c_{K+m}|^2}\leq\|B^{-1}\|.$
    Then:
   \begin{eqnarray*}
	\|(A_E)_{s}(\theta)|_{E_{A_E}^c}\psi\| \leq \sum_{j=K-m_E+1}^{K+m_E}\|c_j \lambda_{j,E}b_j(\theta+Ks\alpha)\|\leq \sum_{j=K-m_E+1}^{K+m_E}|c_j| \|B\|
	\leq 2m_E\|B^{-1}\|\|B\|.
\end{eqnarray*}
    The uniform bound follows from the analyticity of $B$ and $B^{-1}$.
\end{proof}

\begin{proof}[Proof of Theorem \ref{ac}]
    By Proposition \ref{np}, it suffices to show that for Lebesgue-almost every $\theta \in \mathbb{T}^d$, the operator $L_{\varepsilon v,w,\alpha,\theta}$ has purely absolutely continuous spectrum. Partition the spectrum as
        \[
        \Sigma = (\Sigma \cap \mathcal{S}) \cup (\Sigma \cap \mathcal{S}^c),
        \]
        where $\mathcal{S} \subset \Sigma$ is the Lebesgue measure zero set from Lemma \ref{lem:reduction}.

      Let $\mu_\theta$ denote the spectral measure of $H_{V,K\alpha,\theta}$. By Theorem \ref{thm:main-spectral} and Lemma \ref{iml}, the singular component satisfies
    $ \mu_{\theta,s}(\Sigma \cap \mathcal{S}^c) = 0.$
        Using Proposition \ref{idsac}, Lemma \ref{iml}, and the IDS definition, we bound
        \[
        \mu_\theta(\Sigma \cap \mathcal{S}) \leq \mu_\theta(\mathcal{S}) = 0 \quad \text{for a.e. } \theta.
        \]

   Consequently, $\mu_\theta(\Sigma) = \mu_{\theta,ac}(\Sigma)$ for a.e. $\theta$. Thus, $H_{V,K\alpha,\theta}$ has purely absolutely continuous spectrum. By Proposition \ref{uni}, the same holds for $L_{\varepsilon v,w,\alpha,\theta}$. 
  
\end{proof}

\section{All phases absolutely continuous spectrum}\label{all}
In this section, we assume that $ H_{v,\alpha,\theta} $ is a type I operator, where $ v $ is a trigonometric polynomial of degree $ m $. Let $ L_{2\cos,w,\alpha,\theta} $ denote its dual operator, and let $ (\alpha,A_E) $ represent the corresponding Schr\"odinger cocycle, with $ {L}_i(E) $ denoting its Lyapunov exponents.

\subsection{Monotonicity argument}

The key to the entire proof is the following monotonicity argument:

\begin{prop} \label{cr} Let $ E \in \Sigma $ and $ \omega(E) = 1, L(E) > 0 $. Then, $ (\alpha, A_E) $ is partially hyperbolic with a two-dimensional center. There exists a neighborhood $ \I $ of $ E $ and $ \mathcal{B}_E(\theta) \in C^\omega(\I \times \mathbb{T}, \mathrm{HSP}(2m, \mathbb{C})) $ such that \begin{equation}\label{aaa18}
\mathcal{B}_E(\theta + \alpha)^{-1} A_E^{(2)}(\theta) \mathcal{B}_E(\theta) = \diag\{H_E(\theta), (H_E(\theta)^*)^{-1}\} \diamond \phi_E(\theta) C_E(\theta),
\end{equation} where $ H_E(\theta) \in C^\omega(\I \times \mathbb{T}, \mathrm{GL}(m-1, \mathbb{C})) $, $ \phi_E \in C^\omega(\I \times 2\mathbb{T}, \mathbb{T}) $, and $ C_E(\theta) \in C^{\omega}(\I\times 2\mathbb{T}, \mathrm{SL}(2, \mathbb{R})) $ is homotopic to the identity. Moreover, we have the following:
\begin{enumerate}
\item $(\alpha, C_E)$ is subcritical, i.e. for sufficiently small \( y \), we have   \( L_1(C_E(\cdot + iy)) = L_1(C_E) = 0 \). 
\item $(\alpha, C_E)$ is monotonic with respect to \( E \).
\item \( L_1(C_{E + i\varepsilon}) \geq \kappa |\varepsilon| \) for sufficiently small \( \varepsilon \).
\end{enumerate} \end{prop}

The starting point of Proposition \ref{cr} is the recently developed Quantitative Avila's Global Theory \cite{GJYZ}, which establishes the connection between $ H_{v,\alpha,\theta} $ and its dual cocycle $(\alpha, A_E)$ as follows:

\begin{lemma}\cite{GJ,GJYZ}\label{censub}
Let \( E \in \Sigma \) and \( \omega(E) = 1, L(E) > 0 \). Then, \( (\alpha, A_E) \) is partially hyperbolic with a two-dimensional center. Furthermore, \( L_m(E + iy) = L_m(E) = 0 \) for any \( |y| < \frac{L(E)}{2\pi} \).
\end{lemma}

The second ingredient is that the Schrödinger cocycle is monotonic with respect to $ E \in \mathbb{R} $, a fact first observed by Avila \cite{A09}. We now generalize this observation to the Schrödinger operator on a strip:

\begin{lemma}\label{monsc}
The Hermitian symplectic cocycle \( A_E^{(2)}(\theta) = A_E(\theta + \alpha) \circ A_E(\omega) \) (with respect to the Hermitian symplectic structure \( \psi \)) is monotonic with respect to the parameter (energy) \( E \in \mathbb{R} \).
\end{lemma}
	\begin{proof}In fact, we can prove stronger result:  for any $v=\begin{pmatrix}v_1\\ v_2\end{pmatrix} \in \CC^{2d}\backslash\{0\}$, $\psi(\frac{d}{dE}A_E^{(2)}(\theta)\cdot v, A_E^{(2)}(\theta)\cdot v)<0$. Using the chain rule and the fact that $A_E(\theta),A_E(\theta+\alpha)$ preserving $\psi$, by a straight forward calculation we get that 
		$$\psi(\frac{d}{dE}A_E^{(2)}(\theta)\cdot v, A_E^{(2)}(\theta)\cdot v)=-\|v_1^2\|-\|w_1\|^2$$where $w=\begin{pmatrix}w_1\\w_2\end{pmatrix}=A_E(\theta)\cdot v$. If $\|v_1\|=0$, then by definition of $A_E$ and invertibility of $C$ we know $\|w_1\| \neq 0$ unless $v=0$, hence we get the proof.
	\end{proof}

Once we have these, we can finish the proof based on ideas developed in Section \ref{idea2}.\\

\textbf{Proof of Proposition \ref{cr}:}

Lemma \ref{monsc} asserts the monotonicity of \(A_E^{(2)}\). Leveraging this, we apply Theorem \ref{mon1} to establish the monotonicity of the center bundle. The key insight here is that, in the one-frequency quasi-periodic setting (where the base is \(\mathbb{T}\)), the bundles \(E_{A_E}^{*}\) are actually trivial. Lemma \ref{censub} then enables us to reduce the center dynamics to a monotonic \(\SL(2,\mathbb{R})\) cocycle. We first formalize the results on trivial bundles as follows:

\begin{lemma}\label{change}  
Suppose \(A_t \in C^\omega(\I\times \mathbb{T}, \mathrm{HSP}(2d))\) and \((\alpha, A_t)\) is partially hyperbolic with \(\dim E_t^c = 2d - 2n\).  
\begin{enumerate}  
\item There exist vectors \(v_{t,\pm (n+1)}, v_{t,\pm (n+2)}, \dots, v_{t,\pm d} \in C^\omega(\I \times \mathbb{T}, \mathbb{C}^{2d})\) that form a canonical basis for \(E_t^c(\theta)\).  

\item There exist vectors \(v_{t,\pm 1}, \dots, v_{t,\pm n} \in C^\omega(\I\times \mathbb{T},\mathbb{C}^{2d})\) that form a canonical basis for \(E_t^u \oplus E_t^s\). Moreover,  
\[  
E_t^u(\theta) = \operatorname{span}_{\mathbb{C}}\{v_{t,1}(\theta), v_{t,2}(\theta), \dots, v_{t,n}(\theta)\}, \quad E_t^s(\theta) = \operatorname{span}_{\mathbb{C}}\{v_{t,-1}(\theta), v_{t,-2}(\theta), \dots, v_{t,-n}(\theta)\}.  
\]  
\end{enumerate}  
\end{lemma}

\begin{rem}
    This result was essentially obtained in  \cite[Proposition 5.2]{WXZ}. Here we give an independent proof in the sprit of Lemma \ref{ht}. One can consult Appendix \ref{change-proof} for details. 
\end{rem}

By Lemmas \ref{censub} and \ref{change}, there exists $\tilde{v}_{E,\pm m} \in C^\omega(\I \times \mathbb{T}, \mathbb{C}^{2d})$ that forms a canonical basis for $E_{A_E}^c(\theta)$, and $v_{t,\pm 1}, \dots, v_{t,\pm (m-1)} \in C^\omega(\I \times \mathbb{T}, \mathbb{C}^{2d})$ that forms a canonical basis for $E_{A_E}^u \oplus E_{A_E}^s$. Notably, by Lemma \ref{monsc}, $A_E^{(2)}$ is monotonic; thus, one can apply Theorem \ref{mon1} to obtain a Hermitian-symplectic map $B_E$ such that
$$
B_E(\theta+\alpha)^{-1} \circ A_E^2(\theta) \circ B_E(\theta)
$$
is monotonic with respect to $E$. Let $v_{E,\pm m} = B_E \cdot \tilde{v}_{E,\pm m}$, and define
$$
\mathcal{B}_E(\theta) = \begin{pmatrix}
v_{E, 1}, \dots, v_{E, (m-1)}, v_{E,m}, v_{E, -1}, \dots, v_{E, -(m-1)}, v_{E,-m}
\end{pmatrix}.
$$
Then, there exist $H_E(\theta) \in C^\omega(\I \times \mathbb{T}, \mathrm{GL}(m-1,\mathbb{C}))$ and $\mathcal{C}_E(\theta) \in C^{\omega}(\I \times \mathbb{T}, \mathrm{HSP}(2))$ such that
$$
\mathcal{B}_E(\theta+\alpha)^{-1} A_E^2(\theta) \mathcal{B}_E(\theta) = \diag\{H_E(\theta), (H_E(\theta)^*)^{-1}\} \diamond \mathcal{C}_E(\theta),
$$
where $\mathcal{C}_E(\theta)$ is monotonic with respect to $E$.
Define $\phi_E(\theta) = \sqrt{\det \mathcal{C}_E(\theta)}$ and $C_E(\theta) = \frac{\mathcal{C}_E(\theta)}{\phi_E(\theta)}$. By direct computation and Lemma \ref{sqrt}, we have $\phi_E \in C^\omega(\I \times 2\mathbb{T}, \mathbb{T})$ and $C_E(\theta) \in C^{\omega}(\I \times 2\mathbb{T}, \mathrm{SL}(2,\mathbb{R}))$. Thus, we obtain equation \eqref{aaa18}.

By Lemma \ref{censub}, we have
\[
L\left(A_E(\theta+\alpha+iy)A_E(\theta+iy)\right) = L\left(\phi_E(\theta+iy)C_E(\theta+iy)\right) = \int \log \left| \phi_E(\theta+iy) \right| \, d\theta + L\left(C_E(\theta+iy)\right).
\]
Thus, \( L(C_E(\theta)) = 0 \). By Jensen's formula, the integral \( \int \log \left| \phi_E(\theta+iy) \right| \, d\theta \) is a convex function, and Avila's global theory \cite{avila} asserts that \( L(C_E(\theta+iy)) \) is also convex. Since \( L\left(A_E(\theta+\alpha+iy)A_E(\theta+iy)\right) = 0 \), it follows that \( L(C_E(\theta+iy)) = 0 \), which proves the statement (1).

Given that any vector $ v \in \mathbb{R}^2 $ is an isotropic vector, we have the following relationship:
$$
\begin{aligned}
\psi(\mathcal{C}_E'v, \mathcal{C}_Ev) =
\psi(\phi_E' C_Ev, \phi_E C_Ev) + \psi(\phi_E C_E'v, \phi_E C_Ev) = \psi(C_E'v, C_Ev),
\end{aligned}
$$
where we have utilized the property that $ v $ is an isotropic vector and $ \phi_E \in \T $. Consequently, by definition, $ C_E $ is monotonic with respect to $ E $, and $(2)$ follows.

The statement (3) follows directly from the subsequent lemma.
\begin{lemma}\label{lem: complx LE} \cite{avila-kam,AK}
Let \( \I \subset \mathbb{R} \) be an interval. Suppose \( E \mapsto D_E \) is a one-parameter analytic quasi-periodic cocycle \( \I \to C^\omega(\mathbb{T}^1, \mathrm{SL}(2,\mathbb{R})) \) that is monotonic in \( E \in \I \). Then for any compact interval \( \overline{\I} \subset \I \), there exists \( \kappa(\overline{\I}) > 0 \) such that for any small \( \epsilon > 0 \) and any \( E' \in \overline{I} \), the Lyapunov exponent satisfies \( L(D_{E' + i\varepsilon}) \geq \kappa|\varepsilon| \).
\end{lemma}

\begin{rem}
The proof is essentially established in \cite{AK}; see also the proof of \cite[Lemma 7]{avila-kam}. We provide the proof here for completeness. Detailed arguments can be found in Appendix \ref{mono-lya}.
\end{rem}

\qed

In this context, note that a given matrix-valued function $ D: \mathbb{T} \rightarrow \mathrm{SL}(2, \mathbb{R}) $ can also be regarded as a function on $ 2\mathbb{T} $. In this case, the rotation numbers of the respective cocycles are related by a factor of 2. For the sake of simplicity, we will not distinguish between $ \mathbb{T} $ and $ 2\mathbb{T} $ from this point onward.

\subsection{Regularity of the rotation number}
Another ingredient is the regularity of the rotation number $\rho(C_E)$. Since $C_E(\theta)$ is subcritical by Proposition \ref{cr}, it is homotopic to the identity, allowing us to define the rotation number $\rho(C_E)$ unambiguously. Indeed, we have the following simple observation:

\begin{lemma}\label{rho-c}
For any \(E \in \Sigma\), there exists a constant \(k_E \in \mathbb{Z}\) such that \(2\rho(C_E) = m(1 - N(E)) + \frac{k_E \alpha}{2} \mod \mathbb{Z}\).
\end{lemma}

\begin{proof}
By observing that \(\rho(\alpha, A_E^2) = 2\rho(\alpha, A_E)\), this follows directly from Proposition \ref{rotids}, Proposition \ref{rho}, and Proposition \ref{cr}.
\end{proof}

By \cite[Theorem 2.3]{GJ} and \cite[Theorem 1.3]{HS}, we obtain the following:

\begin{lemma}\cite{GJ,HS}\label{holder}
    If $E\in\Sigma$, then for sufficiently small $\epsilon$,
    $N(E+\epsilon)-N(E-\epsilon)\leq \epsilon^{1/2}.$
\end{lemma}

Our primary observation is the following lower bound estimate of the rotation number, which generalizes the result presented in \cite[Lemma 3.11]{Avila: ac}:

\begin{prop}\label{holderl}
	If $E\in\Sigma$, then for sufficiently small $\epsilon$, we have $$\rho(C_{E+\epsilon})-\rho(C_{E-\epsilon})\geq c\epsilon^{\frac{3}{2}}.$$
\end{prop}

\begin{proof} 
Since $\rho(C_E)$ and $N(E)$ are  continuous  with respect to $E$, hence $k_E$ is independent of $E$ in the small neighborhood of $E$.  
Therefore, we only need to estimate the lower bound of $N(E).$
	Let $\delta=c \epsilon^{\frac{3}{2}}$,
 then by Thouless formula (Theorem \ref{thouless}), we have
	$$
	L^{m}(E+i\delta)-L^{m}(E)=\int \frac{1}{2} \ln (1+\frac{\delta^{2}}{\left|E-E^{\prime}\right|^{2}})\ d N\left(E^{\prime}\right).
	$$
	We split the integral into four parts: $I_{1}=\int_{\left|E-E^{\prime}\right|\geq1}$, $ I_{2}=\int_{\epsilon\le\left|E-E^{\prime}\right|<1}$, $I_{3}=\int_{\epsilon^{4}\le\left|E-E^{\prime}\right|<\epsilon}$ and $I_{4}=\int_{\left|E-E^{\prime}\right|<\epsilon^{4}}$.
	
	For sufficiently small $\epsilon>0$, by Lemma \ref{holder}, we have $I_1 <c^2\epsilon^3,$
	and
	\begin{equation*}
		\begin{split}
			I_{4}&=\sum_{k \geq 4} \int_{\epsilon^{k}>\left|E-E^{\prime}\right|\ge\epsilon^{k+1}}  \frac{1}{2}\ln (1+\frac{\delta^{2}}{\left|E-E^{\prime}\right|^{2}})\ d N\left(E^{\prime}\right) 
			\leq \frac{1}{2}\sum_{k \geq 4} \epsilon^{\frac{k}{2}} \ln (1+c^{2} \epsilon^{1-2 k})\leq \epsilon^{\frac{7}{4}}.
		\end{split}
	\end{equation*}
	We also have the estimate
	$$
	\begin{aligned} I_{2} & \leq \sum_{k=0}^{m} \int_{e^{-k-1}\le|E-E'|<e^{-k}} \frac{1}{2}\ln (1+\frac{\delta^{2}}{\left|E-E^{\prime}\right|^{2}})\ d N\left(E^{\prime}\right)\le\sum_{k=0}^{m}\frac{1}{2}e^{-\frac{k}{\kappa}}\delta^2e^{2k+2}\leq Cc^2\delta,\end{aligned}
	$$
	with $m=[-\ln \epsilon]$. 
	It follows that
	$$I_3\ge L^{m-1}(E+i\delta)+L_m(E+i\delta)-L^{m-1}(E)-c\delta.$$
    Noting that $ I_{3} \leq C(N(E+\epsilon) - N(E-\epsilon)) \ln \epsilon^{-1} $, it suffices to demonstrate that $ I_3 \geq c \delta $. Since the constant $ c $ is consistent with our choice of $ \delta $, we can adjust $ \delta $ as needed.

    By Proposition \ref{cr}, we have \( L(C_{E+i\delta}) \geq \frac{\kappa\delta}{2} \) and 
\[
L_m(E+i\delta) = L(C_{E+i\delta}) - \int \log \left| \phi_{E+i\delta}(\theta) \right| d\theta =: L(C_{E+i\delta}) + \tilde{\phi}(E+i\delta).
\]
By Lemma \ref{sqrt}, there exists an analytic function \( q(E,\theta) \in C^\omega(\I \times \mathbb{R}, \mathbb{C}) \) such that \( \phi_{E+i\delta}(\theta) = e^{\frac{1}{2}q(E+i\delta,\theta)} \). Therefore, \( \tilde{\phi}(E+i\delta) = -\frac{1}{2} \int \Re q(E+i\delta,\theta) d\theta \), which is a harmonic function.
On the other hand, since \( L^{m-1}(E) = L^{m-1}\left( A_E|_{E_{A_E}^u} \right) \), it follows that \( L^{m-1}(E) \) is also a harmonic function in a neighborhood of \( E \) \cite{positive,AJS}.
 
In particular, \( \frac{\partial (L^{m-1} + \tilde{\phi})(z)}{\partial \operatorname{Im} z} \) is a continuous function in a neighborhood of \( E \). We divide the proof into three cases:
 
\smallskip
\textbf{Case I:} If \( \frac{\partial (L^{m-1} + \tilde{\phi})(z)}{\partial \operatorname{Im} z} > 0 \), there exists a neighborhood of \( E \) such that \( L^{m-1}(E+i\delta) + \tilde{\phi}({E+i\delta}) - L^{m-1}(E) > 0 \), which implies \( I_3 \geq c\delta \).
 
\smallskip
\textbf{Case II:} If \( \frac{\partial (L^{m-1} + \tilde{\phi})(z)}{\partial \operatorname{Im} z} < 0 \), we only need to consider \( L^{m}(E-i\delta) \) instead of \( L^{m}(E+i\delta) \).
 
\smallskip
\textbf{Case III:} If \( \frac{\partial (L^{m-1} + \tilde{\phi})(z)}{\partial \operatorname{Im} z} = 0 \). Since \( \frac{\partial (L^{m-1} + \tilde{\phi})(z)}{\partial \operatorname{Im} z} \) is continuous in a neighborhood of \( E \), there exists a neighborhood of \( E \) such that \( \left| \frac{\partial (L^{m-1} + \tilde{\phi})(z)}{\partial \operatorname{Im} z} \right| < \frac{\kappa}{4} \). Hence,
\[
\left| L^{m-1}(E+i\delta) + \tilde{\phi}(E+i\delta) - L^{m-1}(E) \right| \leq \frac{\kappa}{4}\delta.
\]
It follows that
\[
I_3 \geq \frac{\kappa}{2}\delta - \left| L^{m-1}(E+i\delta) + \tilde{\phi}(E+i\delta) - L^{m-1}(E) \right| - c\delta \geq c\delta.
\]
\end{proof}

\subsection{Quantitative almost reducibility}
The third  ingredient is almost reducibility.  Recall that $(\alpha,A_1)$ is conjugated to $(\alpha,A_2)$, if there exists $B\in C^\omega(\T,\mathrm{PSL}(2,\R))$
such that 
$$B^{-1}(x+\alpha)A_1(x)B(x)=A_2(x).$$ 
Then $(\alpha,A)$ is almost reducible if the closure of its analytic conjugate class contains the constant. The key to us is Avila's solution to his almost reducibility conjecture: 

\begin{theorem}[\cite{avila-kam,avila2010almost}] \label{arc-conjecture}
	Any subcritical $(\alpha,A)$ with $\alpha\in \R\backslash\Q$, $A\in C^\omega(\T,\SL(2,\R))$, is almost reducible.
\end{theorem}

Noting that by Proposition \ref{cr}, $ (\alpha, C_E(\theta)) $ is subcritical, we can conclude that, according to Theorem \ref{arc-conjecture}, $ (\alpha, C_E(\theta)) $ is almost reducible. Furthermore, a straightforward continuity argument enables us to achieve a uniform global-to-local reduction:

\begin{lemma}\label{global-local}\cite{GYZ,WXYZZ}
Let   $\alpha\in \R\backslash \Q$, \( \I \subset \mathbb{R} \) being an interval. For any $\epsilon_0>0,$ there exist
$\bar{h}=\bar{h}(\alpha)>0$ and $\Gamma=\Gamma(\alpha,\epsilon_0)>0$ such that for any $
E\in  \I$,
there exist $\Phi_{E}\in C^{\omega}_{\bar{h}}(\mathbb{T},PSL(2,\mathbb{R}))$ with $\|\Phi_{E}\|_{\bar{h}}<\Gamma$ such that
\begin{equation*}
    \Phi_{E}(\theta+ \alpha)^{-1}C_{E}(\theta)\Phi_{E}(\theta)=R_{\Phi_{E}}e^{f_{E}(\theta)}
\end{equation*}
with $\left\|f_{E}\right\|_{\bar{h}}<\epsilon_0,$ $\left|\operatorname{deg} \Phi_{E}\right| \leq C|\ln \Gamma|$ for some constant $C=$
$C(V,\alpha)>0 .$ 
\end{lemma}

Once having Lemma \ref{global-local}, one can apply the KAM scheme to get precise control of the growth of the cocycles in the resonant sets.  We inductively give the parameters, 
for any $\bar{h}>\tilde{h}>0$, $\gamma>0,\sigma>0$, define 
\begin{align*}
 h_0=\bar{h}, \qquad  \epsilon_0 \leq D_0(\frac{\gamma}{\kappa^{\sigma}},\sigma) (\frac{\bar{h}-\tilde{h}}{8})^{C_0\sigma},\end{align*}
 where $D_{0} = D_0(\gamma,\sigma)$ and $C_{0}$ are numerical constant,  and define
$$
\epsilon_j= \epsilon_0^{2^j}, \quad h_j-h_{j+1}=\frac{\bar{h}-\frac{\bar{h}+\tilde{h}}{2}}{4^{j+1}}, \quad N_j=\frac{2|\ln\epsilon_j|}{h_j-h_{j+1}}.
$$
Then we have the following:
\begin{proposition}\label{local kam}\cite{LYZZ, WXYZZ}
Let $\alpha\in DC( \gamma, \sigma)$. Then there exists $B_{j} \in C_{h_{j}}^{\omega}\left(\mathbb{T}, PSL(2, \mathbb{R})\right)$ with
$|\deg B_j |  \leq 2 N_{j-1}$,
such that
$$B_{j}^{-1}(\theta+ \alpha) R_{\Phi_{E}}e^{f_{E}(\theta)} B_{j}(\theta)=A_{j}(E) e^{f_{j}(\theta)},$$
with  estimates  $
\left\|B_{j}\right\|_{0} \leq |\ln\epsilon_{j-1}|^{4\sigma}$, $\left\|f_{j}\right\|_{h_{j}} \leq \epsilon_{j}.$ 
Moreover,  for any $0<|n| \leq N_{j-1}$, denote 
\begin{equation*}
\begin{split}
\Lambda_{j}(n)=\left\{E\in \mathcal{I}(\delta_0):\|2 \rho(\alpha, A_{j-1}(E))- \langle n, \alpha \rangle\|_{\T}< \epsilon_{j-1}^{\frac{1}{15}}\right\}.
\end{split}
\end{equation*}
If 
 $E\in K_j:= \cup_{|n|=1}^{N_{j-1}}\Lambda_{j}(n)$, then then there exists $\tilde{n}_j \in \mathbb{Z}$ with $0<|\tilde{n}_j| <2N_{j-1}$ such that
\begin{equation*}
\| \ 2\rho(C_E)+\langle \tilde{n}_j, \alpha\rangle\|_{\mathbb{R} / \mathbb{Z}} \leqslant 2 \epsilon_{j-1}^{\frac{1}{15}}.
\end{equation*}
Moreover,  we have 
\begin{equation*}
\sup _{0 \leq s \leq \epsilon_{j-1}^{-\frac{1}{8}}}\|(C_{E})_s\|_{0} \leq 4\Gamma^2 |\ln\epsilon_{j-1}|^{8\sigma}.
\end{equation*}
\end{proposition}

\subsection{Proof of Theorem \ref{type1ac}:}
Within these three main ingredients, we are ready to finish the proof of Theorem \ref{type1ac}. Fix $\theta \in \mathbb{T}$. Let $\mathcal{G}_\theta = \mathcal{G}_\theta^u \cap \mathcal{G}_\theta^s$, where $\mathcal{G}_\theta^u$ and $\mathcal{G}_\theta^s$ are defined in Section \ref{ver-sch}. By Corollary \ref{JL}, for any $E \in \mathcal{G}_\theta$, there exists $\delta_E > 0$ such that for any $E' \in B(E, \frac{\delta_E}{5})$, we have \begin{equation*}
\mu_\theta(E' - \epsilon, E' + \epsilon) \leq \epsilon C(E, \theta) \sup_{|s| \leq 3 \epsilon^{-1}} \|(A_E)_s(\theta)|_{E_{A_{E'}}^c}\|^{42}.
\end{equation*} By the Vitali covering lemma, there exist $E_i$ and $\delta_i$ such that $\mathcal{G}_\theta \subset \bigcup_i B(E_i, \delta_i)$. On the other hand, by Lemma \ref{eig} and Proposition \ref{apoint}, we have $\mu_{\theta}(\Sigma_{\theta} \backslash \mathcal{G}_\theta) = 0$. Therefore, it suffices to prove that $\mu_{\theta,s}(B(E_i, \delta_i)) = 0$ for any $i$.

In the following, we fix $i$ and consider the interval $\mathcal{I}(\delta_0) = [E_i - \delta_i + \delta_0, E_i + \delta_i - \delta_0]$ for any sufficiently small $\delta_0$. We need to show that $\mu_{\theta,s}(\mathcal{I}(\delta_0)) = 0$ for any sufficiently small $\delta_0$.
We first  estimate the spectral measure using the center cocycle $C_E(\theta).$ 

\begin{lemma}\label{inej}
    For any \( E \in B(E_i, \delta_i) \), we have 
    \[
    \mu_\theta(E - \epsilon, E + \epsilon) \leq \epsilon \, C(E_i, \theta) \sup_{|s| \leq 3 \epsilon^{-1}} \|(C_E)_s(\theta)\|^{42}.
    \]
\end{lemma}

\begin{proof}
  Since $$\mathcal{B}_E(\theta)=\begin{pmatrix}
    v_{E, 1}, \dots, v_{E, (m-1)},v_{E,m},v_{E, -1}, \dots, v_{E, -(m-1)},v_{E,-m}
\end{pmatrix}$$
    Let \( \psi \in E_{A_E}^c(\theta) \) with \( \|\psi\| = 1 \). Decompose \( \psi = c_1 v_{E,m}(\theta) + c_2 v_{E,-m}(\theta) \).  
    This means 
    \[
    \mathcal{B}_E(\theta) \begin{pmatrix} 0 & \cdots & 0 & c_1 & 0 & \cdots & 0 & c_2 \end{pmatrix}^\top = \psi,
    \]
    from which it follows that 
    \[
    \sqrt{|c_1|^2 + |c_2|^2} \leq \|\mathcal{B}_E^{-1}\|.
    \]
    Then we have:
    \[
    \begin{aligned}
        \|(A_E)_s(\theta)|_{E_{A_E}^c} \psi\| 
        &= \left\| \left( u_E(\theta + s\alpha), v_E(\theta + s\alpha) \right) (C_E)_s(\theta) \begin{pmatrix} c_1 \\ c_2 \end{pmatrix} \right\| \\
        &\leq \|\mathcal{B}_E^{-1}\| \, \|\mathcal{B}_E\| \, \|(C_E)_s(\theta)\|.
    \end{aligned}
    \]
    The uniform bound follows from the analyticity of \( \mathcal{B}_E \) and \( \mathcal{B}_E^{-1} \).
\end{proof}

Let $\mathcal{B}$ be the set of $E\in  \mathcal{I}(\delta_0) $ such that the cocycle $(\alpha, C_{E})$ is bounded. 
 By Lemma \ref{inej}, it is enough to prove that
$\mu_{\theta}( \mathcal{I}(\delta_0)\backslash\mathcal{B})=0$.
 Let $\mathcal{R}$ be the set of $E\in  \mathcal{I}(\delta_0) $ such that $( \alpha,C_E)$ is reducible, then $\mathcal{R}\backslash\mathcal{B}$ only contains $E$ for which $(\alpha,C_E)$ is analytically reducible to a constant parabolic cocycle. Therefore $\mathcal{R}\backslash\mathcal{B}$
  is countable by famous result of Eliasson \cite{Eli92}. By Proposition \ref{apoint}, there are no eigenvalues in $\mathcal{R}$ and $\mu_{\theta}(\mathcal{R}\backslash\mathcal{B})=0$.  Therefore, it is enough to show that for sufficiently small $\delta_{0}>0$, 
$\mu_{\theta}( \mathcal{I}(\delta_0)\backslash\mathcal{R})=0.$
Note that $\mathcal{I}(\delta_0)\backslash\mathcal{R}\subset\limsup K_{m}$, by Borel-Cantelli Lemma, we only need to prove $\sum_m \mu_{\theta}(\overline{K}_{m})<\infty$. 

Let $J_{m}(E)$  be an open $\epsilon_{m-1}^{\frac{2}{45}}$  neighborhood of $E\in K_m$. By Proposition \ref{inej} and Proposition \ref{local kam}, we have
\begin{equation*}
\begin{split}
\mu_{\theta}(J_m(E))&\le \sup_{0\le s\le\epsilon_{m-1}^{-\frac{2}{45}}}||(C_E)_s||_0^{42}|J_m(E)|\\&\le\sup_{0\le s\le\epsilon_{m-1}^{-\frac{1}{8}}}||(C_E)_s||_0^{42}|J_m(E)|\le C|\ln\epsilon_{m-1}|^{336\sigma}\epsilon_{m-1}^{\frac{2}{45}}.
\end{split}
\end{equation*}
Let $\cup_{l=0}^rJ_m(E_l)$ be a finite subcover of $\overline{K}_m$. By refining this subcover, we can assume that every $E\in \mathcal{I}_\delta$ is contained in at most two different $J_m(E_l)$. 

On the other hand, by Proposition \ref{local kam}, if $E\in K_m$, then
$||  2\rho(C_E)+\left<n,  \alpha\right>||_{\mathbb{R/Z}}\le 2\epsilon_{m-1}^{\frac{1}{15}},$
for some $|n|<2N_{m-1}$.
This shows that $ \rho({K}_{m})$ can be covered by $2N_{m-1}$  intervals $I_{s}$ of length  $2 \epsilon_{m-1}^{\frac{1}{15}}$.
By Proposition  \ref{holderl},
$\rho(J_m(E))\ge c|J_m(E)|^{\frac{3}{2}},$
thus  by our selection $|I_{s}|\leq \frac{1}{c}| \rho(J_{m}(E))|$   for any $s$
and $E\in{K}_{m}$, there are at most $2([\frac{1}{c}]+1)+4$ intervals $J_m(E_l)$ such that $ \rho(J_m(E_l))$  intersects $I_{s}$.
We conclude that there are at most $2(2([\frac{1}{c}]+1)+4)N_{m-1}$
intervals $J_m(E_l)$ to cover $K_m$. Then 
\begin{equation*}
\mu_{\theta}(\overline{K}_{m})\leq \sum_{j=0}^{r}\mu(J_{m}(E_{j}))\le CN_{m-1}|\ln\epsilon_{m-1}|^{
336\sigma}\epsilon_{m-1}^{\frac{2}{45}}< \epsilon_{m-1}^{\frac{1}{45}},
\end{equation*}
which gives $\sum_m \mu_{\theta}(\overline{K}_{m})<\infty$. 
\qed
\subsection{Proof of Corollary \ref{pamo}}
The dual operator of $ \tilde{L}_{w,\alpha,\theta} $ is given by $ \lambda H_{\frac{\varepsilon w + 2\cos}{\lambda}, \alpha, x} $. We may assume that $ \varepsilon $ is sufficiently small such that $ \Sigma\left(H_{\frac{\varepsilon w + 2\cos}{\lambda}, \alpha, x}\right) \subset [-2 - \frac{2}{\lambda} - 1, 2 + \frac{2}{\lambda} + 1] =: J $. Let $ L(\varepsilon, E) $ denote the Lyapunov exponent and $ \omega(\varepsilon, E) $ represent the acceleration of the associated cocycle. Since $ L(0, E) \geq \log\frac{1}{\lambda} $ and $ \omega(0, E) \leq 1 $ for all $ E \in \mathbb{R} $, it follows from Lemma \ref{Lecon} and the upper semicontinuity of the acceleration \cite{avila} that there exists $ \varepsilon_0 > 0 $ such that, for all $ |\varepsilon| < \varepsilon_0 $, we have $ L(\varepsilon, E) > 0 $ and $ \omega(\varepsilon, E) \leq 1 $ on $ J $. Therefore, Theorem \ref{type1ac} can be applied to derive the conclusion.

\section{Absolutely continuous spectrum for infinite-range operator}\label{proof-infinite}

\subsection{Pure point spectrum of dual operator}
In this section, we prove the dual operator of ${L}_{\varepsilon v,w,\alpha,\theta}$ has $\ell^1(\Z^d)$ eigenfunction, the proof is derived from Eliasson's work \cite{E}. 
Let 
$$
D: \ell^2(\Z^d ) \times \T \rightarrow \ell^2(\Z^d )
$$
which are covariant with respect to an $\Z^d $-action
$$
T: \Z^d  \times \T \rightarrow \T, T(n,x)=x+\langle n,\alpha\rangle.
$$
A covariant matrix $D(x)$ is pure point if, for almost all $x$, there exists an eigenvector $q(x)$ such that $\left\{q^a(x)=: \tau_a q\left(T_a(x)\right): a \in \Z^d \right\}$ is a basis for $\ell^2(\Z^d )$. $\Omega(x) \subset \Z^d $ is a block for $q(x)$ if
$$
\Omega(x) \supset\left\{a \in \Z^d :\langle q(x),\delta_a\rangle\neq0\right\} .
$$
A partition of a manifold $\T$ is a (locally finite) collection of open subsets-pieces-$\mathcal{P}$ such that $\cup_{Y \in \mathcal{P}} Y$ is of full measure in $\T$. If $\mathcal{P}$ and $\mathcal{Q}$ are two such decomposition and if $T$ is a homeomorphism on $\T$, then
$$
\left\{\begin{array}{l}
\mathcal{P} \vee \mathcal{Q}=\{Y \cap Z: Y \in \mathcal{P}, Z \in \mathcal{Q}\}, \\
T(\mathcal{P})=\{T(Y): Y \in \mathcal{P}\}.
\end{array}\right.
$$
A function is said to be smooth on $\mathcal{P}$ if it is smooth on each piece of $\mathcal{P}$.

\begin{definition}\cite{E} \label{normal}
    We say that a covariant matrix $D$ is pure point with a eigenvector $q(x)$, corresponding eigenvalue $E(x)$ and block $\Omega(x)$ is on normal form 
$$D \in \mathcal{N} \mathcal{F}(\alpha, \ldots, \rho ; \Omega, \mathcal{P}) \& \mathcal{T}(\sigma, s)$$
if the following holds:
\begin{enumerate}
    \item Exponential decay off the diagonal.
$
|D_a^b|_{\mathcal{C}^0} \leq \beta e^{-\alpha|b-a|}.
$
\item Smoothness. The components of $D$ are piecewise smooth and satisfy 
$
|D_a^b|_{C^k} \leq \beta e^{-\alpha|b-a|} \gamma^k, \quad \forall k \geq 1.
$
\item Block dimensions and block extensions.  For all $x \in \T$,
$
\Omega(x) \subset\{a:|a| \leq \lambda\}, 
\# \Omega(x) \leq \mu.
$

\item Block overlapping. For all $x \in \T,$
$
\# \bigcup_{\Omega^a(x) \cap \Omega(x) \neq 0} \Omega^a(x) \leq \mu.
$
\item Resonant block separation. For all $x \in \T$,
$
\left|E^a(x)-E(x)\right| \leq \rho \Longrightarrow  \Omega^a(x)=\Omega(x) \text { or } \underline{\operatorname{dist}}\left(\Omega^a(x), \Omega(x)\right) \geq \nu \geq 1.
$

\item Partition. There is a locally finite partition $\mathcal{P}$ such that
$D$  and  $\Omega$  are smooth on the partition  $\mathcal{P}.
$
By a matrix $D$ being smooth on a partition, we mean that each $D_a^b$ is smooth on, a neighborhood of the closure of, the pieces of
$
\vee\{T_c(\mathcal{P}):\left|c+\frac{b+a}{2}\right| \leq|\frac{b-a}{2}|\}.
$
By a set $\Omega$ being smooth on a partition, we mean that each characteristic function $\chi_{\Omega^b}(a)$ is smooth on, a neighborhood of the closure of, the pieces of the same partition.
\item Non-degenerate. Let $
u_{a, b}(x, y)=\operatorname{Res}\left(D_{\Omega^a(x+y)}(x+y), D_{\Omega^b(x)}(x)\right).
$ For all $x, y \in \T$ and for all $i$,
$$
\begin{gathered}
\max _{0 \leq k \leq J}|\frac{1}{(k!)^2 \gamma^k} \partial_{y_i}^k u_{a, b}(x, y)| \geq \sigma, \\
\max _{0 \leq k \leq J}|\frac{1}{(k!)^2 \gamma^k} \partial_{x_i}^k u_{a, b}(x, y)| \geq \sigma|y|^{\#(\Omega^a(x+y) \cap \Omega^b(x))},
\end{gathered}
$$
where $J=\# \Omega^a(x+y) \times \# \Omega^b(x) s$ and $s$ is a fixed parameter.
\end{enumerate}

\end{definition}

Choice of parameter
 $\alpha_j, \beta_j, \gamma_j, \lambda_j, \mu_j, \nu_j, \rho_j, \sigma_j, \varepsilon_j$ satisfying for all $j \geq 1,$
$$
\left\{\begin{array}{l}
\beta_{j+1}=\left(1+\sqrt{\varepsilon_j}\right) \beta_j,\gamma_{j+1}=\left(\frac{1}{C} \frac{\beta_j \mu_j^2}{\rho_j}\right)^{7 \mu_j^3} \gamma_j,\varepsilon_{j+1}=e^{-\frac{1}{2} \sqrt{(\nu_{j+1}-8 \lambda_j)(\alpha_j-\alpha_{j+1})}}, \\
\sigma_{j+1}=C^{\mu_{j+1}^4}(C \frac{\rho_j}{\beta_j \mu_j^2})^{14 s \mu_j^3 \mu_{j+1}^2(\frac{\sigma_j}{\beta_j})^{s \mu_{j+1}^4}}, \lambda_{j+1}=(\frac{\mu_{j+1}}{\mu_j})(\nu_{j+1}+2 \lambda_j),\mu_{j+1}=\tilde{p} 2^{2 s \mu_j^2},\\
\tilde{p}=16 \frac{\gamma_1}{\sigma_1}(4 \beta_1 \mu_1^2)^{\mu_1^2}(s \mu_1^2)^2(p+|\mathcal{P}_1(\lambda_1)|^{-1}),\nu_{j+1}=\frac{\mu_j}{2 \mu_{j+1}}(\frac{\kappa}{8 \beta_j \mu_{j+1}})^{\frac{1}{\tau}}(\frac{\sigma_j}{\rho_j})^{4 \tau s \mu_j^2}, \\
\alpha_j=\frac{1}{\lambda_j},\rho_j^{\left(\mu_1 \ldots \mu_{j+1}\right)^6}=\varepsilon_j. 
\end{array}\right.
$$
A direct calculation can show that ((5.20) in \cite{E}) \begin{equation}\label{estimate}
    \varepsilon_j\leq e^{-2j}\alpha_j^{2d}.
\end{equation}

We have the following conclusion:
\begin{proposition}\cite{E}\label{eli}
Let $D \in \mathcal{N} \mathcal{F}(\alpha, \ldots, \rho ; \Omega, \mathcal{P}) \& \mathcal{T}(\sigma, s)$ be covariant with respect to the quasi-periodic $\Z^d $-action, and assume that $D$ is truncated at distance $\nu$ from the diagonal. Let $F$ be a covariant matrix, smooth on $\mathcal{P}$ and
$$
\left|F_a^b\right|_{\mathcal{C}^k} \leq \varepsilon e^{-\alpha|b-a|} \gamma^k \quad \forall k \geq 0.
$$
Suppose $D$ and $F$ are Hermitian,
then there exists a constant $C$ depends only on $d , \kappa, \tau, s, \alpha, \beta$, $\gamma, \lambda, \mu, \nu, \rho, \sigma, \# \mathcal{P}-$ such if $\varepsilon \leq C$, then 

\begin{enumerate}
\item (Proposition 5(II) and the comment of Corollary 4)  There exists a matrix $U_j(x)$ such that
$$
U_j(x)^*(D_{j-1}(x)+F_{j-1}(x)) U_j(x)=D_j(x)+F_j(x), \quad \forall x \in X,
$$
$$
\left|(U_j-I)_a^b\right|_{\mathcal{C}^k} \leq \sqrt{\varepsilon_{j-1}} e^{-\frac{\alpha_{j-1}}{2}|b-a|}\left(\gamma_{j}\right)^k,\left|\left(F^{j}\right)_a^b\right|_{\mathcal{C}^k} \leq \varepsilon_j e^{-\alpha_j|b-a|}\left(\gamma_j\right)^k
$$
 for all $k \geq 0$.

\item (Proposition 8 (I), Proof of Proposition 6 (IV)) $D_j\in \mathcal{N} \mathcal{F}(\alpha_j, \ldots, \rho_j ; \Omega_j, \mathcal{P}_j) \& \mathcal{T}(\sigma_j, s)$. 
Moreover, for a.e. $x$, there exists $j(x)$, such that $\Omega_j(x)=\Omega_{j(x)}$ for all $j\geq j(x).$
    \item (Theorem 12)  There exits $D_{\infty}(x)$, such that $\|D_j(x)-D_{\infty}(x)\|_{\ell^2}\rightarrow0$. Moreover $D_{\infty}(x)$ is pure point for a.e. $x$.
\item (Proof of Proposition 6, Theorem 12) $U(x)=\lim_{j\rightarrow\infty}U_1(x)U_{2}(x)\cdots U_j(x)$ convergence in the $\ell^2$ norm and $U(x)$ is an unitary operator.

\item (Theorem 12)The limit $\lim _{j \rightarrow \infty} E_j(x)=E_{\infty}(x)$ is uniform, and it is a measurable function, satisfies for all subsets $Y$
$$
m(E_{\infty}^{-1}(Y))=0 \quad \text { if } \quad m(Y)=0.
$$

\end{enumerate}
\end{proposition}

We shall prove the general result, we assume that  $W(\cdot)$ is a Gevrey function, satisfies: there exists $s>0$, such that
	\begin{equation}\label{gevery}
	\begin{cases}
   |W|_{C^k}\leq \beta \gamma^k,
   \\\max _{0 \leqslant \nu \leqslant s}\left|\partial_x^\nu(W(\theta+x)-W(\theta))\right| \geqslant \xi>0, & \forall \theta,  \forall x, \\ \max _{0 \leqslant \nu \leqslant s}\left|\partial_\theta^\nu(W(\theta+x)-W(\theta))\right| \geqslant \xi\|x\|, & \forall \theta, \forall x .\end{cases}	\end{equation}

	\begin{theorem}\label{point}
	Suppose $\alpha\in DC$, $W(\cdot)$ is a Gevrey function that satisfies the assumption \eqref{gevery}. Let $(Hu)_n=\varepsilon\sum_{k_\in\Z^d}v_ku_{n+k}+W(x+\langle n,\alpha\rangle)u_n$, then there exists $\varepsilon_5(\alpha,V,W)$ and Borel measurable function $E_{\infty}(\cdot)$, such that for  $|\varepsilon|<\varepsilon_5$,  for a.e. $x$, $E_{\infty}(x+\langle k,\alpha\rangle)$ is an eigenvalue of $H$ for all $k\in\Z^d$, and corresponding eigenfunction is in $\ell^1(\Z^d).$ Moreover, for all subsets $Y$
$$
m(E_{\infty}^{-1}(Y))=0 \quad \text { if } \quad m(Y)=0.
$$
	\end{theorem}
\begin{proof}
    Let $D(x)=\diag\{W(x+\langle n,\alpha\rangle)\}_{n\in\Z^d}$ and $F=H(x)-D(x)$. The assumption of $W(x)$ implies $D\in \mathcal{N} \mathcal{F}(\alpha=1,\beta,\gamma,\lambda=1, \mu=1,\nu=1,\rho=1) \& \mathcal{T}(\frac{\sigma}{(s!)^2\gamma^s}, s).$ Hence, we can apply Proposition \ref{eli}.
    
By Proposition \ref{eli} (1), there exists $U_j(x)$, such that
$$U_j^*(x)\cdots U_2^*(x)U_1^*(x)(D(x)+F(x))U_1(x)U_2(x)\cdots U_j(x)={D}_{j}(x)+F_j(x).$$
And by Proposition \ref{eli} (2), $D_j\in \mathcal{N} \mathcal{F}(\alpha_j, \ldots, \rho_j ; \Omega_j, \mathcal{P}_j) \& \mathcal{T}(\sigma_j, s)$. Moreover, there exists full measure set $\mathcal{A}$, such that for any $x\in\mathcal{A}$, there exists $j(x)$, such that $\Omega_j(x)=\Omega_{j(x)}$ for all $j\geq j(x).$ By the Definition \ref{normal}, there exists $q_j(x)\in\C^{\Omega_{j(x)}}$ with $\|q_j(x)\|=1$ for all $j\geq j(x)$, and $$D_j(x)q_j(x)=E_j(x)q_j(x).$$
Since $\{q_j(x)\}_{j\geq j(x)}$ is a bounded sequence in finite-dimensional space, there must exist a convergence sequence, we also denote $\{q_j(x)\}$, and $q_\infty(x):=\lim q_j(x).$
According to Proposition \ref{eli} (3), we have $D_j(x)\rightarrow D_{\infty}(x)$. It follows that 
$$D_\infty(x)q_\infty(x)=E_\infty(x)q_\infty(x).$$
By Proposition \ref{eli} (4), we have 
$$U_{\infty}(x)^*(D(x)+F(x))U_\infty(x)=D_{\infty}(x).$$
Therefore, $U_\infty(x)q_{\infty}(x)$ is an eigenvector of $D(x)$ with eigenvalue $E_{\infty}(x)$.
Let $\tilde{\mathcal{A}}=\cap_{n\in\Z^d}(\mathcal{A}+\langle n,\alpha\rangle),$
we find $\{ \tau_{n}U_\infty(x+\langle n,\alpha\rangle)q_{\infty}(x+\langle n,\alpha\rangle)\}_{n\in\Z^d}$ is the eigenfunction of $D(x)$ with eigenvalue $\{E_{\infty}(x+\langle n,\alpha\rangle)\}_{n\in\Z^d}$ for $x\in\tilde{\mathcal{A}}.$ 
Again by Proposition \ref{eli} (1) 
$$
\begin{aligned}
&\quad\|U_j(x)-I\|_{l^1}\leq\sqrt{\varepsilon_{j-1}}\sum_{k}e^{-\frac{\alpha_{j-1}}{2}|k|}=\sqrt{\varepsilon_{j-1}}\sum_{t\geq0}\sum_{|k|=t}e^{-\frac{\alpha_{j-1}}{2}|t|}\\
&\leq\sqrt{\varepsilon_{j-1}}\bigg(\sum_{t\geq 1}C|t|^{d-1}e^{-\frac{\alpha_{j-1]}}{2}|t|}+1\bigg)
\leq\sqrt{\varepsilon_{j-1}}\frac{C}{\alpha_{j-1}^d}+\sqrt{\varepsilon_{j-1}}\leq  2e^{-{j-1}},
\end{aligned}$$
where we used \eqref{estimate}.
Therefore, $U(x)=\lim\prod_{k=1}^{j}U_k(x)$ convergence in $\ell^1$ norm. Consequently, all eigenfunctions will belong to $\ell^1(\Z^d).$ 
    
\end{proof}

\subsection{Proof of Theorem \ref{acqq}}
By Theorem~\ref{point}, for Lebesgue-almost every \( x \in \mathbb{T} \), there exist \(\ell^1(\mathbb{Z}^d)\) eigenfunctions \(\{\psi^n\}_{n \in \mathbb{Z}^d}\) with corresponding eigenvalues \(\{E(x + \langle n, \alpha \rangle)\}\). Define the Fourier series $
\psi_k(\theta) = \sum_{n \in \mathbb{Z}^d} (\psi^k)_n e^{i\langle n, \theta \rangle},
$
which satisfies \(\|\psi_k\|_\infty < \infty\) for all \(k \in \mathbb{Z}^d\). Let
\[
\phi(\theta, k)_n = \psi_k(\theta + n\alpha)e^{2\pi i x}.
\]
Then \(\phi(\theta, k)\) is a non-trivial bounded solution of \({L}_{\varepsilon v,w,\alpha,\theta}\) with energy \(E(x + \langle k, \alpha \rangle)\).
Define the set
\[
X = \bigcup_{k \in \mathbb{Z}^d} \big\{E(x + \langle k, \alpha \rangle) : x \in \tilde{\mathcal{A}} \big\}.
\]
Since \(E(\cdot)\) is Borel measurable and images of Borel sets under \(E(\cdot)\) are Lebesgue measurable (Lemma~\ref{leb}), \(X\) is Lebesgue measurable. If \(m(X) = 0\), this would contradict Theorem~\ref{point}; hence, \(m(X) > 0\). By the regularity of Lebesgue measure, there exists a \(F_\sigma\) set \(S \subset X\) with \(m(S) = m(X)\). Theorem~\ref{ac} then follows from Theorem~\ref{sub} and Proposition~\ref{np}.
\qed

\appendix
\section{Property of dominated splitting}\label{ds-con}

The following continuity result of dominated splitting is considered standard. As we did not find the exact formulation in the literature, we include the result and its proof for completeness. Readers can refer to \cite{CP} for the continuity proof and to Section 6 of \cite{AJS} for the holomorphicity.

\begin{lemma}Let $T:\Omega\to \Omega$ be a homeomorphism on a compact metric space $\Omega$ and $A_t:\Omega\to GL(\C,m)$ be a family of continuous linear cocycle over $T$ with parameter $\theta$ in a region  $U\subset \C$. If \begin{itemize}
    \item $t\mapsto A_t(\cdot)$ is in $C^0(U, C^0(\Omega,GL(\C,m))$;
    \item $t \mapsto A_t(\omega)$ is $C^l(l=1,\cdots,\infty,\omega) $ for every $\omega\in \Omega$. 
    \item there exists $1\leq k\leq m-1$ such that for some $t_0\in U$, $A_{t_0}$ preserves a dominated splitting $E_{t_0}\oplus_{>} F_{t_0}$ with $\dim E_\theta=k$,
\end{itemize}
then for any $\theta$ close to $\theta_0$, $A_\theta$ preserves a dominated splitting  $E_t(\omega)\oplus_>F_t(\omega)$ with $\dim E_t=k$ and $E_t(\omega), F_t(\omega)$ continuously depend on $(t,\omega)$ and $C^l$ depend on $t$.

Generally, let $\mathcal V=\{V_\omega\}_{\omega\in \Omega}$ be a continuous complex vector bundle over a compact metric space $\Omega$. Let $T:\Omega\to \Omega$ be a homeomorphism on $\Omega$.
The cocycle $A_t:\V\to \V$ over $T:\Omega\to \Omega$ be a family continuous linear cocycle. The conclusion presented above is also valid.

\end{lemma}
\begin{proof}

Let $\mathrm{Grass}(k,\mathbb{C}^m)$ denote the Grassmannian of $k$-dimensional subspaces of $\mathbb{C}^m$. The dominated splitting (replace $A$ by $A^n$ if necessary) implies the existence of a continuous family of \textbf{small} cones $\mathcal{C}_{\theta_0}(\omega)$ around $E_{\theta_0}(\omega)$ satisfying:
\begin{equation}\label{eq:cone_condition}
A_{t_0}(\omega)\mathcal{C}_{t_0}(\omega) \subset \mathcal{C}_{t_0}(T\omega) \quad \text{with uniform contraction on } \mathcal{C}_{t_0}(\omega)\subset \mathrm{Grass}(k,\mathbb{C}^m).
\end{equation}

Consider the fiberwise graph transform:
    \[
    \Gamma_t: C^0(\Omega, \mathcal{G}_k(\mathbb{C}^m)) \to C^0(\Omega, \mathcal{G}_k(\mathbb{C}^m)), \quad \Gamma_\theta(V)(\omega) = A_t(T^{-1}\omega)V(T^{-1}\omega).
    \]
    By \eqref{eq:cone_condition},  for $t$ sufficiently close to $t_0$, $\Gamma_t$ is also a  uniform contraction in the neighborhood $\mathcal C_{t_0}$ of the section $E_{t_0}$, yielding a unique fixed point $E_{t}$ near $E_{t_0}$. 
    In particular for $t$ close to $t_0$, $\omega'$ close to $\omega$, $E_{t}(\omega')\in \mathcal C_{t_0}(\omega')$ (using the contraction property) is close to $E_{t_0}(\omega)$ (using the continuity of the cone field $\mathcal C_{t_0}$), hence we get continuity of $E_t(\omega)$ in $(t,\omega)$. Consider the inverse cocycle $A^{-1}$ we could get the corresponding conclusion for $F_t$.

 By Proposition $\ref{lem: Hol dep AJS}$, $E_t(\omega)$ is holomorphically dependent on $A_t(\omega)$, by assumption, $A_t(\omega)$ is $C^l$  on $t$. Therefore, $E_t(\omega)$ is $C^l$  on $t$.

For the general cocycle on a bundle, the first two proofs are similar. The last proof is the same as Section 6 of \cite{AJS},  if $A$ has dominated splitting, for any $A'$ in the neighbourhood of $A$,  use the fact that $E(A')$ is the limit of holomorphic map $(A')^n(T^{-n}(\omega))(E(A))$, then the holomorphicity basically follows Montel theorem. By assumption, $A_t(\omega)$ is $C^l$  on $t$. Therefore, $E_t(\omega)$ is $C^l$  on $t$.
\end{proof}
\section{Square root of analytic function}

\begin{lemma}\label{sqrt}
	Suppose $f(E,\theta)\in C^\omega(B_\delta(E_0)\times\T_\delta,\C)$, and $\min_{B_\delta(E_0)\times\T_\delta}|f(E,\theta)|>0$, then there exists $g\in C^\omega(B_\delta(E_0)\times2\T_\delta)$ such that $g^2=f.$ 
\end{lemma}
\begin{proof}
Since $f(E,\theta)\in C^\omega(B_\delta(E_0)\times\T_\delta)$, we have $f(E,\theta)=\sum_{j}c_j(E)e^{2\pi i jx}.$ Let $$f_n(E,\theta):=\sum_{j=-n}^nc_j(E)e^{2\pi i jx}=e^{-2\pi i n\theta}\sum_{j=0}^{2n}c_{j-n}(E)e^{2\pi i j\theta}.$$
Let $\tilde{f}_n(E,\theta)=\sum_{j=0}^{2n}c_{j-n}(E)e^{2\pi i j\theta}.$ It is a analytic function   in $B_\delta(E_0)\times\{|\Im \theta|<\delta\}.$ Since the regime is simply connect and $\tilde{f}_n(E,\theta)$ bounded from $0$, there exists analytic function $q_n(E,\theta)$, such that $e^{q_n(E,\theta)}=\tilde{f}_n(E,\theta)$. In fact, $$q_n(E,\theta)=\int_{\gamma}\frac{d\tilde{f}_n}{f_n}-\log \tilde{f}_n(0,0),$$
where $\gamma$ is the a differentiable path joining $(0,0)$ to $(E,\theta)$. Since $B_\delta(E_0)\times\{|\Im \theta|<\delta\}$ is simply connect, the integral is independent of path.   In particular, fix $E$, $q_n^\prime(\theta) = \dfrac{\tilde{f}_n^\prime(\theta)}{\tilde{f}_n(\theta)}$. By Lemma 2.3 in \cite{JM}, $e^{\frac{1}{2}q_n(E,\theta)}\in C^\omega(2\T_\delta,\C)$. Fix $\theta$, obviously, $e^{\frac{1}{2}q_n(E,\theta)}\in C^\omega(B_\delta(E_0),\C)$. By Hartogs's theorem, $e^{\frac{1}{2}q_n(E,\theta)}\in C^\omega(B_\delta(E_0)\times 2\T_\delta,\C)$.

We now have 
\begin{eqnarray*}
	g_n:=\mathrm{e}^{\frac{1}{2} q_n-\pi in\theta} \in \mathcal{C}_\delta^\omega(B_\delta(E_0)\times2\T_\delta, \mathbb{C}),
	q_n(0,0) = \log \tilde{f}_n(0,0) \to \log f(0,0)
\end{eqnarray*}
defined (eventually) with respect to a common branch of the log since $f(0,0) \neq 0$ and $e^{-2\pi i n\theta}\tilde{f}_n \to f$. 

Since $d(q_n-2\pi i n\theta)=d{f}_n$
and by construction ${f}_n \to f$ uniformly with $f$ bounded from zero, this in turn implies uniform convergence of $q_n-2\pi i n\theta$ on $B_\delta(E_0)\times\{|\Im \theta|<\delta\}$.

Finally, letting $q:=\lim_{n \to \infty} (q_n-2\pi i n\theta)$ and defining $g:=\mathrm{e}^{\frac{1}{2} q}$ we obtain the Lemma's claim.
\end{proof}

	\section{Proof of Lemma \ref{change}:}\label{change-proof}

 Let $V=\operatorname{span}_{\C}(v_1,v_2,\cdots,v_{2n})$ be a basis of 
the hermitian-subspace $V \subset \C^{2d}$  Define the \textit{Krein matrix} 
\[
G(V) = i \big( \psi(v_i, v_j) \big)_{1 \leq i,j \leq 2n} \in \mathrm{GL}(2n, \mathbb{C})\cap \mathrm{Her}(2n,\C).
\] By Lemma \ref{parbasis}, 
there exists
$\hat{v}_{1,t}, \cdots, \hat{v}_{d,t}, \hat{v}_{-1,t}, \cdots, \hat{v}_{-d,t} \in C^{\omega}(\I\times\T, \mathbb{C}^{2d})$, such that 
$$E^{u}_t(\theta)=\operatorname{span}_{\C}\left(\hat{v}_{1,t}(\theta), \cdots, \hat{v}_{n,t}(\theta)\right),E^{s}_t(\theta)=\operatorname{span}_{\C}\left(\hat{v}_{-1,t}(\theta), \cdots, \hat{v}_{-n,t}(\theta)\right),$$
$$
E^c_t(\theta)=\operatorname{span}_{\C}(\hat{v}_{n+1,t}(\theta), \cdots, \hat{v}_{d,t}(\theta),\hat{v}_{-(n+1,t)}(\theta), \cdots, \hat{v}_{-d,t}(\theta)).
$$
Since $E^u_t(\theta),E^s_t(\theta)$ are Hemitian-isotropic subspaces \cite{WXZ}, which implies there exists $H_t \in C^\omega(\T,\GL(n,\C))$ such that the Krein matrix $$G(\hat{v}_{1,t}(\theta),\cdots,\hat{v}_{n,t}(\theta),\hat{v}_{-1,t}(\theta),\cdots,\hat{v}_{-n,t}(\theta))=i \begin{pmatrix}
    O & -H_t(\theta) \\
    H_t(\theta)^* & O
\end{pmatrix}.
$$
Hence if we take 
$$\begin{aligned}
    & \quad \begin{pmatrix}
   {v}_{1,t}(\theta),\cdots,{v}_{n,t}(\theta),{v}_{-1,t}(\theta),\cdots,{v}_{-n,t}(\theta)
\end{pmatrix}\\&=\begin{pmatrix}
    \hat{v}_{1,t}(\theta),\cdots,\hat{v}_{n,t}(\theta),\hat{v}_{-1,t}(\theta),\cdots,\hat{v}_{-n,t}(\theta)
\end{pmatrix}\operatorname{diag}(I_n,H_t(\theta)^{-1}),
\end{aligned}$$
then its corresponding Krein matrix satisfy 
$$G( {v}_{1,t}(\theta),\cdots,{v}_{n,t}(\theta),{v}_{-1,t}(\theta),\cdots,{v}_{-n,t}(\theta)) =
\begin{pmatrix}
    O & -iI_n \\
    iI_n & O \\
\end{pmatrix}$$
which implies desired result of $(1)$. 

To prove $(2)$, 
we need the following Lemma:
\begin{lemma}	\label{hetong}
			Let $G_t\in C^\omega(\I\times\mathbb{T},\mathrm{Her}(m,\CC)\cap \mathrm{GL}(m,\CC))$. Then there exists  $N_t\in C^\omega(\I\times\mathbb{T}, \mathrm{GL}(m,\CC))$, $p \in \mathbb{N}_+$ such that  $N_t(\theta)^*
			G_t(\theta)N_t(\theta)=\diag (I_p,-I_{m-p}).$
		\end{lemma}
\begin{proof}
		Since $G_t(\cdot)\in C^\omega(\I\times\mathbb{T},\mathrm{Her}(m,\CC)\cap \mathrm{GL}(m,\CC))$, it has continuous eigenvalues $\lambda_{t,i}(\theta)$ for $1 \leq i \leq m$. Suppose that $\lambda_{t,i}(\theta) > 0$ for $1 \leq i \leq p$ and $\lambda_{t,i}(\theta) < 0$ for $p + 1 \leq i \leq m$ for any $\theta\in\TT$.
			Since $\TT$ is compact, we can let $\Gamma_1$ be the circle that encloses all positive eigenvalues, while $\Gamma_2$ is the circle that encloses all negative eigenvalues. Define
			$$
			P_{t,1}(\theta) = \frac{1}{2\pi i}\int_{\Gamma_1}(zI-G_t(\theta))^{-1} dz, \quad P_{t,2}(\theta) = \frac{1}{2\pi i}\int_{\Gamma_2}(zI-G_t(\theta))^{-1} dz.
			$$
			Then, $P_{t,1}(\theta)$ and $P_{t,2}(\theta)$ are projection operators. 
        
			Define
			$$
			Q_{1,t}(\theta) = \text{Range}(P_{1,t}(\theta)) \quad \text{and} \quad Q_{1,t}(\theta) = \text{Range}(P_{2,t}(\theta)),
			$$
			which correspond to continuous $p$-dimensional and $(m-p)$-dimensional invariant subspaces, respectively.
			By Lemma \ref{parbasis}, there exist $\{q^1_{t,i}(\theta)\}_{i=1}^p$ be a \textbf{analytic global basis} for $Q_{1,t}(\theta)$, and $\{q^2_{i,t}(\theta)\}_{i=1}^{m-p}$ be a  \textbf{analytic global basis} for $Q_{2,t}(\theta)$. The subsequent proof is analogous to that of Lemma \ref{ht}.
		\end{proof}

By Lemma \ref{hetong}, and the signature of the center bundle is zero ($\operatorname{Sign}(E^c_t)=0 $)\cite{WXZ}, we can find $N_t \in C^\omega(\I\times\T, \GL(2(d-n),\C))$ such that
$$N_t(\theta)^*G(\hat{v}_{n+1}(\theta),\cdots,\hat{v}_d(\theta),\hat{v}_{-(n+1)}(\theta),\cdots,\hat{v}_{-d}(\theta))N_t(\theta)=
\begin{pmatrix}
    I_{d-n}& O \\
    O & -I_{d-n} \\
\end{pmatrix}.$$
Let \(M \in \mathrm{GL}(2(d-n), \mathbb{C})\) satisfy
\[
M^* \mathrm{diag}(I_p, -I_{2k-p}) M = \begin{pmatrix}
0 & -iI_{d-n} \\
iI_{d-n} & 0
\end{pmatrix},
\] then $$\begin{aligned}
    & \quad\begin{pmatrix}
    {v}_{n+1}(\theta),\cdots,{v}_d(\theta),{v}_{-(n+1)}(\theta),\cdots,{v}_{-d}(\theta)
\end{pmatrix}\\&=\begin{pmatrix}
    \hat{v}_{n+1}(\theta),\cdots,\hat{v}_d(\theta),\hat{v}_{-(n+1)}(\theta),\cdots,\hat{v}_{-d}(\theta)
\end{pmatrix}N_t(\theta)M,
\end{aligned}$$
is the canonical basis of $E^c_t(\theta).$
        
\qed

\section{Proof of Lemma \ref{lem: complx LE}: }\label{mono-lya}

For a one-parameter monotonic analytic $\SL(2,\R)$-cocycle $E\to C_E$, consider $\mathring C_E:=QC_EQ^{-1}$ with $Q=\frac{1}{1+i}\begin{pmatrix}
   1&-i\\1&i
\end{pmatrix}$, then the complexification $\mathring{C}_{E+i\epsilon}(\theta)$, for $\epsilon>0$ which is small, for any $\theta$, maps Poincar\'e disc $\mathbb D=\{z:|z|<1\}$ into $\DD_{e^{-\kappa \epsilon}}:= \{z:|z|<e^{-\kappa \epsilon}\}$ for some $\kappa$ only depends on the lower bound of 
\begin{equation*}\label{monlow}
    |\det ((\partial_t \mathring{C})\cdot \mathring{C}^{-1})|.
\end{equation*}| By the definition of monotonicity, it is bounded away from $0$ on any compact set. In particular, by Schwarz lemma, for any closed interval $\bar{I'}\subset I$, $\mathring{C}_{E+i\epsilon}(\theta)$ contracts $\DD$ with rate bounded from above $e^{-\kappa \epsilon}$ for some $\kappa=\kappa(\bar{I'})>0$. Then the lemma for $\SL(2,\R)$-case follows the fact that the Lyapunov exponent can then be computed as $-\frac{1}{2}$ logarithm of 
the average rate of projective contraction of at the unstable direction, calculated with respect to Poincar\'e metric on $\DD$. 
\qed

\section*{Acknowledgements} 
 Part of this work was done during D. X. visiting Chern Institute of Mathematics, Nankai University. D.
X. thanks Chern Institute of Mathematics for their hospitality during the visit. We thank Svetlana Jitomirskaya for posing the question if one can develop the full version of the subordinacy theory for long-range operators to D. X. and Q. Z. and providing very useful comments for the first version of our draft. We thank Xianzhe Li, Zhengyi Zhou, Weisheng Wu for useful discussions. Zhenfu Wang is supported by NSFC grant (124B2011) and Nankai Zhide Foundation. Disheng Xu is supported by National Key R\&D Program of China No. 2024YFA1015100, NSFC 12090010 and 12090015. Qi Zhou is supported by National Key R\&D Program of China (2020 YFA0713300) and Nankai Zhide Foundation.

\end{document}